\documentclass[11pt,twoside]{article} % For LaTeX2e

\usepackage{eqnarray}
\usepackage{amsmath}
\usepackage[utf8]{inputenc} % allow utf-8 input
\usepackage[T1]{fontenc}    % use 8-bit T1 fonts
\usepackage{booktabs}       % professional-quality tables
\usepackage{amsfonts}       % blackboard math symbols
\usepackage{nicefrac}       % compact symbols for 1/2, etc.
\usepackage{microtype}      % microtypography

\usepackage{epsf}
\usepackage{epsfig}
\usepackage{fancyhdr}
\usepackage{graphics}
\usepackage{graphicx}
\usepackage{psfrag}
\usepackage{fullpage}
\usepackage{pdfpages}
\usepackage{ragged2e}

\usepackage{url}% for url's in bib
\usepackage[colorlinks,linkcolor=magenta,citecolor=blue, pagebackref=true]{hyperref}
\usepackage{stfloats}
\renewcommand*{\backrefalt}[4]{%
    \ifcase #1 \footnotesize{(Not cited.)}%
    \or        \footnotesize{(Cited on page~#2.)}%
    \else      \footnotesize{(Cited on pages~#2.)}%
    \fi}
\usepackage{color}

\usepackage{amsthm}
\usepackage{amssymb,bbm}
\usepackage{caption}
\usepackage{subcaption}
\usepackage{algorithmic}
\usepackage{algorithm}
\usepackage{textcomp}
\usepackage{siunitx}
\usepackage{wrapfig}
\usepackage{multirow}
\usepackage{multicol}% colors

\usepackage{natbib}

\newtheorem{assumption}{Assumption}
\newtheorem{lemma}{Lemma}

\newtheorem{theorem}{Theorem}
\newtheorem{proposition}{Proposition}
\newtheorem{definition}{Definition}
\usepackage{eqnarray,amsmath}
\usepackage{booktabs}       % professional-quality tables
\usepackage{nicefrac}       % compact symbols for 1/2, etc.
\renewcommand{\arraystretch}{1.35}

\usepackage{epsf}
\usepackage{epsfig}
\usepackage{fancyhdr}
\usepackage{graphics}
\usepackage{graphicx}
\usepackage{epstopdf}
\usepackage{psfrag}
\usepackage{fullpage}
\usepackage{pdfpages}

\usepackage{enumerate}   
\usepackage{multirow}
\usepackage{bm}
\usepackage{verbatim} 

\usepackage{mathtools}

\usepackage{url}% for url's in bib
\usepackage{color}

\usepackage{amsthm}
\usepackage{amsmath}
\usepackage{amssymb,bbm}
\usepackage{caption}
\usepackage{algorithmic}
\usepackage{algorithm}
\usepackage{textcomp}
\usepackage{siunitx}
\usepackage{wrapfig}
\usepackage{algorithmic}
\usepackage{algorithm}
\usepackage{multirow}
\usepackage{multicol}% colors
\usepackage{mathtools}

\def\1{\bm{1}}

\DeclareMathAlphabet{\mathsfit}{\encodingdefault}{\sfdefault}{m}{sl}
\SetMathAlphabet{\mathsfit}{bold}{\encodingdefault}{\sfdefault}{bx}{n}

\newcommand{\Gn}{G_{n}}
\newcommand{\Gsn}{G_{*,n}}

\newcommand{\hGn}{\widehat{G}_{n}}

\newcommand{\lGn}{\overline{G}_{n}}

\newcommand{\Gs}{G_{*}}

\newcommand{\taus}{\tau^{*}}
\newcommand{\etas}{\eta^{*}}
\newcommand{\betas}{\beta^{*}}

\newcommand{\dtwo}{{
{D_2}(\Gn, \Gsn)
}}

\newcommand{\done}{D_1(\Gn,\Gsn)}

\newcommand{\xetas}{ X,\etasn }

\newcommand{\lbg}{ G}
\newcommand{\lbgs}{ \Gs}
\newcommand{\hlbgn}{ \hGn}
\newcommand{\llbgn}{ \lGn}
\newcommand{\plbg}{p_{\lbg}}
\newcommand{\plbgs}{p_{\lbgs}}
\newcommand{\phlbgn}{p_{\hGn}}
\newcommand{\barplbg}{\Bar{p}_{\lbg}}
\newcommand{\barphlbgn}{\Bar{p}_{\hlbgn}}

\newcommand{\linephalf}{\overline{\mathcal{P}}^{1/2}}

\DeclareMathOperator*{\argmax}{arg\,max}

\usepackage{eqnarray,amsmath}

\usepackage{amsthm}
\usepackage{amsmath}
\usepackage{amssymb,bbm}

\allowdisplaybreaks
\newcommand{\bbE}{\mathbb{E}}

\newcommand{\GELU}{\mathop{\mathrm{GELU}} }

\newcommand{\plbgn}{p_{ \Gn}}

\newcommand{\Ione}{
{\text{\uppercase\expandafter{\romannumeral1}}}}
\newcommand{\Itwo}{
{\text{\uppercase\expandafter{\romannumeral2}}}
}
\newcommand{\Ithree}{
{\text{\uppercase\expandafter{\romannumeral3}}}
}
\newcommand{\Ifour}{
{\text{\uppercase\expandafter{\romannumeral4}}}
}
\newcommand{\Ifive}{
{\text{\uppercase\expandafter{\romannumeral5}}
}}

\newcommand{\betan}{\beta_n}
\newcommand{\betasn}{\beta_n^*}

\newcommand{\etan}{\eta_{n}}
\newcommand{\etasn}{\eta^{*}_{n}}

\newcommand{\detasn}{\Delta\eta^{*}_{n}}

\newcommand{\etani}{\eta_{ni}}
\newcommand{\etasni}{\eta^{*}_{ni}}
\newcommand{\detani}{\Delta\eta_{ni}}
\newcommand{\detasni}{\Delta\eta^{*}_{ni}}

\newcommand{\etanj}{\eta_{nj}}
\newcommand{\etasnj}{\eta^{*}_{nj}}

\begin{document}

% \begin{center}

% {\bf{\LARGE{Does Contamination Statistically Affect Pre-trained Model: Insights from Mixture of Experts?}}}
  
% %\vspace*{.2in}
% %{\large{
% %\begin{tabular}{ccc}
% %Huy Nguyen$^{\dagger}$ & Nhat Ho$^{\dagger}$
% %\end{tabular}
% %}}

% \vspace*{.2in}

% \begin{tabular}{c}
% The University of Texas at Austin$^{\dagger}$
% \end{tabular}

% \today

% \vspace*{.2in}
% \end{center}

\begin{center}

{\bf{\Large{
{Improving Minimax Estimation Rates 
% \\ \vspace{0.2cm} 
for Contaminated Mixture of \\ \vspace{0.1cm}
Multinomial 
Logistic Experts 
% \\ \vspace{0.2cm}
via Expert Heterogeneity} 
% \\ \vspace{0.2cm}
% Discrete-Contaminated Mixture of Experts
% \\ \vspace{0.2cm}
% Mixture of Experts
}}}

\vspace*{.2in}
{{
% \begin{tabular}{cccccc}
% &Fanqi Yan$^{\star}$ & Dung Le$^{\star}$ &Trang Pham \\ 
% & Huy Nguyen && Nhat Ho 
% \end{tabular}
\begin{tabular}{ccccc}
Fanqi Yan$^{\star}$ & Dung Le$^{\star}$ & Trang Pham & Huy Nguyen & Nhat Ho
%\multicolumn{3}{c}{Huy Nguyen \quad Nhat Ho}
\end{tabular}
}}

\vspace*{.2in}

\begin{tabular}{c}
The University of Texas at Austin
\end{tabular}

\vspace*{.2in}

\today

\vspace*{.2in}

\end{center}

\begin{abstract}
Contaminated mixture of experts (MoE) is motivated by transfer learning methods where a pre-trained model, acting as a frozen expert, is integrated with an adapter model, functioning as a trainable expert, in order to learn a new task. Despite recent efforts to analyze the convergence behavior of parameter estimation in this model, there are still two unresolved problems in the literature. First, the contaminated MoE model has been studied solely in regression settings, while its theoretical foundation in classification settings remains absent. Second, previous works on MoE models for classification capture pointwise convergence rates for parameter estimation without any guaranty of minimax optimality. In this work, we close these gaps by performing, for the first time, the convergence analysis of a contaminated mixture of multinomial logistic experts with homogeneous and heterogeneous structures, respectively. In each regime, we characterize uniform convergence rates for estimating parameters under challenging settings where ground-truth parameters vary with the sample size. Furthermore, we also establish corresponding minimax lower bounds to ensure that these rates are minimax optimal. Notably, our theories offer an important insight into the design of contaminated MoE, that is, expert heterogeneity yields faster parameter estimation rates and, therefore, is more sample-efficient than expert homogeneity.
\end{abstract}
\let\thefootnote\relax\footnotetext{$^\star$Co-first authors}
% , $^\dagger$Co-last authors.}

\section{Introduction}
\label{sec:introduction}
Complex datasets often exhibit heterogeneous structures that a single global model struggles to capture. To address this limitation, the mixture-of-experts (MoE) framework \cite{Jacob_Jordan-1991} decomposes the prediction task into several specialized experts, each responsible for modeling a particular aspect of the data. This adaptive combination provides the MoE architecture with far greater expressive power than any single model. A gating function then evaluates the relevance of each expert for a given input and assigns appropriate weights before aggregating their outputs. From a theoretical perspective, MoE models possess universal approximation capabilities and can approximate arbitrary density functions \cite{bacharoglou2010}. This expressive capacity has also been validated empirically across a wide range of applications in natural language processing \cite{shazeer2017topk,fedus2021switch, Du_Glam_MoE,deepseekv3}, computer vision \cite{Riquelme2021scalingvision, lepikhin_gshard_2021}, multimodal learning \cite{han2024fusemoe, yun2024flexmoe}, domain generalization \cite{nguyen2025cosine,li2023sparse}, reinforcement learning \cite{ceron2024rl, chow_mixture_expert_2023} and multilingual tasks \cite{li2024moe, zhao2024sparse, zhou2025moe, cao2025m}. 

\vspace{0.5em}
\noindent
Meanwhile, contaminated MoE models are formulated to capture transfer-learning mechanisms such as low-rank adaptation \cite{hu2022lora}. Whereas all experts are trainable in standard MoE architectures, a contaminated MoE combines a frozen pre-trained expert with a trainable adapter expert responsible for learning downstream tasks. Despite their practical importance, only two prior works have investigated the theoretical foundations of contaminated MoE. First, \cite{yan2025contaminated} studied an input-free gating mixture between a distribution with known parameters and a Gaussian distribution with a learnable mean and variance. They showed that the convergence rates for estimating the mean expert and variance are inversely proportional to the rate at which the mixture weight vanishes to zero; consequently, faster vanishing weights lead to slower parameter-estimation rates. Second, \cite{yan2025on} analyzed a more realistic contaminated MoE model with a softmax gating function. Because softmax weights cannot converge to zero under compact parameter spaces and bounded input domains, the estimation rates for the mean expert and variance are no longer affected by vanishing weights. However, if the adapter expert learns representations overlapping with the pre-trained expert, the estimation rate of the gating parameters deteriorates significantly.
These theoretical insights, however, have been established only for regression settings. The classification counterpart of contaminated MoE models remains largely unexplored.

\vspace{0.5em}
\noindent
\textbf{Contributions.} Therefore, the primary objective of this paper is to establish a novel theoretical foundation for the contaminated mixture of multinomial logistic experts defined in equation~\eqref{eq:contaminated_pretrain_model_general}. More specifically, we characterize the uniform convergence rates of maximum likelihood estimators and provide guarantees of minimax optimality in a challenging scenario where the model parameters are allowed to vary with the sample size. We also examine two regimes of expert structures: the \emph{homogeneous-expert regime}, in which the adapter expert shares the same functional form as the pre-trained expert, and the \emph{heterogeneous-expert regime}, in which the two experts possess distinct structures.

\vspace{0.5em}
\noindent
\emph{1. Homogeneous-expert regime.} In this setting, the adapter model may merge into the pre-trained model, thereby diminishing the model’s ability to learn downstream tasks. In Theorem~\ref{thm:d2_mle_rate}, we show that the convergence rates of the estimators depend on the rate at which the adapter parameters approach the pre-trained parameters. Consequently, the estimation rates for the adapter and gating parameters become slower than the standard parametric order of $\widetilde{\mathcal{O}}(n^{-1/2})$. In addition, Theorem~\ref{thm:d2_minimax} establishes the corresponding minimax lower bounds, confirming that these rates are indeed minimax optimal.

\vspace{0.5em}
\noindent
\emph{2. Heterogeneous-expert regime.} In this regime, the adapter model is guaranteed not to merge with the pre-trained model. As a result, Theorem~\ref{thm:not_equal} shows that the convergence rates for estimating the adapter and gating parameters improve to the parametric order $\widetilde{\mathcal{O}}(n^{-1/2})$. Theorem~\ref{thm:d1_minimax} further provides matching minimax lower bounds, verifying the optimality of these rates.
Taken together, these theoretical results offer a key insight into the design of contaminated MoE models: expert heterogeneity leads to faster parameter-estimation rates and is therefore more sample-efficient than expert homogeneity.
% Lastly, in Section~\ref{sec:experiments}, we carry out several numerical experiments to empirically justify our theoretical results, and then conclude the paper in Section~\ref{sec:conclusion}.
% Rigorous proofs are provided in the Appendices. 
Finally, in Section~\ref{sec:experiments}, we present a series of numerical experiments that empirically validate our theoretical findings before concluding the paper in Section~\ref{sec:conclusion}. All rigorous proofs and additional experimental results are deferred to the appendices.
\renewcommand{\arraystretch}{1.6} % Increase row height
\begin{table*}[!ht]
\caption{Summary of parameter estimation rates in the contaminated mixture of multinomial logistic experts. Notice that the rates are in expectation. For the notation, please refer to equations~\eqref{eq:contaminated_pretrain_model_general} and \eqref{eq:MLE} with $\Delta\eta^*:=\eta^*-\eta_0$.
}
\centering
\begin{tabular}{ |c|c|c|c|c|c|} 
\hline
 \textbf{Regimes}
& $\boldsymbol{|\exp(\widehat{\tau}_n)-\exp(\tau^*)|}$ &$\boldsymbol{\|\widehat{\beta}_n-\beta^*\|}$ 
&$ \boldsymbol{\|\widehat{\eta}_n-\eta^*\|}$ 
\\
\hline
 Heterogeneous-expert (Thm. \ref{thm:not_equal}, \ref{thm:d1_minimax})
 & $\widetilde{\mathcal{O}} (n^{-1/2})$ 
 & \multicolumn{2}{c|}{$\ \ \widetilde{\mathcal{O}} (n^{-1/2})\ \ $}
 \\ 
 \hline
 Homogeneous-expert (Thm. \ref{thm:d2_mle_rate}, \ref{thm:d2_minimax})

& $\widetilde{\mathcal{O}} (n^{-1/2}\cdot\Vert\Delta\eta^*\Vert^{-2})$ & \multicolumn{2}{c|}{$\ \ \widetilde{\mathcal{O}} (n^{-1/2}\cdot\Vert \Delta\eta^*
 \Vert^{-1})\ \ $} 
\\ \hline
\end{tabular}
\label{table:parameter_rates}
\end{table*}

\vspace{0.5em}
\noindent
\textbf{Notation.}
% For any $n \in \mathbb{N}$, let $[n] := \{1,2,\ldots,n\}$. For a vector $u$, we denote its Euclidean norm by $\|u\|$. Given two positive sequences $(a_n)_{n\ge 1}$ and $(b_n)_{n\ge 1}$, we write $a_n = \mathcal{O}(b_n)$ or $a_n \lesssim b_n$ if there exists a constant $C>0$ such that $a_n \le C b_n$ for all $n \in \mathbb{N}$. We further write $a_n = \widetilde{\mathcal{O}}(b_n)$ to indicate that $a_n \lesssim b_n \mathrm{polylog}(b_n)$, where $\mathrm{polylog}(b_n)$ denotes a factor that is polylogarithmic in $b_n$. Finally, for any two densities $p$ and $q$ (with respect to the Lebesgue measure), the squared Hellinger distance is defined as
% $
% d_H^2(p,q) := \frac{1}{2}\int \bigl(\sqrt{p(x)}-\sqrt{q(x)}\bigr)^2 dx,
% $
% and the total variation distance is
% $
% d_V(p,q) := \frac{1}{2}\int |p(x)-q(x)| dx .
% $
For any positive integer $n$, we denote $[n]:=\{1,\ldots,n\}$ to be the set of all positive integer from 1 to $n$. For two non-negative consequences $(a_n)_{n\geq 1}$
and $(b_n)_{n\geq 1}$, we label $a_n = \mathcal{O}(b_n)$ or $a_n \lesssim b_n$ if there exists a constant $C$ such that $a_n \leq Cb_n$ for all $n \in \mathbb{N}$. Similarly, $a_n = \tilde{\mathcal{O}}(b_n)$ stands for $a_n\lesssim b_n \mathrm{polylog}(b_n)$, where $\mathrm{polylog}(b_n)$ represent a polylogarithmic function of $b_n$. For a vector or a real number $u$, we denote $\|u\|_{\cdot}$ be the norm of $u$, whether exact type of this norm is specified in $\cdot$. For example, $\|u\|_1$ denotes the $1$-norm, $\|u\|_{\infty}$ denotes the $\infty$-norm, and $\|u\|_2$, or simply $\|u\|$ denotes the Euclidean norm. Lastly, 
% for two discrete measures $p = \sum_{i=1}^n p_i\delta_{a_i}$ and $q = \sum_{i=1}^n q_i\delta_{a_i}$ with the same finite support $\mathcal{S} = \{a_i\}_{i \in [n]}$, 
for two discrete measures $p$ and $q$ with the same finite support, 
we denote the Hellinger distance by $d_H(p,q) := \Big(\frac{1}{2} \sum_{i=1}^n|\sqrt{p_i}-\sqrt{q_i}|^2\Big)^{1/2}$, and the total variation distance by $d_V(p,q):=\frac{1}{2}\sum_{i=1}^n |p_i-q_i|$.

\section{Preliminaries}
\label{sec:preliminaries}
In this section, we first introduce the problem formulation and review related
work, along with the associated challenges, in
Section~\ref{sec:problem_setup}. We then investigate the fundamental properties
of the contaminated mixture of multinomial logistic experts in
Section~\ref{sec:fundamental}, focusing in particular on model identifiability
and convergence behavior.

\subsection{Problem Setup}
\label{sec:problem_setup}

% In this paper, we assume that the output $Y \in \{1, 2, \ldots, K\}$ is a discrete response variable, where $K \in \mathbb{N}$, while $X \in \mathcal{X}$ is a covariate vector having an effect on $Y$, in which $\mathcal{X}$ is a compact subset of $\mathbb{R}^d$. 
% Next, the data points $(X_1, Y_1), (X_2, Y_2), \ldots, (X_n, Y_n)$ are independently drawn from the standard softmax gating multinomial logistic contaminated mixture of experts,
% which admits the conditional probability function $p_{G_*}(Y = s  |  X)$ defined for any $s \in \{1, 2, \ldots, K\}$ as follows:

We consider a setting in which the response variable
$Y \in \{1,2,\ldots,K\}$ is discrete, with $K \in \mathbb{N}$, and
$X \in \mathcal{X} \subset \mathbb{R}^d$ is a covariate vector.
The observations $(X_1,Y_1),\ldots,(X_n,Y_n)$ are independently and identically
distributed according to a softmax-gated 
contaminated mixture of multinomial logistic experts model.
The model is characterized by the conditional probability function
$p_{G_*}(y=s |x)$, defined for all $s \in \{1,2,\ldots,K\}$ as
\begin{align}
\label{eq:contaminated_pretrain_model_general}
    &p_{ G_*}(y=s|x) 
    \nonumber 
    \\& 
    := \frac{1}{1+\exp((\beta^*)^{\top}x+\tau^*)}\cdot f_0(y=s|x,\eta_0)  
    % \nonumber \\& 
    % \hspace{1 em} 
    + \frac{\exp((\beta^*)^{\top}x+\tau^*)}{1+\exp((\beta^*)^{\top}x+\tau^*)}\cdot f(y=s|x,\eta^*),
    \nonumber
    \\
    & 
    := \frac{1}{1+\exp((\beta^*)^{\top}x+\tau^*)}\cdot \frac{\exp(h_0(x,\eta_{0s}))}{\sum_{i=1}^K \exp(h_0(x,\eta_{0i}))} 
    % \nonumber 
    % \\& \hspace{1 em} 
    + \frac{\exp((\beta^*)^{\top}x+\tau^*)}{1+\exp((\beta^*)^{\top}x+\tau^*)}\cdot 
    \frac{\exp(h(x,\eta_{s}^*))}{\sum_{j=1}^K \exp(h(x,\eta_{j}^*))}.
\end{align}

% \begin{align}
% g_{G^*}(Y = s  |  X) = 
% & \underbrace{
% \frac{e^{\beta_{11}^\top X + \beta_{01}}}
% {e^{\beta_{11}^\top X + \beta_{01}} + e^{\beta_{12}^\top X + \beta_{02}}}
% }_{\text{gating weight for expert 1}} \cdot
% \underbrace{
% \frac{e^{a_{1s} + b_{1s}^\top X}}{\sum_{\ell=1}^K e^{a_{1\ell} + b_{1\ell}^\top X}}
% }_{\text{expert 1 prediction for class } s} 
% % \nonumber
% % \\
% + 
% & \underbrace{
% \frac{e^{\beta_{12}^\top X + \beta_{02}}}
% {e^{\beta_{11}^\top X + \beta_{01}} + e^{\beta_{12}^\top X + \beta_{02}}}
% }_{\text{gating weight for expert 2}} \cdot
% \underbrace{
% \frac{e^{a_{2s} + b_{2s}^\top X}}{\sum_{\ell=1}^K e^{a_{2\ell} + b_{2\ell}^\top X}}
% }_{\text{expert 2 prediction for class } s}
% \end{align}

% Above, the pre-trained model $f_0(y=s|x,\eta_0)$ corresponds to as a fixed and known 
% multinomial logistic regression 
% parametrized by the $\eta_0 = (\eta_{01}, \cdots\eta_{0K})\in \mathbb{R}^{q\times K}$, where $\eta_{0i}\in \mathbb{R}^{q}, i\in[K]$. 
% $f_0(\cdot)$ is the softmax of the pre-trained expert $h_0(\cdot)$.

\vspace{0.5em}
\noindent
\textbf{Assumptions.}
% Throughout the paper, we impose the following assumptions on the data-generating
% process and the model parameters.
For theoretical analysis, we adopt the following modeling, parameter-space, and
sampling assumptions and conventions throughout the paper:

% \textit{(i) Expert assumptions.} 
% The pre-trained model $f_0(y=s  |  x,\eta_0)$ is fixed and known,
% and takes the form of a multinomial logistic regression parameterized
% In particular,
% $f_0(\cdot)$ is induced by applying a softmax transformation to the
% pre-trained expert $h_0(\cdot)$.
% On the other hand, the adapter model $f(y=s  |  x,\eta^*)$ is specified as a
% multinomial logistic model whose parameters
% The induced class probabilities arise from a softmax
% normalization of the adapter expert  $h(\cdot)$, which determines the
% functional form of $f(\cdot)$.

\vspace{0.5em}
\noindent
\textit{(i) Model specification and expert structure.} 
The pre-trained model $f_0(y=s | x,\eta_0)$ is fixed and known, and takes the
form of a multinomial logistic regression which is
induced by applying a softmax transformation to a pre-trained expert function
$h_0(\cdot;\eta_0)$. 
The adapter model $f(y=s | x,\eta^*)$ is also specified as a multinomial
logistic model, whose class probabilities are obtained via a softmax
normalization of an adaptive expert function $h(\cdot;\eta^*)$. Moreover, $f(\cdot)$ is fully determined by the corresponding expert
function $h(\cdot)$.

\vspace{0.5em}
\noindent
\textit{(ii) Parameter space and localization.} 
% $\eta_0 = (\eta_{01},\ldots,\eta_{0K}) \in \mathbb{R}^{q \times K}$, with
% $\eta_{0i} \in \mathbb{R}^q$ for all $i \in [K]$.
% $\eta^* = (\eta_1^*,\ldots,\eta_K^*) \in \mathbb{R}^{q \times K}$ are unknown and
% adapted from data.
% In addition, the gating parameters are given by $\beta^* \in \mathbb{R}^d$ and
% $\tau^* \in \mathbb{R}$. We collect all unknown parameters associated with the
% gating function and the adapter expert into a single vector
% $G_* = (\beta^*, \tau^*, \eta^*)$, which is assumed to belong to a compact parameter
% space $\Xi \subseteq \mathbb{R}^d \times \mathbb{R} \times \mathbb{R}^{q \times K}$.
The parameters in the pre-trained model are given by
$\eta_0 = (\eta_{01},\ldots,\eta_{0K}) \in \mathbb{R}^{q\times K}$, where
$\eta_{0i} \in \mathbb{R}^q$ for all $i\in[K]$.
The adapter parameters
$\eta^* = (\eta_1^*,\ldots,\eta_K^*) \in \mathbb{R}^{q\times K}$ are unknown and
learned from data. In addition, the gating parameters are given by
$\beta^* \in \mathbb{R}^d$ and $\tau^* \in \mathbb{R}$.
We collect all unknown parameters associated with the gating function and the
adapter expert into a single vector
$G_* = (\beta^*, \tau^*, \eta^*)$, which is assumed to belong to a compact
parameter space
$\Xi \subseteq \mathbb{R}^d \times \mathbb{R} \times \mathbb{R}^{q\times K}$.

% \textit{(iii) Other Assumptions.} 
% Note that we allow these parameters to depend on the sample size $n$; however, for notational convenience, we suppress the dependence of $G_*$ on $n$ throughout the paper.
% Finally, for any vector $v = (v_1,\ldots,v_k) \in \mathbb{R}^k$, we define the
% softmax function component-wise as
% $
% \mathrm{Softmax}(v)_i := {\exp(v_i)}/{\sum_{j=1}^{k} \exp(v_j)},i \in [k].
% $
% For theoretical analysis, we assume that the covariate space $\mathcal{X}$ is bounded.
\vspace{0.5em}
\noindent
\textit{(iii) Sampling and notational assumptions.}
We allow the true parameter $G_*=(\beta^*,\tau^*,\eta^*)$ to depend on the sample
size $n$; however, for notational convenience, we suppress this dependence
throughout the paper. And we assume that the covariate
space $\mathcal X$ is bounded.
Finally, for any vector $v=(v_1,\ldots,v_k)\in\mathbb{R}^k$, we define the
softmax function component-wise as
$
\mathrm{Softmax}(v)_i := {\exp(v_i)}/{\sum_{j=1}^{k} \exp(v_j)},i \in [k].
$

\vspace{0.5em}
\noindent
\textbf{Maximum likelihood estimation (MLE).} 
We estimate the unknown parameters $G^*=(\beta^*,\tau^*,\eta^*)$ of the softmax-gated contaminated mixture of multinomial logistic experts model in \eqref{eq:contaminated_pretrain_model_general} via maximum likelihood estimation \citep{Vandegeer-2000}. Given i.i.d.\ observations $\{(X_i,Y_i)\}_{i=1}^n$, where $Y_i\in\{1,\dots,K\}$, we define the MLE as
\begin{align}
\label{eq:MLE}
    \widehat{G}_n
    :=
    (\widehat{\beta}_n,\widehat{\tau}_n,\widehat{\eta}_n)
    \in
    \argmax_{G\in\Xi}
    \sum_{i=1}^n
    \log\bigl(p_G(Y_i  |  X_i)\bigr),
\end{align}
where $p_G(Y_i  |  X_i)$ denotes the model-implied conditional probability evaluated at the realized class label $Y_i$. 
% and that the parameter space $\Xi$ is compact.

\noindent

\vspace{0.5em}
\noindent
\textbf{Related work.} %
MoE models have demonstrated strong capability in regression tasks. \cite{zeevi1998approximation} showed that, under suitable regularity conditions, mixtures of normalized ridge functions with logistic gating can approximate arbitrary functions. \cite{mendes2011convergence} extended this result to polynomial experts and provided a detailed analysis of the trade-off between the number of experts and the complexity within each expert required to attain optimal estimation rates. Mixtures of Gaussian experts with softmax gating were studied in \cite{nguyen2023demystifying}, where the authors established that, with linear adapter functions, maximum likelihood estimation achieves optimal sample complexity. They further showed that, even under overspecification—when the true number of experts is unknown or exceeded—the correct parameters remain recoverable, albeit with potentially slower convergence.
MoE frameworks have also been investigated in classification settings. \cite{chen2022theory} analyzed binary classification with deep-learning-based experts and demonstrated that performance depends critically on the cluster structure, expert nonlinearity, and router design. For general multiclass classification, \cite{nguyen2024general} developed a comprehensive framework for softmax-gated multinomial logistic MoE models and proposed a new class of modified gating functions that address limitations of linear gates, supported by rigorous theoretical guarantees.

% MoE models further provide insight into contaminated models, where a frozen pre-trained model is fine-tuned through mixtures of prompt-based experts. \cite{yan2025contaminated} studied contaminated MoE models in a Gaussian setting; however, the assumptions remained restrictive due to the use of constant gating and linear prompts. A subsequent work \cite{yan2025on} considered more realistic prompt functions and softmax gating, demonstrating that softmax routers mitigate the prompt-vanishing issue, where contributions from either the pre-training or fine-tuning component diminish. Despite these advances in regression, contaminated MoE models for classification remain largely unexplored.
\vspace{0.5em}
\noindent
\textbf{Main Challenges.}
Compared with previous studies, our analysis involves three fundamental challenges.

\vspace{0.5em}
\noindent
\emph{1. Minimax optimal convergence rates.}  We derive minimax lower bounds for all considered estimators that match the corresponding upper bounds up to logarithmic factors, thereby providing strong evidence for the near-optimality of our estimator. In contrast to prior literature \cite{nguyen2024general,chen2022theory,zeevi1998approximation}, which either does not furnish minimax lower bounds or provides them only for restricted settings, our results hold under substantially broader conditions. Moreover, unlike earlier works, we allow the true parameter of $G$ to vary with the sample size, which enables us to establish uniform bounds on the estimation error. This additional flexibility yields results that are more aligned with practical modeling scenarios.

\vspace{0.5em}
\noindent
\emph{2. More general experts.}
We study mixture-of-experts models with general neural-network expert functions, leading to results that are more directly applicable in modern practice. Previous studies \cite{nguyen2024general} primarily restrict attention to linear experts, thereby limiting the scope and applicability of their theoretical guarantees. Our analysis removes this restriction and accommodates a significantly richer class of expert functions.

\vspace{0.5em}
\noindent
\emph{3. Theoretical justification for expert heterogeneity.} The empirical advantage of heterogeneous experts over homogeneous experts has been documented in prior work \cite{wang-etal-2025-hmoe,ersoy2025hdee,chen2025hetero}. However, these studies provide only empirical observations without accompanying theoretical justification. Our work fills this gap by offering the first rigorous theoretical explanation for this phenomenon, thereby establishing a principled foundation for the use of heterogeneous experts in practice.

% \subsection{Fundamental Properties of the Contaminated Mixture of Multinomial Logistic Experts}
\subsection{Fundamental Statistical Properties }
\label{sec:fundamental}
% In this section, we study the identifiability of the standard
% softmax gating contaminated multinomial logistic MoE model and the convergence
% behavior of density estimation under that model.
In this section, we investigate two fundamental theoretical properties of the 
softmax-gated contaminated mixture of multinomial logistic experts model. Specifically, we establish the identifiability of the model parameters and characterize the convergence behavior of conditional density estimation based on maximum likelihood estimation.

% As mentioned above, when the prompt's learned skills overlap with those of the pre-trained model, estimating the prompt parameters become challening due to potential non-identifiability. To capture that issue precisely, we introduce an analytic condition called distinguishability in Definition~\ref{def:distinguishability}.
% \begin{definition}[Distinguishability]
% \label{def:distinguishability}
% We say that \( f_0 \) is \emph{distinguishable from} \( f \) if the following hold:
% for any distinct pairs of parameters \( (\eta_1, \nu_1), (\eta_2,\nu_2) \in \Theta \), if there exist measurable real-valued functions \( x \in \mathcal{X} \mapsto b_0(x) \), \( x \in \mathcal{X} \mapsto b_1(x) \), and \(  x \in \mathcal{X} \mapsto  \{c_\alpha(x)\}_{0\leq|\alpha| \leq 1}  \), where \( \alpha = (\alpha_1, \alpha_2) \in \mathbb{N}^q \times \mathbb{N} \) with \( |\alpha| = |\alpha_1| + \alpha_2 \leq 1 \) such that
% \begin{align*}
%     b_0(x)\cdot f_0(y  |  h_0(x, \eta_0), \nu_0)
%     &+b_1(x)\cdot f(y|h(x,\eta_1),\nu_1)\\
%     &+ \sum_{0\leq|\alpha| \leq 1} c_\alpha(x) \cdot 
%     \frac{\partial^{|\alpha|} f}{\partial \eta^{\alpha_1} \partial \nu^{\alpha_2}}(y  |  h(x, \eta_2), \nu_2) = 0,
% \end{align*}
% for almost every \( (x,y)  \in \mathcal{X} \times \mathcal{Y} \), then it must be the case that
%     \[
%     b_0(x) = b_1(x) = 0, \quad c_\alpha(x) = 0 \quad \text{ for all } 0\leq|\alpha| \leq 1, \quad \text{for almost every } x.
%     \]
% \end{definition}

\vspace{0.5em}
\noindent
The first result concerns the identifiability of the proposed contaminated mixture of multinomial logistic experts model.

\begin{proposition}[Identifiability]
\label{prop:identifiability}
    Let $G, G'$ be two components in $\Xi$,
    % . Suppose that $f$ is distinguishable from $f_0$
    , 
    then if 
    $p_{G}(y|x) =p_{G^\prime}(y|x) $
    holds for almost all $(x,y)\in \mathcal{X}\times\mathcal{Y}$, then we obtain $G = G^\prime $.
\end{proposition}

\noindent
The proof of Proposition \ref{prop:identifiability} is in Appendix \ref{appendix:identifiability_proof}. Given that our MLE is consistent—meaning that $\widehat{G}_n$ converges to the true underlying model $G_*$—this result shows that convergence to the ground-truth model further implies convergence to the ground-truth parameter values. We then proceed to describe the convergence properties of our estimated model with respect to the Hellinger distance.

\begin{proposition}[Density Estimation Rate]
\label{prop:density-rate}
With the MLE $\widehat{G}_n$ defined in                      equation~\eqref{eq:MLE}, the convergence rate of the
density estimate $p_{\widehat{G}_n}$ to the true density $p_{\Gs}$ is given by
\begin{align}
  \sup_{G^* \in \Xi}\mathbb{E}_{p_{G_{*},n}}\left[\mathbb{E}_X \bigl[
    d_H \bigl(p_{\widehat{G}_n}( \cdot  |  X), p_{G_*}( \cdot  |  X)\bigl)\bigr]\right] 
% \nonumber
% \\
\lesssim 
(\log(n)/n)^{1/2}.
\end{align}
\end{proposition}
\noindent
The above result shows that our density $p_{\widehat{G}_n}$ estimator converges, in terms of the Hellinger distance, to the true density $p_{G_*}$ at a near-parametric rate of order $\tilde{\mathcal{O}}(n^{-1/2})$. See Appendix \ref{appendix:ConvergenceRateofDensityEstimation} for the proof. 

\vspace{0.5em}
\noindent
These two properties constitute the key components of our analysis. In particular, Proposition \ref{prop:density-rate} shows that the study of the parametric convergence rate can be carried out by exploiting an equality that relates the density discrepancy to the parameter distance. The proof of this relation is based on a contradiction argument, in which we consider a sequence of $\widehat{G}_n$ densities converging to the ground-truth density. 
Recall that Proposition \ref{prop:identifiability} then ensures the convergence of the density necessarily implies convergence of the corresponding parameters to their true values.

\section{Convergence Analysis of Parameter Estimation}
\label{sec:theory}
% In this section, we present various convergence rates for the MLE estimator of the model prompt and gating parameters. 
% In Sections \ref{sec:neq} and \ref{sec:eq} we provide separate minimax analyses, depending on
% whether the structure of prompt expert $h$ is different from the frozen expert $h_0$ or not, respectively.
% In this section, we establish convergence rates for the maximum likelihood estimator (MLE) of the model’s prompt and gating parameters.

% Sections \ref{sec:neq} and \ref{sec:eq} present separate minimax analyses, corresponding to the cases where the prompt expert 
% $h$ differs from the frozen expert 
% $h_0$ and where the two coincide, respectively.

% In this section, we establish convergence rates for the maximum likelihood
% estimator (MLE) of the gating and adapter parameters.
% Sections~\ref{sec:eq} and \ref{sec:neq} develop separate minimax analyses for the \emph{homogeneous-expert} regime, in which $h$ and $h_0$ belong to the same
% function family, respectively, and the \emph{heterogeneous-expert} regime, in which the adapter expert $h$ differs
% structurally from the pretrained expert $h_0$, respectively.
In this section, we establish convergence rates for the maximum likelihood
estimator (MLE) of the gating and adapter parameters. Sections~\ref{sec:eq} and~\ref{sec:neq} present separate minimax analyses for the
\emph{homogeneous-expert} regime, in which $h$ and $h_0$ belong to the same
function family, and the \emph{heterogeneous-expert} regime, in which the
adapter expert $h$ differs structurally from the pretrained expert $h_0$,
respectively.

\subsection{Homogeneous-expert regime}
\label{sec:eq}
% We now turn to the homogeneous-expertsregime, in which disentanglement becomes
% even more challenging, as the pretrained and adapter experts belong to the
% same function family.
% If the pre-trained and prompt model use the same expert function, i.e. $h_0=h$.
% Under this setting, the prompt model may converge to the pre-trained model.
% We will thus focus on this challenging scenario. 
% We begin by considering the homogeneous-expertsregime, in which the pretrained expert $h_0$ differs structurally from the adapter expert $h$.
% We begin by studying the convergence properties of the MLE in the
% homogeneous-expertsregime, in which the adaptive and pretrained experts share
% the same functional form.

We begin by studying the convergence properties of the MLE in the
homogeneous-expert regime, in which the adapter and pretrained experts share
the same functional form. This setting corresponds to the standard mixture-of-experts
framework, where all experts belong to the same model class but differ only in
their parameter values.

% Building on the density estimation rate established in Proposition~\ref{prop:density-rate}, our goal is to construct a loss function between the MLE $\widehat{G}_n$ and the ground-truth parameters $G_*$ that can be upper bounded by the Hellinger distance between the corresponding model densities, thereby enabling the derivation of parameter estimation rates.

% In the homogeneous-expertsregime, disentangling the contributions of the pre-trained and adapter experts is substantially is challenging, since the pre-trained and adaptive models belong to the same function family, that is, $h_0 = h$, and thus share the same expert function. As a consequence, the adapter expert may converge to the pre-trained expert, and the resulting model can effectively collapse to the pre-trained one. This setting represents a particularly challenging scenario for disentanglement, and we therefore focus our analysis on this regime.

\vspace{0.5em}
\noindent
In the homogeneous-expert regime, disentangling the respective contributions of
the pretrained and adapter experts becomes substantially challenging.
Indeed, since both experts belong to the same function family, 
% that is,
% $h_0 = h$, 
the adapter expert can converge toward the pretrained one, causing
the resulting mixture model to effectively collapse to the pretrained expert.
This potential degeneracy makes the homogeneous-expert regime particularly
demanding from an inferential standpoint, and it is therefore the primary
focus of our subsequent analysis.

% In the homogeneous-expertsregime, the pre-trained and adapter experts share the same functional form, that is, $h_0 = h$, and thus belong to the same softmax model family.
% In particular, under this setting, it is possible for the adapter expert parameters $\eta^*$ to converge to the pre-trained parameters $\eta_0$ as the sample size $n$ increases. When this occurs, the corresponding expert distributions will become $f(\cdot | x, \eta^*)  \longrightarrow  f_0(\cdot | x, \eta_0)$, 
% % uniformly in $x\in\mathcal{X}$, 
% and the two frozen and adaptive terms in equation \eqref{eq:contaminated_pretrain_model_general} become asymptotically indistinguishable. As a consequence, the discrete contaminated MoE model may effectively collapse to the pre-trained model, regardless of the value of the gating function.

\vspace{0.5em}
\noindent
In this regime, the adapter expert parameters $\eta^*$ may converge to the
pretrained parameters $\eta_0$ as the sample size $n$ grows. In such a case,
the associated expert distributions satisfy
$f(\cdot | x,\eta^*) \to f_0(\cdot | x,\eta_0)$, and the pretrained and
adapter components in \eqref{eq:contaminated_pretrain_model_general} become
asymptotically indistinguishable. Consequently, the softmax-gated contaminated mixture of multinomial logistic experts model may effectively collapse to the pretrained model,
irrespective of the value of the gating function.
This collapse has important implications for statistical estimation. Since the gating parameters $(\beta,\tau)$ influence the conditional distribution only through the contrast between the two components, when $\eta^*$ is close to $\eta_0$, variations in the gating function induce only negligible changes in the overall density. 
% Consequently, the likelihood surface becomes nearly flat in the gating directions, rendering the gating parameters weakly identifiable. 
% In this regime, one therefore expects the estimation of $\tau,\beta$ and $\eta$ to be substantially more difficult and to exhibit slower convergence rates.
 This suggests that parameter estimation in the
homogeneous-expert regime is inherently  delicate.

\vspace{0.5em}
\noindent
Building on the density estimation rate established in
Proposition~\ref{prop:density-rate}, our objective is to translate
distributional convergence into quantitative guarantees for parameter
estimation in the homogeneous-expert regime. To this end, we decompose the
density discrepancy $p_G - p_{G_*}$ into linearly independent components with
respect to the parameters via a Taylor expansion. This approach requires a
careful analysis of the conditional density
$g(y=s  |  x; \beta, \eta)
:= \exp(\beta^\top x) 
{\exp(h(x,\eta_s))}/{\sum_{j=1}^K \exp(h(x,\eta_j))}$.
Due to the homogeneity, the mapping from parameters to the induced
density cannot, in general, be resolved using only first-order information.
Consequently, the parameter--density relationship must be characterized
through higher-order derivatives of the expert function $h$ with respect to
$\eta$.

\vspace{0.5em}
\noindent
To rigorously analyze this setting, it is therefore necessary to impose additional structural assumptions on the expert function $h$. In particular, we introduce a strong identifiability condition on the expert structure, which ensures sufficient linear independence among the relevant derivative terms and enables a refined analysis of the estimation behavior when homogeneity holds.

\begin{definition}[Strong Identifiability]
\label{appendix_def:distinguish_linear_independent}  

The expert function $x \mapsto h(x,\eta)$ is said to be \emph{strongly identifiable} if it is twice differentiable with respect to the parameter vector $\eta=(\eta_1,\eta_2,\ldots,\eta_K)\in\mathbb{R}^{q\times K}$ for almost every $x\in\mathcal{X}$, and if, for any fixed $\beta\in\mathbb{R}^d$ and any collection of parameters $\eta_i\in\mathbb{R}^q$ with $i\in\{1,2,\ldots,K\}$, each of the sets of real-valued functions of $x$ specified below is linearly independent over $\mathbb{R}$. For notational simplicity, we write $h(\cdot)$ in place of $h(\cdot,\eta)$ in the sequel.

\begin{enumerate}
    \item 
    \textbf{Second-order logit-derivative pullback set:}
    \begin{align*}
        \left\{
        \frac{\partial  h}{\partial \eta_i^{(u)} }
        \frac{\partial  h}{\partial \eta_j^{(v)} },\,
        \exp(\beta^{\top}x)
        \frac{\partial  h}{\partial \eta_i^{(u)} }
        \frac{\partial  h}{\partial \eta_j^{(v)} }
        \right\}^{u,v \in [q]}_{i,j \in [K]}.
    \end{align*}
\item \textbf{First-order logit-derivative pullback set:}
    \begin{align*}
        \Bigg\{
        &\frac{\partial  h}{\partial \eta_i^{(u)}},\,
        x^{(w)}\frac{\partial h}{\partial \eta_i^{(u)}},\,
        \exp(\beta^{\top}x)\frac{\partial  h}{\partial \eta_i^{(u)}},\,
        % \\&\hspace{0.65em}
        \frac{\partial^2  h}{\partial \eta_i^{(u)}\partial \eta_i^{(v)}},\,
    \exp(\beta^{\top}x)\frac{\partial^2  h}{\partial \eta_i^{(u)}\partial \eta_i^{(v)}}
        \Bigg\}^{u,v \in [q], w \in [d]}_{i\in [K]}.
    \end{align*}
\end{enumerate}
\end{definition}

\noindent
% \paragraph{Intuition.}
The strong identifiability condition ensures that distinct perturbations of the expert parameters induce distinguishable changes in the conditional distribution, even in the homogeneous-expert regime where the pre-trained and adapter experts share the same functional form. The second-order logit-derivative pullback set captures quadratic interaction directions arising from the second-order derivatives of the softmax density with respect to the logits, which become essential when first-order effects alone are insufficient to separate nearby parameter values. The first-order logit-derivative pullback set, on the other hand, collects all directions induced by the first-order sensitivity of the density to the logits, together with modulation by the covariates and the gating mechanism. Requiring linear independence of both sets guarantees that neither linear nor quadratic effects can cancel at the density level, thereby restoring identifiability and enabling meaningful control of parameter estimation errors when homogeneity holds.

% \noindent
% \textbf{Examples.} The expert functions $h(x, \eta) = \mathrm{GELU}(\eta^\top x)$, $h(x, \eta) = \mathrm{sigmoid}(\eta^\top x)$, and $h(x, \eta) = \tanh(\eta^\top x)$ satisfy the strong identifiability condition, as their nonlinearities avoid degeneracies. 
% In contrast, $h(x, \eta) = \mathrm{ReLU}(\eta^\top x)$ fails the second-order independence condition, as the second-order derivatives vanish almost everywhere.
% Another failure case arises when $h(x, \eta) = \sigma(a^\top x + b)$, where $\eta = (a, b)$ and $\sigma$ is any scalar activation function.
% This leads to $\partial h / \partial a = x \cdot \partial h / \partial b$, directly violating Condition~2.
% \\
\vspace{0.5em}
\noindent
\textbf{Examples.}
Under the general assumptions adopted in this paper, the following expert
functions satisfy the strong identifiability condition in
Definition~\ref{appendix_def:distinguish_linear_independent}:
$h(x,\eta)=\mathrm{GELU}(\eta^\top x)$,
$h(x,\eta)=\mathrm{sigmoid}(\eta^\top x)$, and
$h(x,\eta)=\tanh(\eta^\top x)$.
These activation functions are smooth and nonlinear, and their first- and
second-order derivatives do not vanish on sets of positive measure. As a
consequence, the associated logit-derivative pullback sets are linearly
independent almost everywhere. Notably, GELU is the standard activation used in
Transformer feedforward networks
\citep{hendrycks2016gaussian,devlin2019bert,brown2020language}.

\vspace{0.5em}
\noindent
In contrast, the expert function $h(x,\eta)=\mathrm{ReLU}(\eta^\top x)$ fails to
satisfy strong identifiability. This failure arises because ReLU is not twice
differentiable and its second-order derivatives vanish almost everywhere, which
causes the second-order pullback set to degenerate. Another failure mode occurs
for expert functions of the form $h(x,\eta)=\sigma(a^\top x+b)$, where
$\eta=(a,b)$ and $\sigma$ is any scalar activation function. In this case,
$\partial h/\partial a = x  \partial h/\partial b$, which induces an exact
linear dependence in the first-order pullback set and violates condition~2 of
Definition~\ref{appendix_def:distinguish_linear_independent}.

\noindent
Overall, strong identifiability favors smooth, nonlinear expert
parameterizations in which distinct parameter components induce genuinely
independent functional variations with respect to the covariates, as is typical
in modern large language models.

\vspace{0.5em}
\noindent
To analyze the convergence rates of the MLE under the strong identifiability assumption, we define the following discrepancy measure between $G$ and $G_*$ for the homogeneous expert regime:
\begin{align}
     &D_1(G,G_*) :
     =
     \nonumber\\&
    \exp(\tau)
    \sum_{i=1}^{K}  \|\Delta\eta_{i} \|^2 
    +
    \exp(\tau^*)
    \sum_{i=1}^{K}  \|\Delta\eta^*_{i} \|^2 
    % \nonumber\\&
    -\min\{\exp(\tau), \exp(\tau^*)\}
    \Big(
    \sum_{i=1}^{K} \|\Delta\eta_{i} \|^2 
    +
    \sum_{i=1}^{K} \|\Delta\eta^*_{i} \|^2 
    \Big)
    \nonumber\\&
    +
    \sum_{i=1}^{K}
    \Big[ 
    (\exp(\tau)
     \|\Delta\eta_{i} \| 
    +\exp(\tau^*)
      \|\Delta\eta^*_{i} \| )
    % \nonumber\\&\hspace{3cm} 
    \times
    \big( \|\beta - \beta^*\| +   \|\eta_{i} - \eta_{i}^*\|  \big)\Big],
\end{align}
where we denote $\Delta\eta_i=\eta_i-\eta_{0i}$ and $\Delta\eta_i^*=\eta_i^*-\eta_{0i}$.

% Now we could get the the convergence behavior of the MLE in the homogeneous expert regime via the following theorem.
\vspace{0.5em}
\noindent
The convergence behavior of the MLE in the homogeneous-expert regime is
established by the following theorem.
% \begin{theorem}[Homogeneous Expert Regime]
% Suppose that the pre-trained expert and the prompt expert belong to the same function family.
% That is, there exists a function class $\mathcal H$ such that
% $h_0(\cdot;\eta_0)\in\mathcal H$ and $h(\cdot;\eta)\in\mathcal H$, and in particular $h_0 = h$.
% We refer to this setting as the \emph{homogeneous expert regime}.
% Then, there exists a positive constant $C_2$, depending on $\Xi,\eta_0,\nu_0$, such that the Hellinger lower bound
% $\mathbb{E}_X \left[ d_H \left(p_G(\cdot\mid X),  p_{G_*}(\cdot\mid X)\right) \right]
%  \ge 
% C_2   D_2(G,G_*)$
% holds for all parameters $G\in\Xi$.
% As a result, we obtain
% \end{theorem}
\begin{theorem}
\label{thm:d2_mle_rate}
%Suppose that the pre-trained expert and the adapter expert belong to the same function family.
%That is, there exists a function class $\mathcal H$ such that
%$h_0(\cdot;\eta_0)\in\mathcal H$ and $h(\cdot;\eta)\in\mathcal H$.
%We refer to this setting as the \emph{homogeneous expert regime}.
Under the homogeneous-experts regime, suppose that the expert function $h$ satisfies the strong identifiability condition,
then there exists a positive constant $C_1$, depending on $\Xi$ and $\eta_0$, such that the Hellinger lower bound
$\mathbb{E}_X   \left[
d_H  \left(p_G(\cdot\mid X),  p_{G_*}(\cdot\mid X)\right)
\right]
 \ge 
C_1  D_1(G,G_*)$
holds for all parameters $G=(\beta,\tau,\eta)\in\Xi(l_n)$.
As a result, we obtain
\begin{align}
    \label{eq:nondis_bound_1}
    &
    \sup_{G_*\in\Xi(l_n)}
    % \hspace{-1em} 
    \mathbb{E}_{p_{\Gs,n}} 
    \Big[
        \Big(
    \sum_{i=1}^{K}
    \Vert \Delta\eta_i^*\Vert^2
    \Big)
    \times
    |\exp(\widehat{\tau}_n)
    -\exp(\tau^*)|
    \Big]
    % \nonumber\\&\hspace{4.75cm}
    \lesssim (\log(n)/n)^{1/2},\\
    \label{eq:nondis_bound_2}
    &
    \sup_{G_*\in\Xi(l_n)}
    % \hspace{-1em}
    \mathbb{E}_{p_{\Gs,n}} 
    \Big[
    \sum_{i=1}^{K}
    \Vert \Delta\eta_i^*\Vert
    \Big( \|\widehat{\beta}_n - \beta^*\| +   \|\widehat{\eta}_{n,i} - \eta_{i}^*\|  \Big)
    %     \nonumber
    % \\&\hspace{3.
    % cm}
    \times
    \exp(\tau^*)
    \Big]
    \lesssim 
    (\log(n)/n)^{1/2},
\end{align}
% for any sequence $(l_n)_{n\geq 1}$ such that $l_n/\log n\to\infty$ as $n\to\infty$, where we denote for all $i\in\{1,\ldots,K\}$ that
% \begin{align*}
%     \Xi(l_n) := 
% \Bigg\{ 
% G \in \Xi : 
% \frac{l_n 
% }{
%     \min\limits_{\substack{1 \le u \le q \\ 1 \le v \le d}}
%     \left\{ |\eta_{i,u}|^2, |\beta_{v}|^2 \right\} 
%     \sqrt{n}
% }
% \leq \exp(\tau)
% \Bigg\}.
% \end{align*}
for any sequence $(l_n)_{n\geq 1}$ such that $l_n/\log n \to \infty$ as $n\to\infty$, 
where the parameter space $\Xi(l_n)\subset\Xi$ is defined, for all $i\in\{1,\ldots,K\}$, as
\begin{align*}
    \Xi(l_n) := 
\Bigg\{ 
G \in \Xi : 
\frac{l_n 
}{
    \min\limits_{\substack{1 \le u \le q ,\, 1 \le v \le d}}
    \left\{ |\eta_{i,u}|^2, |\beta_{v}|^2 \right\} 
    \sqrt{n}
}
\leq \exp(\tau)
\Bigg\}.
\end{align*}

\end{theorem}
\noindent
The proof of Theorem~\ref{thm:d2_mle_rate} is in Appendix~\ref{app_proof: d2_loss}. 
% The condition imposing a lower bound on $\exp(\tau)$ in the definition of $\Xi(\ell_n)$ is necessary to guarantee that $(\widehat{\beta}_n,\widehat{\tau}_n,\widehat{\eta}_n)$ are consistent estimators of $(\beta^*,\tau^*,\eta^*)$, respectively.
% The lower-bound condition on $\exp(\tau)$ in $\Xi(\ell_n)$ is necessary to ensure consistency of the estimators, which underpins both identifiability and the validity of the asymptotic error bounds derived in this work.

\vspace{0.5em}
\noindent
The lower-bound condition on $\exp(\tau)$ in the definition of $\Xi(\ell_n)$ is necessary to guarantee that $(\widehat{\beta}_n,\widehat{\tau}_n,\widehat{\eta}_n)$ are consistent estimators of $(\beta^*,\tau^*,\eta^*)$, which is essential for ensuring identifiability and for establishing the asymptotic convergence rates of the MLE.

\vspace{0.5em}
\noindent
Theorem~\ref{thm:d2_mle_rate} implies that the convergence rates of
$(\widehat{\beta}_n,\widehat{\tau}_n,\widehat{\eta}_n)$ depend on the adapter
deviations $\{\Delta \eta_i^*\}_{i=1}^K$, yielding two main implications:

\vspace{0.5em}
\noindent
{\it (i)} At first, the bound in equation \eqref{eq:nondis_bound_1} indicates that the convergence rate of
$\exp(\widehat{\tau}_n)$ toward its ground-truth counterpart $\exp(\tau^*)$ is
slower than the near-parametric rate $\widetilde{\mathcal O}(n^{-1/2})$, as it
explicitly depends on the rate at which the adapter parameters
$\Delta \eta_i^*$ vanish.
For instance, if the dominant adapter components
$\eta_i^*$ approach their pretrained counterparts $\eta_{0i}$ at the rate
$\widetilde{\mathcal O}(n^{-1/8})$, then \eqref{eq:nondis_bound_1} implies that
$\exp(\widehat{\tau}_n)$ converges to $\exp(\tau^*)$ at the slower rate
$\widetilde{\mathcal O}(n^{-1/4})$. This phenomenon is corroborated by the
numerical experiments reported in the subsequent section.

\vspace{0.5em}
\noindent
{\it (ii)} Similarly, the convergence behavior of the estimators
$(\widehat{\beta}_n, \widehat{\eta}_n)$ is also governed by the decay rate of
the adapter parameters and is therefore slower than
$\widetilde{\mathcal O}(n^{-1/2})$. In particular, if the dominant
$\Delta \eta_i^*$ converges to zero at rate
$\widetilde{\mathcal O}(n^{-1/8})$, then the bound in
\eqref{eq:nondis_bound_2} shows that both $\widehat{\beta}_n$ and
$\widehat{\eta}_n$ converge to $\beta^*$ and $\eta^*$, respectively, at the
rate $\widetilde{\mathcal O}(n^{-3/8})$. This slower convergence is likewise
validated empirically in our numerical studies.

% In our final result, whose proof can be found in Appendix~\ref{apppf:d2_minimax}, we show that the slower  converge rates for the MLE under non-distinguishability from Theorem~\ref{eq:nondis_bound_1} are in fact essentially minimax optimal.

\vspace{0.5em}
\noindent
The near-parametric convergence rates established in
Theorem~\ref{thm:d2_mle_rate} naturally raise the question of their statistical
optimality. We answer this question by deriving matching minimax lower bounds,
as presented below.
\begin{theorem}
\label{thm:d2_minimax}
    % Suppose that $h_0=h$. Then, the minimax lower bounds
% \begin{align*}
%     \Xi(l_n) := \left\{ G = (\tau,\beta,\eta, \nu) \in \Xi : 
% \frac{
%     l_n
% }{
%     \min\limits_{1 \leq i \leq q,1 \leq j \leq d,} 
%     \left\{ |\eta^{(i)}|^2, (\nu)|^2, |\beta^{(j)}|^2 \right\} 
%     \sqrt{n}
% } \leq \exp(\tau)
% \right\}.
% \end{align*}

% \begin{align*}
%     &\inf_{\overline{G}_n }
%     \hspace{-0.2em}
%     \sup_{G\in \Xi(l_n)  }
%     \hspace{-0.5em}
%     \mathbb{E}_{p_{G,n}} \Big[ 
%     \sum_{i=1}^{K-1}\Vert \Delta\eta_i \Vert^4    
%     \times \|\exp(\overline{\tau}_n)-\exp(\tau)\|^2
%     \Big] 
%     \\&
%     \hspace{19em}
%     \gtrsim n^{-1/r},
%     \\
%     &\inf_{\overline{G}_n}
%     \hspace{-0.2em}
%     \sup_{G\in \Xi(l_n) }
%     \hspace{-0.5em}
%     \mathbb{E}_{p_{G,n}} 
%     \Big[ 
%     % \exp^2(\tau) 
%     % \\&
%     % \times
%         \sum_{i=1}^{K-1}
%     \Vert \Delta\eta_i\Vert^2
%     \big( \|\overline{\beta}_n - \beta\|^2 
%         \hspace{-0.2em}
% +       \hspace{-0.2em}
% \|\overline{\eta}_{n,i} - \eta_{i}\|^2  \big) 
%     \\&
%     \hspace{14em}
%     \times
%     \exp^2(\tau)\Big] 
%     \gtrsim n^{-1/r},
% \end{align*}
% hold for any sequence $(l_n)_{n\geq 1}$ and any $0<r < 1$, , where the infimum is over all estimators $\overline{G}_n$ taking values in $\Xi$.

Suppose that the assumptions of Theorem~\ref{thm:d2_mle_rate} hold. Then the following minimax lower bounds 
% \begin{align*}
% &    \inf_{\overline{G}_n\in \Xi }\sup_{G\in \Xi }
%     \mathbb{E}_{p_{G,n}} \Big( 
%     |\exp(\overline{\tau}_n)
%     -\exp(\tau)|^2 \Big) 
%     \gtrsim n^{-1/r},
%     \\
% &\inf_{\overline{G}_n\in \Xi }\sup_{G\in \Xi }
%     \mathbb{E}_{p_{G,n}} 
%     \Big( \exp^2(\tau) 
%     \left[ \|\overline{\beta}_n - \beta\|^2 
% + \sum_{i=1}^{K-1} \big( \overline{\eta}_{n,i} - \eta_{i}\|^2 \big) \right] 
%     \Big) 
%     \\
%     \gtrsim n^{-1/r},
% \end{align*}
\begin{align*}
    &\inf_{\overline{G}_n }
    \hspace{-0.2em}
    \sup_{G\in \Xi(l_n)  }
    \hspace{-0.5em}
    \mathbb{E}_{p_{G,n}} \Big[ 
    \sum_{i=1}^{K-1}\Vert \Delta\eta_i \Vert^4    
    \times \|\exp(\overline{\tau}_n)-\exp(\tau)\|^2
    \Big] 
    % \\&
    % \hspace{19em}
    \gtrsim n^{-1/r},
    \\
    &\inf_{\overline{G}_n}
    \hspace{-0.2em}
    \sup_{G\in \Xi(l_n) }
    \hspace{-0.5em}
    \mathbb{E}_{p_{G,n}} 
    \Big[ 
    % \exp^2(\tau) 
    % \\&
    % \times
        \sum_{i=1}^{K-1}
    \Vert \Delta\eta_i\Vert^2
    \big( \|\overline{\beta}_n - \beta\|^2 
+       
\|\overline{\eta}_{n,i} - \eta_{i}\|^2  \big) 
    % \\&
    % \hspace{14em}
    \times
    \exp^2(\tau)\Big] 
    \gtrsim n^{-1/r},
\end{align*}
hold for any $0<r < 1$.
% where the infimum is over all estimators taking values in $\Xi$.
% Here, the infimum is taken over all sequences of estimates $(\widehat{\beta}_n, \widehat{\tau}_n, \widehat{\eta}_n)$.
% Here, the infimum is taken over all sequences of estimators
% $\overline{G}_n = (\overline{\beta}_n, \overline{\tau}_n, \overline{\eta}_n)$
% taking values in the parameter space $\Xi$, and the expectation is with respect to the joint distribution $p_{G,n}$ of the $n$ i.i.d.\ observations generated under parameter $G$.
Here, the infimum is taken over all  estimators
$\overline{G}_n = (\overline{\beta}_n,\overline{\tau}_n,\overline{\eta}_n)$
that take values in the parameter space $\Xi$, and the expectation is with respect to the joint distribution $p_{G,n}$ of the $n$ i.i.d.\ observations generated under parameter $G\in \Xi(l_n) $.
\end{theorem}
\noindent
The proof of Theorem \ref{thm:d2_minimax} can be found in Appendix~\ref{apppf:d2_minimax}.
% See Appendix~\ref{proof:d1_minimax} for the proof of
% Theorem~\ref{thm:d1_minimax}.
Theorem~\ref{thm:d2_minimax} establishes matching minimax lower bounds for the
homogeneous-expert regime, confirming that the near-parametric rates achieved
by the MLE are optimal up to logarithmic factors.

\subsection{Heterogeneous-expert regime}
\label{sec:neq}
% \begin{align*}
% &D_1(G, G_*) 
% :=  \left| \exp(\tau) - \exp(\tau^*) \right| 
% \\&
% + 
% [\exp(\tau)+\exp(\tau^*)] 
% \left[ \|\beta - \beta^*\| 
% + \sum_{i=1}^{K-1} \big( \|\eta_{i} - \eta_{i}^*\| \big) \right].
% \end{align*}

% To start with, we consider a scenario in which the pre-trained expert $h_0$ is distinct with the prompt expert $h$. Recall that given the density estimation rate in Proposition~\ref{prop:density-rate}, we need to construct a loss function between the MLE $\widehat{G}_n$ and the ground-truth parameters $G_*$, which should be bounded by the Hellinger distance between the two corresponding densities, in order to capture the parameter estimation rates. 
% We begin by considering the setting in which the pre-trained expert $h_0$ is distinct from the prompt expert $h$. 

% Given the complex dependency of the homogeneity induced by the merging of the pre-trained model and adaptor. 
% We will naturally consider that could we faster the convergence rate. 
% And we proposed the heterogeneous-experts regime, in which the
% pretrained expert $h_0$ differs structurally from the adapter expert $h$.
The strong coupling between the pretrained expert and the adapter expert in
the homogeneous-expert regime induces substantial non-identifiability and
slows down parameter estimation. This observation naturally motivates the
question of whether faster convergence rates can be obtained by relaxing the
homogeneity assumption. To this end, we introduce the
\emph{heterogeneous-expert} regime, in which the pretrained expert $h_0$ is
structurally distinct from the adapter expert $h$.

\vspace{0.5em}
\noindent
Analogous to the previous analysis, we introduce a loss function
$D_2(G,G_*)$ that measures the discrepancy between $\widehat{G}_n$ and $G_*$ when heterogeneity holds:
\begin{align}
\label{applossdef:D1-loss}
&D_2(G, G_*) 
:=  \left| \exp(\tau) - \exp(\tau^*) \right| 
% \nonumber
% \\&
+ 
\left[\exp(\tau)+\exp(\tau^*)\right] 
% \\&
% \hspace{7.1em}
\times
\Big[ \|\beta - \beta^*\| 
+ \sum_{i=1}^{K-1}  \|\eta_{i} - \eta_{i}^*\|  \Big].
\end{align}

% \begin{align}
% \label{applossdef:D1-loss}
% &D_1(G, G_*) 
% :=  
% \left[\exp(\tau)+\exp(\tau^*)\right] 
% \times
% \left[ \|\beta - \beta^*\| 
% + \sum_{i=1}^{K-1} \big( \|\eta_{i} - \eta_{i}^*\| \big) \right]
% +
% \nonumber
% \\&
% \left| \exp(\tau) - \exp(\tau^*) \right| 
% % \left[\exp(\tau)+\exp(\tau^*)\right] 
% \end{align}

% \begin{align}
% \label{applossdef:D1-loss}
%     D_1(G, G_*)
%     =
%     |
%     \exp(\tau)-\exp(\tau^*)
%     |
%     +
%     \big(
%     \exp(\tau)+\exp(\tau^*)
%     \big)
%     \|(\beta,\eta,\nu)-(\beta^*,\eta^*,\nu^*)\|.
% \end{align}
% We are ready to determine the convergence behavior of the MLE under this settings. 

\vspace{0.5em}
\noindent
We now proceed to analyze the convergence behavior of the MLE under the
heterogeneous-expert regime.
\begin{theorem}
\label{thm:not_equal}
% Suppose that the pre-trained expert $h_0$ has a different structure from the prompt expert $h$, that is, $h_0 \neq h$.
%Suppose that the pre-trained expert and the adapter expert belong to different function families. 
%That is, there exist function classes $\mathcal{H}_0$ and $\mathcal{H}$ such that 
%$h_0(\cdot;\eta_0)\in\mathcal{H}_0$ and $h(\cdot;\eta)\in\mathcal{H}$, with $\mathcal{H}_0 \neq \mathcal{H}$.
Under the heterogeneous-experts regime, there exists a positive constant $C_2$ that depends on $\Xi$ such that the Hellinger lower bound $\mathbb{E}_X\left[ d_H(p_{ G}\left(\cdot|X),p_{G_*}(\cdot|X)\right) \right]\geq C_2D_2( G, G_*)$ holds for all parameters $G\in\Xi$. 
As a result, we obtain
% \begin{align}
%     \label{eq:dis_bound_1}
%     \sup_{G_*\in\Xi}\mathbb{E}_{p_{\Gs,n}} 
%     \Big[
%     |
%     \exp(\widehat{\tau}_n)
%     -
%     \exp(\tau^*)
%     |^2 
%     \Big] 
%     &\lesssim \log(n)/n,
%     \\
%     \label{eq:dis_bound_2}
%     \sup_{G_*\in\Xi}\mathbb{E}_{p_{\Gs,n}} 
%     \Big[
%     \exp^2(\tau^*)
% \Big( \|\widehat{\beta}_n - \beta^*\|^2 
% + 
%     \nonumber
%     \\
%     \sum_{i=1}^{K-1} \|\widehat{\eta}_{n,i} - \eta_{i}^*\|^2  \Big)
%     \Big] 
%     &\lesssim 
%     \log(n)/n.
% \end{align}
% \begin{align}
%     \label{eq:dis_bound_1}
%     \sup_{G_*\in\Xi}\mathbb{E}_{p_{\Gs,n}} 
%     \Big[
%     |
%     \exp(\widehat{\tau}_n)
%     -
%     \exp(\tau^*)
%     |^2 
%     \Big] 
%     &\lesssim \log(n)/n,
%     \\
%     \label{eq:dis_bound_2}
%     \sup_{G_*\in\Xi}\mathbb{E}_{p_{\Gs,n}} 
%     \Big[
%     \exp^2(\tau^*)
% \Big( \|\widehat{\beta}_n - \beta^*\| 
% + 
%     \nonumber
%     \\
%     \sum_{i=1}^{K-1} \|\widehat{\eta}_{n,i} - \eta_{i}^*\|  \Big)^2
%     \Big] 
%     &\lesssim 
%     \log(n)/n.
% \end{align}
\begin{align}
    \label{eq:dis_bound_1}
    &\hspace{-0.5em}
    \sup_{G_*\in\Xi}
    \hspace{-0.3em}
    \mathbb{E}_{p_{\Gs,n}} 
    \Big[
    |
    \exp(\widehat{\tau}_n)
    -
    \exp(\tau^*)
    | 
    \Big] 
    \lesssim {(\log(n)/n)}^{1/2},
    \\
    \label{eq:dis_bound_2}
    &
    \hspace{-0.5em}
    \sup_{G_*\in\Xi}
    \hspace{-0.3em}
    \mathbb{E}_{p_{\Gs,n}} 
    \Big[
\Big( \|\widehat{\beta}_n - \beta^*\| 
+ 
    \sum_{i=1}^{K-1} \|\widehat{\eta}_{n,i} - \eta_{i}^*\|  \Big)
    %     \nonumber
    % \\
    %     &
    % \hspace{3.15cm}
    \times
    \exp(\tau^*)
    \Big] 
    \lesssim 
    (\log(n)/n)^{1/2}.
\end{align}
% \begin{align}
% \label{applossineq:D1-loss}
%     V(p_{ G},p_{G_*})\geq C_1D_1( G, G_*).
% \end{align}
\end{theorem}
\noindent

\vspace{0.5em}
\noindent
The proof of Theorem~\ref{thm:not_equal} is deferred to
Appendix~\ref{proof:not_equal}.

\vspace{0.5em}
\noindent
In contrast to the convergence behavior established for the
homogeneous-expert regime in Theorem~\ref{thm:d2_mle_rate}, the rates obtained
above in the heterogeneous-expert regime exhibit several notable differences.
In particular, equation~\eqref{eq:dis_bound_1} shows that the gating parameter
estimator $\exp(\widehat{\tau}_n)$ converges to its ground-truth counterpart
$\exp(\tau^*)$ at the near-parametric rate
$\widetilde{\mathcal{O}}(n^{-1/2})$, up to logarithmic factors. This rate is
strictly faster than that in the homogeneous-expert regime, where the
convergence of $\exp(\widehat{\tau}_n)$ is affected by the vanishing behavior
of the adapter parameters $\Delta \eta_i^*$.

\vspace{0.5em}
\noindent
Moreover, equation~\eqref{eq:dis_bound_2} implies that the estimation errors of
the gating parameter $\widehat{\beta}_n$ and the adapter expert parameter
$\widehat{\eta}_n$ are likewise controlled at the same near-parametric rate
$\widetilde{\mathcal{O}}(n^{-1/2})$, up to the scaling factor $\exp(\tau^*)$.
This dependence reflects the fact that the estimability of $(\beta,\eta)$ is
intrinsically governed by the relative contribution of the adapter expert to
the overall mixture. Again, this behavior stands in sharp contrast to the
homogeneous-expert regime characterized in
equation \eqref{eq:nondis_bound_2}, where the structural coincidence of $h$ and $h_0$
causes the adapter parameters $\Delta \eta_i^*$ to enter the convergence
rates, resulting in slower estimation.

% Finally, we recall that the parameter space $\Xi$ is compact, so that $\exp(\tau^*)$ is bounded away from both zero and infinity.

% Meanwhile, in the contaminated MoE with input-free gating in \cite{yan2025contaminated}, the estimation rates for prompt parameters $\eta^*,\nu^*$ are slower than $\widetilde{\mathcal{O}}(n^{-1/2})$ as they depend on the convergence rate of the gating parameter to zero. 
% Therefore, replacing the input-free gating with the softmax gating in the contaminated MoE helps reduce the sample complexity of parameter estimation. 

\noindent
% Given the near-parametric convergence rates established in Theorem~\ref{thm:not_equal}, it is natural to ask whether these rates are optimal. To answer this question in the affirmative, we derive minimax lower bounds below.
% The near-parametric convergence guarantees obtained in Theorem~\ref{thm:not_equal} raise the question of whether such rates are statistically optimal. We address this question by establishing matching minimax lower bounds as follows.
% The near-parametric convergence rates established in
% Theorem~\ref{thm:not_equal} naturally raise the question of their statistical
% optimality. We answer this question by deriving matching minimax lower bounds,
% as presented below.
We next establish minimax lower bounds to characterize the optimality of the
rates in Theorem~\ref{thm:not_equal}.
\begin{theorem}
\label{thm:d1_minimax}
% Under the assumption in Theorem \ref{thm:not_equal},
% the following minimax lower bounds 
Suppose that the assumptions of Theorem~\ref{thm:not_equal} hold. Then the following minimax lower bounds 
% \begin{align*}
% &    \inf_{\overline{G}_n\in \Xi }\sup_{G\in \Xi }
%     \mathbb{E}_{p_{G,n}} \Big( 
%     |\exp(\overline{\tau}_n)
%     -\exp(\tau)|^2 \Big) 
%     \gtrsim n^{-1/r},
%     \\
% &\inf_{\overline{G}_n\in \Xi }\sup_{G\in \Xi }
%     \mathbb{E}_{p_{G,n}} 
%     \Big( \exp^2(\tau) 
%     \left[ \|\overline{\beta}_n - \beta\|^2 
% + \sum_{i=1}^{K-1} \big( \overline{\eta}_{n,i} - \eta_{i}\|^2 \big) \right] 
%     \Big) 
%     \\
%     \gtrsim n^{-1/r},
% \end{align*}
\begin{align*}
&    \inf_{\overline{G}_n\in \Xi }\sup_{G\in \Xi }
    \mathbb{E}_{p_{G,n}} \Big( 
    |\exp(\overline{\tau}_n)
    -\exp(\tau)|^2 \Big) 
    \gtrsim n^{-1/r},
    \\
&\inf_{\overline{G}_n\in \Xi }\sup_{G\in \Xi }
    \mathbb{E}_{p_{G,n}} 
    \Big(  
    \Big[ \|\overline{\beta}_n - \beta\|^2 
+ \sum_{i=1}^{K-1} \| \overline{\eta}_{n,i} - \eta_{i}\|^2  \Big] 
    % \\&\hspace{4cm}
    \times \exp^2(\tau)
        \Big) 
    \gtrsim n^{-1/r},
\end{align*}
hold for any $0<r < 1$.
% where the infimum is over all estimators taking values in $\Xi$.
% Here, the infimum is taken over all sequences of estimates $(\widehat{\beta}_n, \widehat{\tau}_n, \widehat{\eta}_n)$.
% Here, the infimum is taken over all sequences of estimators
% $\overline{G}_n = (\overline{\beta}_n, \overline{\tau}_n, \overline{\eta}_n)$
% taking values in the parameter space $\Xi$, and the expectation is with respect to the joint distribution $p_{G,n}$ of the $n$ i.i.d.\ observations generated under parameter $G$.
Here, the infimum is taken over all  estimators
$\overline{G}_n = (\overline{\beta}_n,\overline{\tau}_n,\overline{\eta}_n)$
that take values in the parameter space $\Xi$, and the expectation is with respect to the joint distribution $p_{G,n}$ of the $n$ i.i.d.\ observations generated under parameter $G$.
\end{theorem}
\noindent
The proof of Theorem \ref{thm:d1_minimax} can be found in Appendix \ref{proof:d1_minimax}.
% % See Appendix~\ref{proof:d1_minimax} for the proof of
% % Theorem~\ref{thm:d1_minimax}.
% Theorem~\ref{thm:d1_minimax} establishes matching minimax lower bounds for the% heterogeneous-experts regime, confirming that the near-parametric rates achieved
% by the MLE are optimal up to logarithmic factors.
% % These lower bounds reflect the intrinsic difficulty of disentangling the contribution of the adapter expert from that of the pretrained expert when the two exhibit different structural forms.
% The proof of Theorem~\ref{thm:d1_minimax} is provided in
% Appendix~\ref{proof:d1_minimax}. 
This theorem confirms that the rates in
Theorem~\ref{thm:not_equal} are minimax optimal up to logarithmic factors in
the heterogeneous-expert regime.

\section{Numerical Experiments}
\label{sec:experiments}

In this section, we conduct experiments with a softmax-gated contaminated mixture of multinomial logistic experts model consisting of one fixed expert and one adapter expert. To empirically investigate the impact of expert identifiability on estimation behavior, we compare two regimes: a \textit{homogeneous-expert} regime, where the adapter expert and the pretrained expert share the same function family, and a \textit{heterogeneous-expert} regime, where they come from different function families.

\vspace{0.5em}
\noindent
\textbf{Experimental Setup.} For each sample size $n$, we repeat the experiment over ten independent runs with different random seeds. In each run, a new dataset is generated and the model parameters are estimated via the EM algorithm \cite{jordan1994hierarchical}, employing BFGS optimizer \cite{broyden1967quasi, fletcher1970new, goldfarb1970family, shanno1970conditioning} for the M-step due to the absence
of a universal closed-form solution. 
We report the mean and standard deviation of the estimation error across runs. To summarize empirical convergence rates, we fit a linear regression on the log–log scale of the mean error versus the sample size.

\vspace{0.5em}
\noindent
\textbf{Synthetic Data Generation.} We generate synthetic data according to the softmax-gated contaminated mixture of multinomial logistic experts model
described in equation \eqref{eq:contaminated_pretrain_model_general}.
For a given sample size $n$, feature vectors are independently drawn as
$X_i \sim \mathcal{N}(0, I_d)$, where $d$ denotes the feature dimension.
Given $X_i$, a logistic gate with probability
$\pi(X_i;\beta^*,\tau^*) = \{1+\exp(-(\beta^{*})^\top X_i-\tau^*)\}^{-1}$
selects the adapter expert, while the pretrained expert is selected with
probability $1-\pi(X_i;\beta^*,\tau^*)$.
The class label $Y_i$ is subsequently generated from the corresponding
expert-specific multinomial logistic model.

\vspace{0.5em}
\noindent
Throughout the experiments, we fix the data-generating parameters as follows.
The number of classes is set to $K=3$ and the feature dimension to $d=8$.
Across both homogeneous and heterogeneous expert regimes, the gating parameters are
$\beta^* = {1}/{\sqrt{d}}\,\mathbf{1}_d$ and $\tau^* = 0$.
The adapter expert parameter is defined as
$\eta_n^* = (1 + 10 n^{-3/8}) \eta_0$, where
$\eta_0 \in \mathbb{R}^{d\times K}$ is the pretrained expert parameter given by
$\eta_0 = e_2 (0.3, -1.5, 0)^\top$, with $e_2 \in \mathbb{R}^d$ denoting the
second canonical basis vector.
The initialization is chosen to be relatively close to these ground truth values in both settings to mitigate potential optimization instabilities.

\vspace{0.5em}
\noindent
\textbf{Results.} Figure~\ref{fig:non_distinguishable_all} presents the results under the homogeneous setting,
in which the pre-trained and adapter experts share the same activation function.
Specifically, Figures~\ref{fig:non_dist_tanh} and~\ref{fig:non_dist_gelu} correspond to
configurations using $\tanh$ and $\GELU$, respectively.
In both cases, the two experts become functionally similar, leading to substantially
slower error decay compared to the heterogeneous setting.
Quantitatively, the estimation error of the mixing weight $\exp(\tau)$ decays with rates
$\mathcal{O}(n^{-0.21})$ and $\mathcal{O}(n^{-0.22})$ for the $\tanh$--$\tanh$ and
$\GELU$--$\GELU$ configurations, respectively, which are close to
$\mathcal{O}(n^{-1/4})$.
Meanwhile, the errors of the gating parameters $\beta$ exhibit rates around
$\mathcal{O}(n^{-0.32})$ and $\mathcal{O}(n^{-0.34})$, and the estimation rates of the
adapter parameters $\eta$ are $\mathcal{O}(n^{-0.43})$ and $\mathcal{O}(n^{-0.41})$,
both of which are approximately $\mathcal{O}(n^{-3/8})$.
This result empirically confirms our theoretical result in
Theorem~\ref{thm:d2_mle_rate} that the emergence of homogeneous structures leads to
diminished identifiability, with convergence rates slower than the parametric benchmark
$\mathcal{O}(n^{-1/2})$.

\vspace{0.5em}
\noindent
By contrast, Figure~\ref{fig:distinguishable_all} reports the results under the
heterogeneous setting, in which the pre-trained and adapter experts employ different
function families.
Case (i), shown in Figure~\ref{fig:dist_case_linear_tanh}, assigns a linear map
$h_0(x;\eta_{0s}) = x^\top \eta_{0s}$ to the pre-trained expert and a
$\tanh$-activated map $h(x;\eta_s) = \tanh(x^\top \eta_s)$ to the adapter expert,
while Case (ii), shown in Figure~\ref{fig:dist_case_tanh_linear}, reverses this assignment.
As shown in the log--log plots, the estimation error of $\exp(\tau)$ follows empirical convergence rates on the order of
$\mathcal{O}(n^{-0.50})$ in Case~(i) and $\mathcal{O}(n^{-0.52})$ in Case~(ii).
Furthermore, in both configurations, the estimation errors associated with the gating
parameters $\beta$ and the adapter parameters $\eta$ 
is $\mathcal{O}(n^{-0.50})$ and $\mathcal{O}(n^{-0.51})$,
close to the rate
$\mathcal{O}(n^{-1/2})$.
These results suggest that, as established in Theorem~\ref{thm:not_equal}, the
structural heterogeneity between the pretrained and adapter experts yields
faster convergence rates than in the homogeneous regime.

%     \begin{figure*}[h]
%         \centering
%         \subfloat[
%         \textbf{Case (i):} 
%         {$h_0(x;\eta_{0s})\mathord{=}\tanh(x^\top \eta_{0s}); h(x;\eta_s)\mathord{=}\tanh(x^\top \eta_s).$}.
%         \label{fig:non_dist_tanh}
%         ]
%         {
%             \includegraphics[width=\textwidth]{figs/non_distinguishable_tanh.pdf}
%         }
        
%         \subfloat[
%         \textbf{Case (ii):} $
% h_0(x;\eta_{0s})\mathord{=}\GELU(x^\top \eta_{0s}); 
% h(x;\eta_s)\mathord{=}\GELU(x^\top \eta_s)
% $.
%         \label{fig:non_dist_gelu}
%         ]
%         {
%         \includegraphics[width=\textwidth]{figs/non_distinguishable_gelu.pdf}
%         }
%   \caption{\textbf{Homogeneous-expert regime.} Log–log plots of parameter estimation errors versus the sample size $n$
%  in the homogeneous setting.
% Figures~\ref{fig:non_dist_tanh} and~\ref{fig:non_dist_gelu} illustrate two
% homogeneous-expert configurations with shared activations: $\tanh$ in
% Figure~\ref{fig:non_dist_tanh} and $\GELU$ in Figure~\ref{fig:non_dist_gelu}.
% In each figure, blue dots denote the mean estimation error across different runs at each sample size with vertical error bars indicating one standard deviation, while the orange dashed line shows the fitted power-law trend.}
%   \label{fig:non_distinguishable_all}
%     \end{figure*}

\begin{figure}[t]
  \centering

  \begin{subfigure}{\textwidth}
    \centering
    \includegraphics[width=\textwidth]{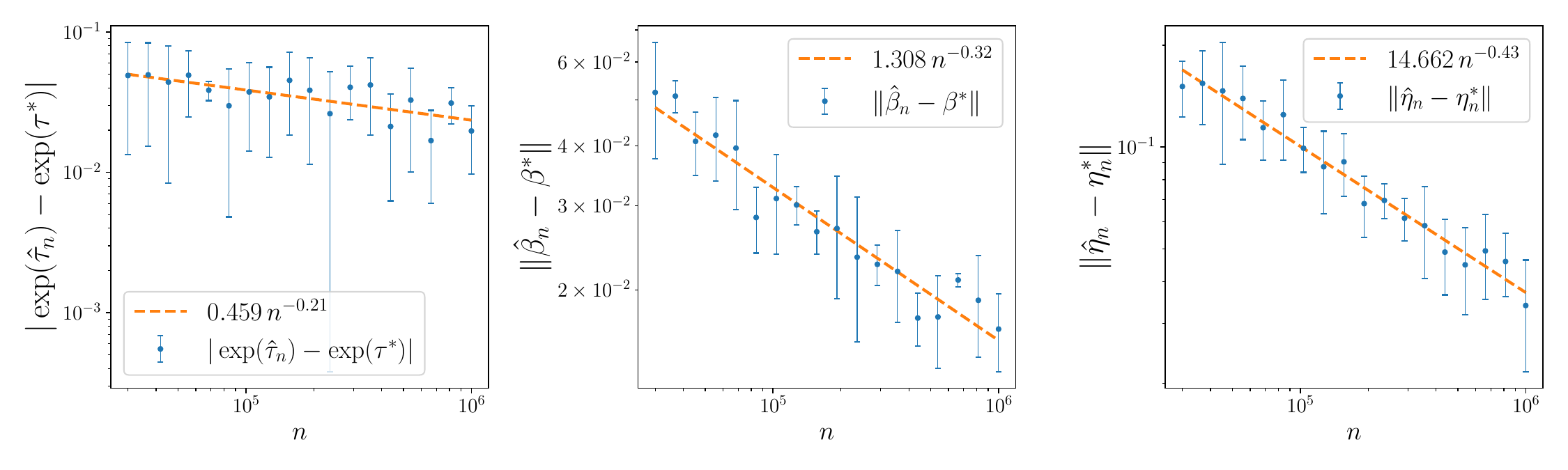}
    \caption{Case (i): \textnormal{$h_0(x;\eta_{0s})\mathord{=}\tanh(x^\top \eta_{0s}); 
h(x;\eta_s)\mathord{=}\tanh(x^\top \eta_s).$}}
    \label{fig:non_dist_tanh}
  \end{subfigure}

  \vspace{0.6em}

  \begin{subfigure}{\textwidth}
    \centering
    \includegraphics[width=\textwidth]{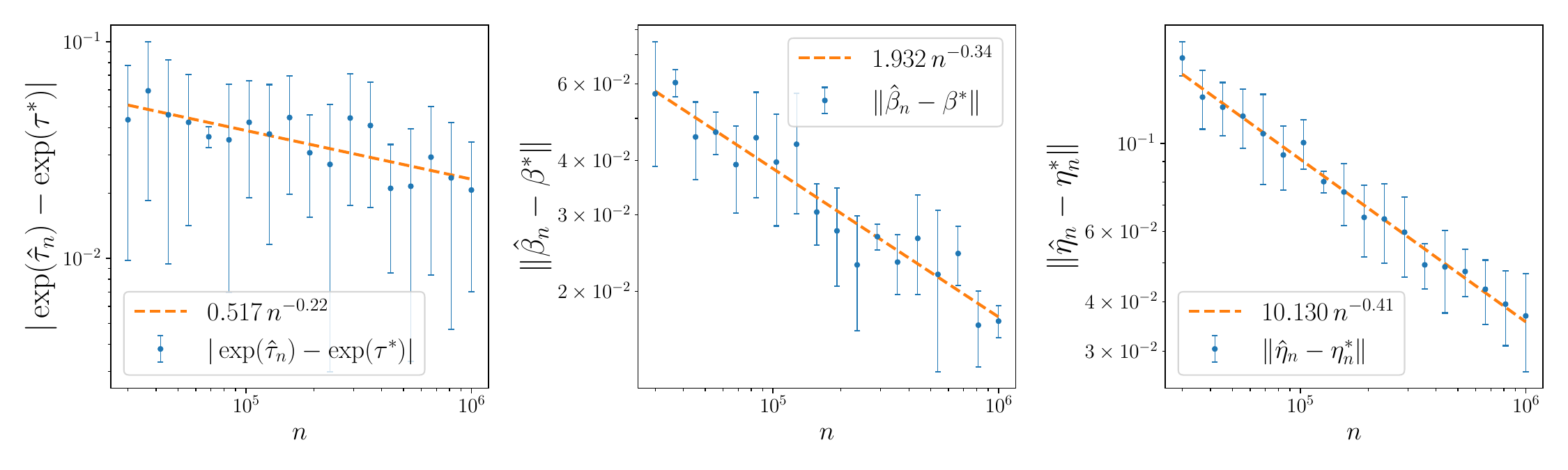}
    \caption{Case (ii): \textnormal{$
h_0(x;\eta_{0s})\mathord{=}\GELU(x^\top \eta_{0s}); 
h(x;\eta_s)\mathord{=}\GELU(x^\top \eta_s).
$}}
  \label{fig:non_dist_gelu}
  \end{subfigure}

  \caption{\textbf{Homogeneous-expert regime.} Log–log plots of parameter estimation errors versus the sample size $n$
 in the homogeneous setting.
Figures~\ref{fig:non_dist_tanh} and~\ref{fig:non_dist_gelu} illustrate two
homogeneous-expert configurations with shared activations: $\tanh$ in
Figure~\ref{fig:non_dist_tanh} and $\GELU$ in Figure~\ref{fig:non_dist_gelu}.
In each figure, blue dots denote the mean estimation error across different runs at each sample size with vertical error bars indicating one standard deviation, while the orange dashed line shows the fitted power-law trend.}
  \label{fig:non_distinguishable_all}
\end{figure}

\begin{figure}[t]
  \centering

  \begin{subfigure}{\textwidth}
    \centering
    \includegraphics[width=\textwidth]{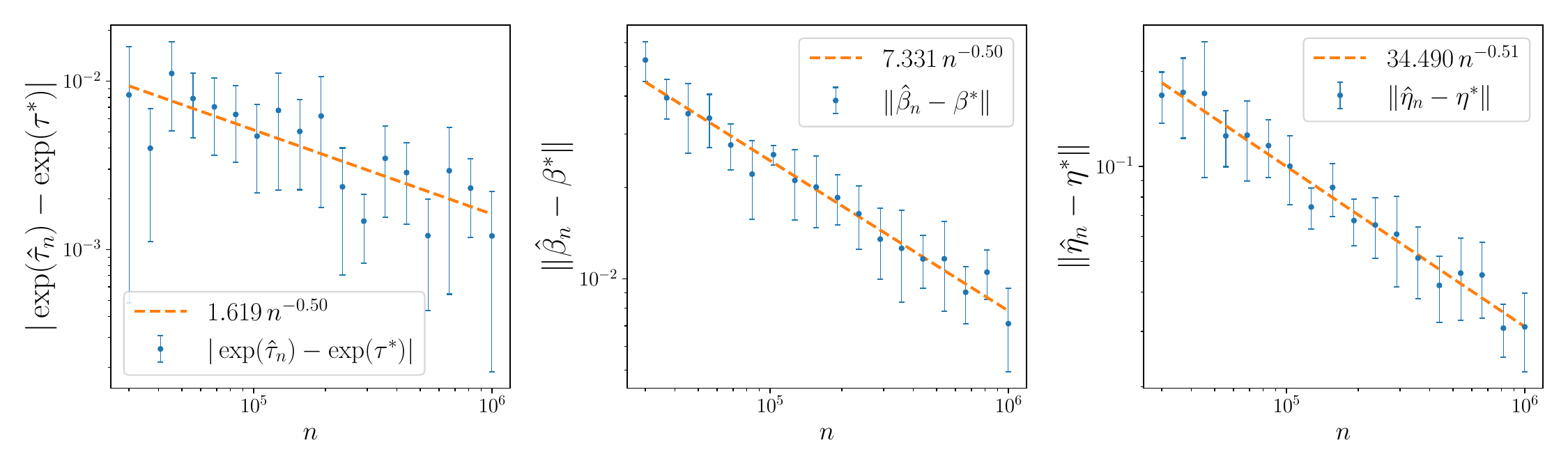}
     \caption{Case (i): \textnormal{$h_0(x;\eta_{0s})\mathord{=}x^\top \eta_{0s}; 
h(x;\eta_s)\mathord{=}\tanh(x^\top \eta_s).$}}
    \label{fig:dist_case_linear_tanh}
  \end{subfigure}

  \vspace{0.6em}

  \begin{subfigure}{\textwidth}
    \centering
    \includegraphics[width=\textwidth]{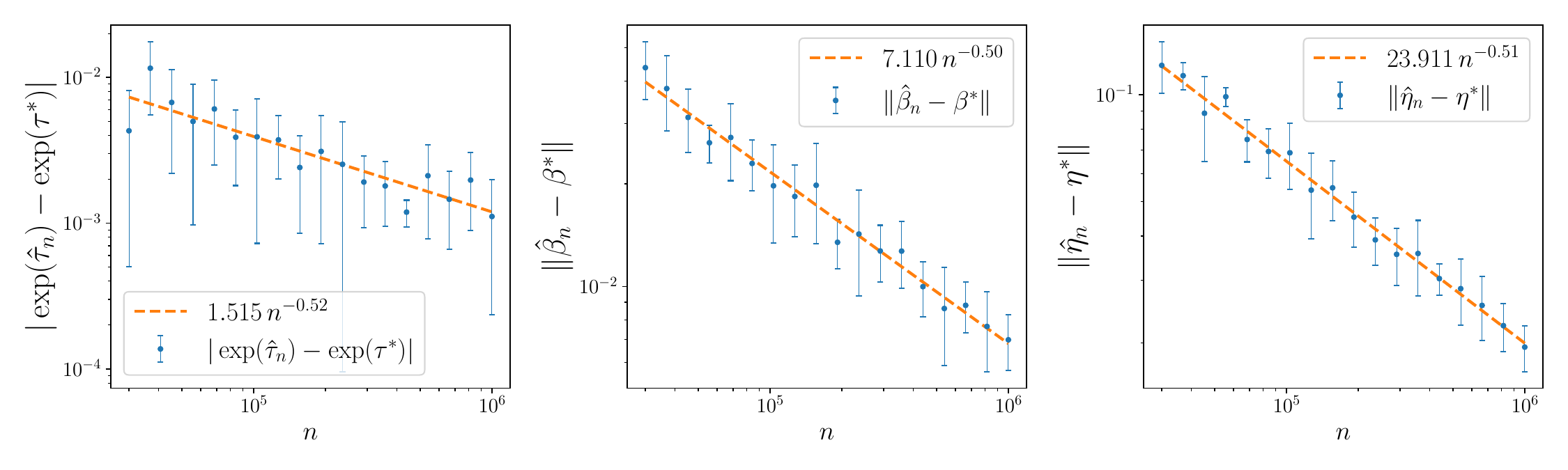}
     \caption{Case (ii): \textnormal{$h_0(x;\eta_{0s})\mathord{=}\tanh(x^\top \eta_{0s}); 
h(x;\eta_s)\mathord{=}x^\top \eta_s.$}}
    \label{fig:dist_case_tanh_linear}
  \end{subfigure}

  \caption{\textbf{Heterogeneous-expert regime.} Log–log plots of parameter estimation errors versus the sample size $n$
 in the heterogeneous setting. 
 Figures~\ref{fig:dist_case_linear_tanh} and~\ref{fig:dist_case_tanh_linear}
compare two configurations with swapped linear and $\tanh$ activations between
the pretrained and adapter experts.
 The remaining visualization conventions are the same as in Figure \ref{fig:non_distinguishable_all}.}
  \label{fig:distinguishable_all}
\end{figure}

\section{Conclusion}
\label{sec:conclusion}

In this article, we conduct a first comprehensive study about the 
softmax-gated contaminated mixture of multinomial logistic experts model
and analyse the convergence behavior of their maximum likelihood estimation. A key difficulty arises from the similarity between the adapter component and the pretrained model, which motivates us to distinguish between two cases: homogeneous-expert regime and heterogeneous-expert regime. Our findings show that the heterogeneous setting is statistically more favorable—our estimator attains the parametric minimax convergence rate in terms of sample size. In contrast, the homogeneous setting creates unnecessary overlap between the pretrained and adapter components, resulting in a significantly slower convergence rate for parameter estimation. 

\vspace{0.5em}
\noindent
Our analysis highlights the critical role of expert selection during the fine-tuning process. In particular, the heterogeneous setting yields a substantially faster convergence rate, suggesting that practitioners should avoid reusing the same expert as the one employed in the pre-trained model. For example, although linear experts are commonly used in large-scale pretrained architecture, in order to achieve a good performance, choosing a different expert function—such as $\tanh$ or $\GELU$—during fine-tuning can lead to better statistical performance.

\vspace{0.5em}
\noindent
The insights developed here motivate several research directions. First, an important extension is to study a more challenging and realistic setting in which a pretrained model is fine-tuned using multiple adapter components rather than a single adapter model \cite{nguyen2024deviated}. Second, it would be valuable to explore more general gating mechanisms beyond the standard softmax function. Investigating alternative gating functions that can better mitigate the interaction between experts may lead to improved stability and overall performance of the model.

\newpage

\appendix
\centering
\textbf{ \Large{Supplementary Material for\\
``Improving Minimax Estimation Rates for Contaminated Mixture 
of Multinomial Logistic Experts via Expert Heterogeneity''}}

\justifying
\setlength{\parindent}{0pt}
\textbf{}\\
\noindent
In this supplementary document, we compile the technical proofs that were left out of the main paper and provide further information on our experimental setup. Appendix~\ref{appsec:main_results} gathers the proofs of the principal theoretical findings, including the results on parameter convergence rates and the minimax lower bounds discussed in Section~\ref{sec:theory}. Additional lemmas and propositions that clarify the core properties of the softmax-contaminated MoE model—introduced earlier in Section~\ref{sec:preliminaries}—are proved in Appendix~\ref{appendix:ProofsforAuxiliaryResults}. Supplementary empirical results for the heterogeneous-expert scenario are reported in Appendix~\ref{app:exp}. The computational infrastructure used for all experiments is described in Appendix~\ref{app:infras}.

\section{Proof of Main Results}
\label{appsec:main_results}
\subsection{Homogeneous-expert regime}
In this section, we provide the proofs for Theorem \ref{thm:d2_mle_rate} and Theorem \ref{thm:d2_minimax} in homogeneous-expert regime. 
\subsubsection{Proof of Theorem \ref{thm:d2_mle_rate}}
\label{app_proof: d2_loss}

\noindent
We now turn to the proof of Theorem~\ref{thm:d2_mle_rate} in the setting where the pre-trained expert $h_0$ and the adapter expert $h$ share the same structure.
% \begin{align*}
%     \overline{D_2}\left( G, G_*) \right)
%  &:=
%     \|\beta-\beta^*\|
%     \exp(\tau)
%     \Vert (\Delta\eta,\Delta\nu) \Vert
%     \\&
%     +
%     \left|
%     \exp(\tau^*)-\exp(\tau)
%     \right|
%     \cdot
%     \Vert (\Delta\eta,\Delta\nu)\Vert
%     \cdot
%     \Vert (\Delta\eta^*,\Delta\nu^*)\Vert
%     \\&
%     +\Vert  (\Delta\eta,\Delta\nu)-(\Delta\eta^*,\Delta\nu^*) \Vert
%     \big( 
%     \exp(\tau)\Vert (\Delta\eta,\Delta\nu) \Vert
%     +\exp(\tau^*)\Vert (\Delta\eta^*,\Delta\nu^*) \Vert
%     \big)
%     \\&
%     +
%     \exp(\tau+\tau^*)\cdot
%     \Vert (\Delta\eta,\Delta\nu)-(\Delta\eta^*,\Delta\nu^*)\Vert^2
% \end{align*}

\begin{proof}

% First we denote $\overline{D_2}(G,\Gs)$ by
% \begin{align*}
%     \overline{D_2}\left( G, G_*) \right)
%  &:=
%     \|\beta-\beta^*\|
%     % \exp(\tau)
%     % \Vert (\Delta\eta,\Delta\nu) \Vert
%     \big( 
%     \exp(\tau)\Vert (\Delta\eta,\Delta\nu) \Vert
%     +\exp(\tau^*)\Vert (\Delta\eta^*,\Delta\nu^*) \Vert
%     \big)
%     \\&
%     +
%     \left|
%     \exp(\tau^*)-\exp(\tau)
%     \right|
%     \cdot
%     \Vert (\Delta\eta,\Delta\nu)\Vert
%     \cdot
%     \Vert (\Delta\eta^*,\Delta\nu^*)\Vert
%     \\&
%     +\Vert  (\Delta\eta,\Delta\nu)-(\Delta\eta^*,\Delta\nu^*) \Vert
%     \big( 
%     \exp(\tau)\Vert (\Delta\eta,\Delta\nu) \Vert
%     +\exp(\tau^*)\Vert (\Delta\eta^*,\Delta\nu^*) \Vert
%     \big)
%     \\&
%     +
%     \exp(\tau+\tau^*)\cdot
%     \Vert (\Delta\eta,\Delta\nu)-(\Delta\eta^*,\Delta\nu^*)\Vert^2
% \end{align*}

Let $\overline{G}=(\Bar{\beta},\Bar{\tau},\Bar{\eta})$ 
and $\Bar{\eta}$ can be identical to ${\eta_0}$. 
Then, we will show that
\begin{itemize}
    \item[(i)] When 
    % $({a_0},{b_0})\neq(\Bar{a},\Bar{b})$,
    ${\eta_0}\neq \Bar{\eta}$,
    \begin{align*}
     \lim_{\varepsilon\to 0}\inf_{G,\Gs}\left\{\frac{\|p_{G} - p_{G_*}\|_{\infty}}{D_2(G,G_*)}:D_2(G,\overline{G})\vee D_2(G_*,\overline{G})\leq\varepsilon\right\}>0.
    \end{align*}
    \item[(ii)] 
    % When $({a_0},{b_0})=(\Bar{a},\Bar{b})$,
    When ${\eta_0} = \Bar{\eta}$,
    \begin{equation}
    \label{eq:claim_nondistinguishable_independent}
     \lim_{\varepsilon\to 0}\inf_{G,\Gs}\left\{\frac{\|p_{G} - p_{G_*}\|_{\infty}}{D_1(G,G_*)}:D_1(G,\overline{G})\vee D_1(G_*,\overline{G})\leq\varepsilon\right\}>0.
    \end{equation}
\end{itemize}
 Part (i) can be proved by using the same arguments as in the proof of Theorem \ref{thm:not_equal} (see Section \ref{proof:not_equal} for details).
 Thus, we will consider only part (ii) in this section, specifically the most challenging setting that 
 ${\eta_0} = \Bar{\eta}$. %$({a_0},{b_0})=(\Bar{a},\Bar{b})$. 
 Under this assumption, we know that $h_i(X,\eta_0):=h_i(X,\eta_{0i})$ and $h_i(X,\eta):=h_i(X,\eta_i)$ are the same 
expert function, for all $i\in[K]$.
s.t. 
$f(Y=s|X;\eta_0)=f(Y=s|X;\eta)$ for almost surely $(X,Y)\in \mathcal{X}\times\mathcal{Y}$.
Assume that the above claim in equation~\eqref{eq:claim_nondistinguishable_independent} does not hold, 
then there exist two sequences $G_n=(\beta_n,\tau_n,\eta_n)$ and $G_{*,n}=(\beta^*_n,\tau^*_n,\eta_n^*)$, 
% and two sequences of gating parameters $\lambda_n=(\beta_n,\tau_n)$ and $\lambda_n^*=(\beta^*_n,\tau^*_n) \in \Lambda$, a compact set, 
such that
\begin{align*}
    \begin{cases}
        D_1(\Gn,\overline{G})
        \to 0,\\
        D_1(\Gsn,\overline{G})
        \to 0,
        \\ {\|p_{ \Gn} - p_{  \Gsn}\|_{\infty}}/{D_1(\Gn,\Gsn)}\to 0.
    \end{cases}
\end{align*}
% Now, we have three primary cases regarding the convergence behaviors between $(\lambda_n,G_n)=(\beta_n,\tau_n,\eta_{n},\nu_{n})$ and $(\lambda_n^*,G_{*,n})=(\beta_n^*, \tau_n^*, \eta^*_{n},\nu^*_{n})$.\\
% 1. $G_n,G_n^*\to G'= G_0$\\
% 2. $G_n,G_n^*\to G'\neq G_0$\\
% 3. $G_n\to G', G_n^*\to G_0$ or $G_n\to G_0, G_n^*\to G'$\\
We now analyze the limiting behavior of the sequences $G_n$ and $G_{*,n}$ as they approach $\overline{G}$. In particular, we distinguish between three asymptotic regimes based on how the expert parameters 
$\eta_n$ and 
$\eta_n^*$ converge.

First, it may occur that both $\eta_n$ and $\eta_n^*$ converge to the same limit $\eta_0$. Alternatively, both sequences may converge to a common limit $\eta' \neq \eta_0$, which is distinct from the true expert. Finally, it is also possible that one sequence converges to $\eta_0$ while the other converges to a different point $\eta' \neq \eta_0$.

In the following, we analyze each of these cases and demonstrate that in all scenarios, the assumption that the normalized difference vanishes leads to a contradiction.

\subsubsection*{Case 1:}
At first we consider that 
$\eta_n$ and $\eta_n^*$ share the same limit of $\eta_0$. Without loss of generality, we can suppose that  $\tau_n^*\geq \tau_n$. Subsequently, we consider
$W_n:=[p_{  G_{n}}(Y=s|X)-
    p_{  G_{*,n}}(Y=s|X)]
    \cdot[1+\exp((\beta_n^*)^{\top}X+\tau_n^*)]
    \cdot[1+\exp((\beta_n)^{\top}X+\tau_n)]$,
which can decomposed as
\begin{align*}
    W_n
    &=
    \exp(\tau_n) 
    \cdot
    [g(Y=s|X;\beta_n,\eta_n)-g(Y=s|X;\beta_n^*,\eta_n^*)]
    \\&
    -
    \exp(\tau_n)
    \cdot
    [g(Y=s|X;\beta_n,\eta_0)-g(Y=s|X;\beta_n^*,\eta_n^*)]
    \\&
    +\exp(\tau^*_n)
    \cdot
    [g(Y=s|X;\beta_n^*,\eta_0)-g(Y=s|X;\beta_n^*,\eta_n^*)]
    \\&
    +\exp\left( (\beta_n^* + \beta_n)^{\top}X + \tau_n^* + \tau_n \right)
    \cdot[f(Y=s|X,\eta_n)-f(Y=s|X,\eta_n^*)]
    \\&
    :=\Ione_{n}-\Itwo_{n}+\Ithree_{n}+\Ifour_{n}
\end{align*}
where we denote
$
    g(Y=s|X;\beta, \eta) =
    e(X;\beta)f(Y=s|X; \eta),
    % =
    % \exp\left(\beta^{\top}X\right) f\left(Y=s |  h(X,\eta), \nu\right).
$
for any $s\in [K]$, and $e(X;\beta) = \exp(\beta^\top X)$.
We perform a local expansion around the reference parameters $\beta_n^*, \eta_n^*$, 
where the parameter differences are defined by
$
    \Delta \eta_{ni} = \eta_{ni} - \eta_{0i}, 
    % \Delta  b_{ni} =  b_{ni} -  b_{0i},
$ and $
    \Delta \eta_{ni}^* = \eta_{ni}^* - \eta_{0i}. 
    % \Delta  b_{ni}^* =  b_{ni}^* -  b_{0i} . 
$
We denote 
% \begin{align*}
% f(Y = s  |  X; \eta) 
% &= \frac{\exp(a_s + b_s^\top X)}{\sum_{\ell=1}^K \exp(a_\ell + b_\ell^\top X)} 
% = \frac{\exp(h_s(X, a_s, b_s))}{\sum_{\ell=1}^K \exp(h_\ell(X, a_\ell, b_\ell))}.
% \end{align*}
$f \bigl(Y=s | X; \eta \bigr)
 = {\exp \bigl(h_{ s}\bigr)}/
         {\sum_{i=1}^{K}
            \exp \bigl(h_{ i}\bigr)} 
 = {\exp \bigl(h_{s}(X,\eta_{ s})\bigr)}/
         {\sum_{i=1}^{K}
            \exp \bigl(h_{i}(X,\eta_{ i})\bigr)}.$
% Applying a second-order Taylor expansion, then we obtain:
% \begin{align}
%     {\Ione}_{n}&=\exp(\tau_n)
%     \Big[
%     \sum_{|\alpha|=1}^2\frac{1}{\alpha!}
%     \prod_{u=1}^d
%     [(\beta_n- \beta^*_n)^{(u)}]^{\alpha_{1u}}
%     \prod_{v=1}^d
%     [(\Delta a_n- \Delta a^*_n)^{(v)}]^{\alpha_{2v}}
%     (\Delta b_n- \Delta b^*_n)^{\alpha_3}
%     \nonumber
%     \\&
%     \hspace{2.5cm}
%     \cdot\frac{\partial^{|\alpha|}g}{\partial \beta^{\alpha_1}\partial  a ^{\alpha_2}\partial b^{\alpha_3}}
%     (Y|X;\beta_n^*, a_n^*, b_n^*)+R_1(X,Y)
%     \Big]\nonumber\\
%     &=\exp(\tau_n)\Big[\sum_{|\alpha|=1}^2\frac{1}{\alpha!2^{\alpha_3}}
%     \prod_{u=1}^d
%     [(\beta_n- \beta^*_n)^{(u)}]^{\alpha_{1u}}
%     \prod_{v=1}^q
%     [(\Delta a_n- \Delta a^*_n)^{(v)}]^{\alpha_{2v}}
%     (\Delta b_n- \Delta b^*_n)^{\alpha_3}
%     \nonumber
%     \\&
%     \hspace{2.5cm}
%     \cdot
%     \exp((\betasn)^{\top}X)
%     \cdot
%     X^{\alpha_1}
%     \frac{\partial^{|\alpha_2|} h}{\partial  a^{|\alpha_2|}}(X,a_n^*)
%     \frac{\partial^{|\alpha_2|+2\alpha_3}f}{\partial h ^{|\alpha_2|+2\alpha_3}}
%     (Y|h\left( X, a_n^*\right), b_n^*)
%     +R_1(X,Y)
%     \Big],
% \end{align}

A second-order Taylor expansion yields
\begin{align}
    {\Ione}_{n}&=\exp(\tau_n)
    \Big[
    \sum_{|\alpha|=1}^2\frac{1}{\alpha!}
    (\beta_n- \beta^*_n)^{\alpha_{1}}
    \prod_{i=1}^{K}
    (\Delta \eta_{ni}- \Delta \eta^*_{ni})^{\alpha_{2i}}
\cdot\frac{\partial^{|\alpha|}g}{\partial \beta_n^{\alpha_1}
    \prod_{i=1}^{K}
    \partial  \eta_{ni}^{\alpha_{2i}}
    }
    (Y=s|X;\beta_n^*, \eta_n^*)+R_1(X,Y)
    \Big],
\end{align}
% where $ R_1(X,Y) $ is the remainder term containing higher-order terms.
% % , and the second equality is due to $\frac{\partial f}{\partial \nu}=\frac{1}{2} \frac{\partial^2 f}{\partial h^2}$. 
% Similarly, we will have that
where $R_1(X,Y)$ is a remainder term collecting higher-order contributions.
Similarly, we have
\begin{align*}
    \Itwo_{n} &= \exp(\tau_n) \Big[ \sum_{|\alpha|=1}^2 \frac{1}{\alpha!} 
    (\beta_n- \beta^*_n)^{\alpha_{1}}
    \prod_{i=1}^{K} (\Delta\eta^*_{ni})^{\alpha_{2i}}
    \cdot\frac{\partial^{|\alpha|}g}{\partial \beta_n^{\alpha_1}
    \prod_{i=1}^{K}
    \partial  \eta_{ni}^{\alpha_{2i}}}
    (Y=s|X;\beta_n^*,\eta_n^*) 
    + R_2(X,Y) \Big],
    \\
    \Ithree_{n} &= \exp(\tau_n^*) \Big[ \sum_{|\alpha|=1}^2 \frac{1}{\alpha!}  
    \prod_{i=1}^{K}(\Delta\eta^*_{ni})^{\alpha_{2i}}
    \cdot\frac{\partial^{|\alpha|}g}{
    \prod_{i=1}^{K}
    \partial  \eta_{ni}^{\alpha_{2i}}}
    (Y=s|X;\beta_n^*,\eta_n^*)     
    + R_3(X,Y) \Big],
    \\
    \Ifour_{n} &= \exp(\tau_n^* + \tau_n)
    \exp\left( (\beta_n^* + \beta_n)^{\top}X\right)
    \Big[ \sum_{|\alpha|=1}^2 \frac{1}{\alpha!} 
    \prod_{v=1}^{K}
    (\Delta\eta_{ni}- \Delta\eta^*_{ni})^{\alpha_{2i}}
    \cdot\frac{\partial^{|\alpha|}f}{
    \prod_{i=1}^{K}
    \partial  \eta_{ni}^{\alpha_{2i}}}
    (Y=s|X;\eta_n^*)  
    + R_4(X,Y) \Big].
\end{align*}
By collecting terms according to the derivative order $\gamma := |\alpha_2|$ and the monomial degree $\zeta := |\alpha_1|$, the expansion can be expressed in a more compact form: 
\begin{align*}
    \Ione_{n}&=
    \sum_{0 \leq |\zeta| \leq 2}
    \left[
    \sum_{0 \leq |\gamma| \leq 2}\Ione_{n,\gamma,\zeta}(X)
   \frac{\partial^{|\gamma|} f}{
      \partial h_1^{|\gamma_{1}|} \cdots 
      \partial h_{K}^{|\gamma_{K}|}
   }(Y = s  |  X; \eta_n^*)
    \exp((\betasn)^{\top}X)
    \right]X^{\zeta}
    +R_1(X,Y) 
\end{align*}
where each coefficient $\mathsf{I}_{n,\gamma,\zeta}(X)$ depends on the parameter differences as well as the derivatives of $h$ with respect to $\eta$.
More specifically, we have that
\begin{align*}
% &\Ione_{n,0,1}(X)=\exp(\tau_n)
% \sum_{1\leq w \leq d}
% (\betan-\betasn)^{(w)}
% \\
&\Ione_{n,0,\zeta}(X)=\exp(\tau_n)
(\betan-\betasn)^{(w)} \text{ where } |\zeta| = 1,\ \zeta_w = 1,  
\\
% &\Ione_{n,0,2}(X)=
% \exp(\tau_n)
% \sum_{1\leq w,r \leq d}
% \frac{(\betan-\betasn)^{(w)}(\betan-\betasn)^{(r)}}{1+\mathbf{1}_{w=r}} 
% \\
&\Ione_{n,0,\zeta}(X)=
\exp(\tau_n)
{(\betan-\betasn)^{(w)}(\betan-\betasn)^{(r)}} \text{ where } |\zeta| = 2,\ \zeta_w = \zeta_r = 1,\ w \neq r,
\\
% &\Ione_{n,1,1}(X)=\exp(\tau_n)
%     \Big[\sum_{1\leq w\leq d, 1\leq u\leq q}{[(\betan-\betasn)^{(w)}][(\Delta\etan-\Delta\etasn)^{(u)}]}\frac{\partial h}{\partial \eta^{(u)}}(X,\etasn)
%     \Big],
% \\
% &\Ione_{n,1,1}(X)=\exp(\tau_n)
%     \Big[\sum_{1\leq w\leq d, 1\leq u\leq q}{[(\betan-\betasn)^{(w)}][(\Delta\etan-\Delta\etasn)^{(u)}]}\frac{\partial h}{\partial \eta^{(u)}}(X,\etasn)
%     \Big],
% \\
&\Ione_{n,\gamma,\zeta}(X)=\exp(\tau_n)
    {[(\betan-\betasn)^{(w)}][(\Delta\etan-\Delta\etasn)^{(u)}]}\frac{\partial h}{\partial \eta^{(u)}}(X,\etasn) \text{ where } |\gamma| = |\zeta| = 1,\ \gamma_u = \zeta_w = 1,
\\
% &\Ione_{n,1,0}(X)=\exp(\tau_n)
%     \Big[\sum_{u=1}^{q}\{(\Delta\etan-\Delta\etasn)^{(u)}\}\frac{\partial  h}{\partial \eta^{(u)}}(X,\etasn)
%     \\&\hspace{2.8cm}
%     +\sum_{1\leq u,v\leq q}\frac{(\Delta \etan-\Delta \etasn)^{(u)}(\Delta \etan-\Delta\etasn)^{(v)}}{1+\mathbf{1}_{u=v}}\frac{\partial^2 h}{\partial \eta^{(u)}\partial \eta^{(v)}}(X,\etasn)
%     \Big],
&\Ione_{n,\gamma,0}(X)=\exp(\tau_n)
   \Big[(\Delta\etan-\Delta\etasn)^{(u)}\frac{\partial  h}{\partial \eta^{(u)}}(X,\etasn)
    +\sum_{v=1}^{q}\dfrac{1}{2}\{(\Delta \etan-\Delta \etasn)^{(u)}(\Delta \etan-\Delta \etasn)^{(v)}\}\frac{\partial^2 h}{\partial \eta^{(u)}\partial \eta^{(v)}}(X,\etasn)\Big] \\
    &\text{ where }|\gamma| = 1, \gamma_u = 1,
\\
    &\Ione_{n,\gamma,0}(X)=\exp(\tau_n)
    \Big[\frac{(\Delta \etan-\Delta \etasn)^{(u)}(\Delta \etan-\Delta\etasn)^{(v)}}{1+\mathbf{1}_{u=v}}
    \frac{\partial h}{\partial \eta^{(u)}}(X,\etasn)
    \frac{\partial h}{\partial \eta^{(v)}}(X,\etasn)
    \Big], \\ 
    &\text{ where }|\gamma| = 2, \gamma_u = 2 \text{ when } u=v, \text{ and } \gamma_u = \gamma_v = 1 \text{ when } u\neq v.  
\end{align*}

% Compact, i–free coefficients. Stack all class blocks.
% Let \boldsymbol{\eta}_n := (\eta_{n1},\dots,\eta_{n,K-1}) \in \mathbb{R}^{m},\ m:=(K-1)q,
% and \boldsymbol{\eta}_n^\ast := (\eta_{n1}^\ast,\dots,\eta_{n,K-1}^\ast).
% Define the stacked Jacobian vector J(X,\eta_n^\ast) \in \mathbb{R}^{m} by
%   J(X,\eta_n^\ast) := \big(J_1^\top,\dots,J_{K-1}^\top\big)^\top,
%   where J_i := \partial h/\partial\eta (X,\eta_{ni}^\ast) \in \mathbb{R}^q.
% Define the block–diagonal Hessian H(X,\eta_n^\ast) \in \mathbb{R}^{m\times m} by
%   H := \mathrm{blkdiag}\big(H_1,\dots,H_{K-1}\big),\ H_i := \partial^2 h/\partial\eta^2 (X,\eta_{ni}^\ast).
% Use component notation v^{(u)} for the u-th entry of any stacked vector v\in\mathbb{R}^m.

Similarly, we can rewrite $\Itwo_{n}$ in the same fashion as follows:
\begin{align*}
    \Itwo_{n}&=
    \sum_{0\leq |\zeta| \leq 2}
    \left[
    \sum_{0 \leq |\gamma|\leq 2}\Itwo_{n,\gamma,\zeta}(X)\frac{\partial^{\gamma} f}{\partial h^{\gamma}}(Y|h(X,\eta_n^*),\nu_n^*)
    \exp((\betasn)^{\top}X)
    \right]X^{\zeta}
    +R_2(X,Y) 
\end{align*}
where 
\begin{align*}
% &\Itwo_{n,0,1}(X)=\exp(\tau_n)
% \sum_{1\leq w \leq d}
% (\betan-\betasn)^{(w)}
% \\
&\Itwo_{n,0,\zeta}(X)=\exp(\tau_n)
(\betan-\betasn)^{(w)} \text{ where } |\zeta| = 1,\ \zeta_w = 1, \\
% &\Itwo_{n,0,2}(X)=
% \exp(\tau_n)
% \sum_{1\leq w,r \leq d}
% \frac{(\betan-\betasn)^{(w)}(\betan-\betasn)^{(r)}}{1+\mathbf{1}_{w=r}} 
% \\
&\Itwo_{n,0,\zeta}(X)=
\exp(\tau_n)
{(\betan-\betasn)^{(w)}(\betan-\betasn)^{(r)}} \text{ where } |\zeta| = 2,\ \zeta_w = \zeta_r = 1,\ w \neq r,
\\
% &\Itwo_{n,1,0}(X)=\exp(\tau_n)
%     \Big[\sum_{u=1}^{q}\{(-\Delta\etasn)^{(u)}\}\frac{\partial  h}{\partial \eta^{(u)}}(X,\etasn)
%     % \\&\hspace{2.8cm}
%     +\sum_{1\leq u,v\leq q}\frac{(-\Delta \etasn)^{(u)}( -\Delta\etasn)^{(v)}}{1+\mathbf{1}_{u=v}}\frac{\partial^2 h}{\partial \eta^{(u)}\partial \eta^{(v)}}(X,\etasn)
%     \Big],
% \\
&\Itwo_{n,\gamma,0}(X)=\exp(\tau_n)
   \Big[(-\Delta\etasn)^{(u)}\}\frac{\partial  h}{\partial \eta^{(u)}}(X,\etasn)
    +\dfrac{1}{2}\sum_{v=1}^q(-\Delta \etasn)^{(u)}( -\Delta\etasn)^{(v)}\frac{\partial^2 h}{\partial \eta^{(u)}\partial \eta^{(v)}}(X,\etasn)\Big] \\
    &\text{ where }|\gamma| = 1, \gamma_u = 1,
\\
&\Itwo_{n,\gamma,\zeta}(X)=\exp(\tau_n)
    {[(\betan-\betasn)^{(w)}][(-\Delta\etasn)^{(u)}]}\frac{\partial h}{\partial \eta^{(u)}}(X,\etasn) \text{ where } |\gamma| = |\zeta| = 1,\ \gamma_u = \zeta_w = 1,
\\
% &\Itwo_{n,1,1}(X)=\exp(\tau_n)
%     \Big[\sum_{1\leq w\leq d, 1\leq u\leq q}{[(\betan-\betasn)^{(w)}][(-\Delta\etasn)^{(u)}]}\frac{\partial h}{\partial \eta^{(u)}}(X,\etasn)
%     \Big],
% \\
&\Itwo_{n,\gamma,0}(X)=\exp(\tau_n)
    \Big[
   \frac{( -\Delta \etasn)^{(u)}( -\Delta\etasn)^{(v)}}{1+\mathbf{1}_{u=v}}
    \frac{\partial h}{\partial \eta^{(u)}}(X,\etasn)
    \frac{\partial h}{\partial \eta^{(v)}}(X,\etasn)
    \Big],\\
&\text{ where }|\gamma| = 2, \gamma_u = 2 \text{ when } u=v, \text{ and } \gamma_u = \gamma_v = 1 \text{ when } u\neq v. 
\end{align*}

In an analogous manner, we can express $\Ithree_{n}$ in the same form as follows. Since the difference for $\beta_n^*$ is zero, so all the coefficients with $\zeta\neq0$ vanish. However, for consistency of notation, we still write $\Ithree_{n}$ as follows
\begin{align*}
    \Ithree_{n}&=
    % \sum_{\zeta=0}^2
    % \left[
    \sum_{1 \leq |\gamma| \leq 2}\Ithree_{n,\gamma,0}(X)\frac{\partial^{\gamma} f}{\partial h^{\gamma}}(Y|h(X,\eta_n^*),\nu_n^*)
    \exp((\betasn)^{\top}X)
    % \right]X^{\zeta}
    +R_3(X,Y) 
\end{align*}
where 
\begin{align*}
% &\Ithree_{n,1,0}(X)=\exp(\tau^*_n)
%     \Big[\sum_{u=1}^{q}\{(-\Delta\etasn)^{(u)}\}\frac{\partial  h}{\partial \eta^{(u)}}(X,\etasn)
%     % \\&\hspace{2.8cm}
%     +\sum_{1\leq u,v\leq q}\frac{( -\Delta \etasn)^{(u)}( -\Delta\etasn)^{(v)}}{1+\mathbf{1}_{u=v}}\frac{\partial^2 h}{\partial \eta^{(u)}\partial \eta^{(v)}}(X,\etasn)
%     \Big],
% \\
&\Ithree_{n,\gamma,0}(X)=\exp(\tau^*_n)
    \Big[(-\Delta\etasn)^{(u)}\frac{\partial  h}{\partial \eta^{(u)}}(X,\etasn) 
    % \\&\hspace{2.8cm}
    +\sum_{1\leq u,v\leq q}\frac{( -\Delta \etasn)^{(u)}( -\Delta\etasn)^{(v)}}{1+\mathbf{1}_{u=v}}\frac{\partial^2 h}{\partial \eta^{(u)}\partial \eta^{(v)}}(X,\etasn)
    \Big], \text{ when } |\gamma| = 1,
\\
&\Ithree_{n,2,0}(X)=\exp(\tau^*_n)
    \Big[
    \sum_{1\leq u,v\leq q}\frac{( -\Delta \etasn)^{(u)}( -\Delta\etasn)^{(v)}}{1+\mathbf{1}_{u=v}}
    \frac{\partial h}{\partial \eta^{(u)}}(X,\etasn)
    \frac{\partial h}{\partial \eta^{(v)}}(X,\etasn)
    \Big]  \text{ when } |\gamma| = 2.
\end{align*}

% \begin{align*}
% \Ithree_{n,1,0}(X)
% &= \exp(\tau_n^*) 
% \sum_{i=1}^{K-1}
% \Big[
% ( -\Delta\eta_{ni}^*)^\top J_i^\ast  +  \dfrac12 ( -\Delta\eta_{ni}^*)^\top H_{i}^\ast ( -\Delta\eta_{ni}^*)
% \Big],\\
% \Ithree_{n,2,0}(X)
% &= \exp(\tau_n^*) \frac{1}{2}
% \sum_{i=1}^{K-1}\sum_{j=1}^{K-1}
% ( -\Delta\eta_{ni}^\ast)(-\Delta\eta_{nj}^\ast) J_i^*J_j^*
% .
% \end{align*}

Now we consider 
$\Ifour_n=\exp\left( (\beta_n^* + \beta_n)^{\top}X + \tau_n^* + \tau_n \right)
    \cdot[f(Y|\sigma(X,\eta_n),\nu_n)-f(Y|\sigma(X,\eta_n^*),\nu_n^*)]$,
which is equivalent to 
\begin{align*}
    \Ifour_{n}&=
    \sum_{1 \leq |\gamma| \leq 2}\Ifour_{n,\gamma,0}(X)\frac{\partial^{\gamma} f}{\partial h^{\gamma}}(Y|h(X,\eta_n^*),\nu_n^*)
    \exp((\betasn)^{\top}X)
    \exp((\betan)^{\top}X)
    +R_4(X,Y) 
\end{align*}
where
\begin{align*}
&\Ifour_{n,\gamma,0}(X)=\exp(\tau^*_n+\tau_n)
    \Big[\sum_{u=1}^{q}\{(\Delta\etan-\Delta\etasn)^{(u)}\}\frac{\partial  h}{\partial \eta^{(u)}}(X,\etasn)
    \\&\hspace{4.3cm}
    +\sum_{1\leq u,v\leq q}\frac{( \Delta\etan-\Delta \etasn)^{(u)}(\Delta\etan -\Delta\etasn)^{(v)}}{1+\mathbf{1}_{u=v}}\frac{\partial^2 h}{\partial \eta^{(u)}\partial \eta^{(v)}}(X,\etasn)
    \Big] \text{ where } |\gamma| = 1, 
\\
&\Ifour_{n,\gamma,0}(X)=\exp(\tau_n^*+\tau_n)
    \Big[
    \sum_{1\leq u,v\leq q}\frac{(\Delta\etan -\Delta \etasn)^{(u)}( \Delta\etan-\Delta\etasn)^{(v)}}{1+\mathbf{1}_{u=v}}
    \frac{\partial h}{\partial \eta^{(u)}}(X,\etasn)
    \frac{\partial h}{\partial \eta^{(v)}}(X,\etasn)
    \Big] \text{ where } |\gamma| = 2. 
\end{align*}

% \begin{align*}
% \Ifour_{n,1,0}(X)
% &= \exp(\tau_n^*+\tau_n) 
% \sum_{i=1}^{K-1}
% \Big[
% ( \Delta\eta_{ni}-\Delta\eta_{ni}^*)^\top J_i^\ast  +  \dfrac12 ( \Delta\eta_{ni}-\Delta\eta_{ni}^*)^\top H_{i}^\ast ( \Delta\eta_{ni}-\Delta\eta_{ni}^*)
% \Big],\\
% \Ifour_{n,2,0}(X)
% &= \exp(\tau_n^*+\tau_n) \frac{1}{2}
% \sum_{i=1}^{K-1}\sum_{j=1}^{K-1}
% ( \Delta\eta_{ni}-\Delta\eta_{ni}^\ast)(\Delta\eta_{nj}-\Delta\eta_{nj}^\ast) J_i^*J_j^*
% .
% \end{align*}

Then we could conclude that 
% \begin{align*}
%     \Ione_{n}&=
%     \sum_{\zeta=0}^2
%     \left[
%     \sum_{\gamma=0}^2\Ione_{n,\gamma,\zeta}(X)
%    \frac{\partial^{|\gamma|} f}{
%       \partial h_1^{|\gamma_{1}|} \cdots 
%       \partial h_{K-1}^{|\gamma_{K-1}|}
%    }(Y = s  |  X; \eta_n^*)
%     \exp((\betasn)^{\top}X)
%     \right]X^{\zeta}
%     +R_1(X,Y) 
% \end{align*}
\begin{align*}
    W_{n}&=
    \sum_{0 \leq |\gamma| \leq 2}
    \Bigg[
    \left(
    \Ione_{n,\gamma,0}(X)+
    \Itwo_{n,\gamma,0}(X)+
    \Ithree_{n,\gamma,0}(X)
    \right)
    +
    \sum_{1\leq |\zeta| \leq 2}
    \left(
    \Ione_{n,\gamma,\zeta}(X)+
    \Itwo_{n,\gamma,\zeta}(X)
    \right)
    X^{\zeta}
    +
    \Ifour_{n,\gamma,0}(X)\exp((\betan)^{\top}X)
    \Bigg]
    \\&
    \hspace{3cm}
    \cdot
   \frac{\partial^{|\gamma|} f}{
      \partial h_1^{|\gamma_{1}|} \cdots 
      \partial h_{K}^{|\gamma_{K}|}
   }(Y = s  |  X; \eta_n^*)
     \cdot
    \exp((\betasn)^{\top}X) .
\end{align*}
Therefore, we can view the quantity $W_n/\dtwo)$ as a linear combination of elements of 
the set $\mathcal{L}\cup\mathcal{K}$, and $\mathcal{L}=\cup_{\gamma=0}^2\cup_{\zeta=0}^2\mathcal{L}_{\gamma,\zeta}$ , $\mathcal{K}=\cup_{\gamma=1}^2\mathcal{K}_{\gamma}$, where
\begin{align*}
    \mathcal{L}_{0,1}&=\left\{
    Xf(Y = s  |  X; \eta_n^*)\exp((\betasn)^{\top}X) 
    \right\}
    \\
    \mathcal{L}_{0,2}&=\left\{
    XX^{\top}f(Y = s  |  X; \eta_n^*)\exp((\betasn)^{\top}X) 
    \right\}
    \\
    \mathcal{L}_{1,1}&=\left\{X
    \frac{\partial h_i}{\partial \eta_i^{(u)}}(X,\eta_{ni}^*)
    \frac{\partial f}{\partial  h_i}(Y = s  |  X; \eta_n^*)\exp((\betasn)^{\top}X) 
    :u\in[q],i\in[K]
    \right\}
    \\
    \mathcal{L}_{1,0}&=\left\{\frac{\partial  h_i}{\partial \eta_i^{(u)}}(\xetas)\frac{\partial f}{\partial  h_i}(Y = s  |  X; \eta_n^*) 
    \exp((\betasn)^{\top}X)
    :u\in[d],i\in[K]\right\}\\
    &\cup\left\{\frac{\partial^2  h_i}{\partial \eta_i^{(u)}\partial \eta_i^{(v)}}(\xetas)\frac{\partial f}{\partial  h_i}(Y = s  |  X; \eta_n^*)\exp((\betasn)^{\top}X) :u,v\in[d],i\in[K] \right\},\\
    \mathcal{L}_{2,0}
    % &=\left\{
    % \frac{\partial^2f}{\partial  h_i \partial  h_j}(Y = s  |  X; \eta_n^*)\exp((\betasn)^{\top}X)  \right\}\\
    &=\left\{
    \frac{\partial  h_i}{\partial \eta_i^{(u)} }(\xetas)
    \frac{\partial  h_j}{\partial \eta_j^{(v)} }(\xetas)
    \frac{\partial^2f}{\partial  h_i \partial  h_j}(Y = s  |  X; \eta_n^*)
    \exp((\betasn)^{\top}X)
    :u,v\in[q];i,j\in[K]  \right\},
\end{align*}
and
\begin{align*}
    \mathcal{K}_{1 }&=\left\{\frac{\partial  h_i}{\partial \eta_i^{(u)}}(\xetas)
    \exp((\betan)^{\top}X)
    \frac{\partial f}{\partial  h_i}(Y = s  |  X; \eta_n^*) 
    \exp((\betasn)^{\top}X)
    :u\in[d],i\in[K]\right\}\\
    &\cup\left\{\frac{\partial^2  h_i}{\partial \eta_i^{(u)}\partial \eta_i^{(v)}}(\xetas)
    \exp((\betan)^{\top}X)
    \frac{\partial f}{\partial  h_i}(Y = s  |  X; \eta_n^*)
    \exp((\betasn)^{\top}X) :u,v\in[d], i\in[K]\right\},\\
    \mathcal{K}_{2 }
    &=\left\{
    \frac{\partial  h_i}{\partial \eta_i^{(u)} }(\xetas)
    \frac{\partial  h_j}{\partial \eta_j^{(v)} }(\xetas)
    \exp((\betan)^{\top}X)
    \frac{\partial^2f}{\partial  h_i \partial  h_j}(Y = s  |  X; \eta_n^*)
    \exp((\betasn)^{\top}X)
    :u,v\in[q];i,j\in[K]\right\}.
\end{align*}

Assume, to the contrary, that all the coefficients of these terms vanish as $n\to\infty$. 
Looking at the coefficients of 
$ \frac{\partial h_i}{\partial \eta_i^{(u)}}(X,\etasn)
    X\frac{\partial f}{\partial  h_i}(Y = s  |  X; \eta_n^*)\exp((\betasn)^{\top}X) $,
we get for all $w\in[d], u\in[q], i\in[K]$
\begin{align}
\label{pfeq:d2term_beta_eta}
\exp(\tau_n)
{[(\betan-\betasn)^{(w)}][(\Delta\eta_{ni})^{(u)}]}
    /\done\to0,
\end{align}
% \begin{align*}
% \Ione_{n,1,1}(X)
% &= \exp(\tau_n) 
% (\beta_n-\beta_n^\ast)
% \sum_{i=1}^{K-1}
% (\Delta\eta_{ni}-\Delta\eta_{ni}^\ast) J_i^*,
% \\
% \Itwo_{n,1,1}(X)
% &= \exp(\tau_n) 
% (\beta_n-\beta_n^\ast)
% \sum_{i=1}^{K-1}
% ( -\Delta\eta_{ni}^\ast) J_i^*,
% \end{align*}
% Looking at the coefficients of 
% $X\frac{\partial^2 f}{\partial  h^2}(Y = s  |  X; \eta_n^*)\exp((\betasn)^{\top}X) $,
% we get for all $w\in[d]$
% \begin{align}
% \label{pfeq:d2term_beta_nu}
% \exp(\tau_n)
% {[(\betan-\betasn)^{(w)}](\Delta\nun)}
%     /\dtwo\to0,
% \end{align}
% % \begin{align*}
% % &\Ione_{n,2,1}(X)=\frac{\exp(\tau_n)}{2}
% %     \Big[
% %     \sum_{1\leq w\leq d, 1\leq u\leq q}
% %     {(\betan-\betasn)^{(w)}(\Delta\nun-\Delta\nusn)^{(u)}}
% %     \Big],
% %     \\
% % &\Itwo_{n,2,1}(X)=\frac{\exp(\tau_n)}{2}
% %     \Big[
% %     \sum_{1\leq w\leq d, 1\leq u\leq q}
% %     {(\betan-\betasn)^{(w)}(-\Delta\nusn)^{(u)}}
% %     \Big],
% % \end{align*}
Looking at the coefficients of 
$\dfrac{\partial^2  h_i}{\partial \eta_i^{(u)}\partial \eta_i^{(v)}}(\xetas)\dfrac{\partial f}{\partial  h_i}(Y = s  |  X; \eta_n^*)\exp((\betasn)^{\top}X) $ 
, 
we get for all $u,v\in[q], i\in[K]$,
\begin{align}
\label{pf:d2_coefficients_1}
[\exp(\tau_n)
{(\Delta \etani-\Delta \etasni)^{(u)}(\Delta \etani-\Delta\etasni)^{(v)}}
+[\exp(\tau^*_n)-\exp(\tau_n)]
( -\Delta \etasni)^{(u)}( -\Delta\etasni)^{(v)}]
\nonumber\\    /\done \to 0,
\end{align}
Looking at the coefficients of 
$\dfrac{\partial  h_i}{\partial \eta_i^{(u)}}(\xetas)\dfrac{\partial f}{\partial  h_i}(Y = s  |  X; \eta_n^*)\exp((\betasn)^{\top}X) $ 
, 
we get for all $u \in[q],i\in[K]$,
\begin{align}
\label{pf:d2_coefficients_2}
    [\exp(\tau_n)(\Delta\etani-\Delta\etasni)^{(u)}
+[\exp(\tau^*_n)-\exp(\tau_n)]
    (-\Delta\etasni)^{(u)}]
% \nonumber\\    
/\done \to 0,
\end{align}
% Looking at the coefficients of 
% $\dfrac{\partial^2 f}{\partial h^2}(Y = s  |  X; \eta_n^*)\exp((\betasn)^{\top}X) $ , we get 
% % for all $u,v\in[q]$,
% \begin{align}
% \label{pf:d2_coefficients_3}
% [\exp(\tau_n)
% (\Delta\nun-\Delta\nusn)
% +[\exp(\tau^*_n)-\exp(\tau_n)]
% (-\Delta\nusn)]
%     /\dtwo&\to 0,
% \end{align}
Looking at the coefficients of 
$\dfrac{\partial h_i}{\partial \eta_i^{(u)}}(X,\etasn)
    \dfrac{\partial h_j}{\partial \eta_j^{(v)}}(X,\etasn)
    \dfrac{\partial^2 f}{\partial h_i \partial h_j}(Y = s|X; \eta_n^*)\exp((\betasn)^{\top}X) $ , we get 
for all $u,v\in[q]$ and $i,j\in[K]$,
\begin{align}
\label{pf:d2_coefficients_4}
[\exp(\tau_n)
(\Delta \etani-\Delta \etasni)^{(u)}(\Delta \etanj-\Delta\etasnj)^{(v)}
+[\exp(\tau^*_n)-\exp(\tau_n)]
( -\Delta \etasni)^{(u)}( -\Delta\etasnj)^{(v)}] 
 \nonumber\\   /\done\to 0,
\end{align}

Looking at the coefficients of 
$
    \dfrac{\partial  h_i}{\partial \eta_i^{(u)}}(X,\etasn)
    \exp((\betan)^{\top}X)
    \dfrac{\partial f}{\partial h_i}(Y = s  |  X; \eta_n^*)\exp((\betasn)^{\top}X) $ , we get 
for all $u\in[q],i\in[K]$,
\begin{align}
\label{pf:d2_coefficients_7}
[\exp(\tau^*_n+\tau_n)
(\Delta\etani-\Delta\etasni)^{(u)}]
    /\done&\to 0,
\end{align}
Looking at the coefficients of 
$
    \dfrac{\partial^2 h_i}{\partial \eta_i^{(u)}\partial \eta_i^{(v)}}(X,\etasn)
    \exp((\betan)^{\top}X)
    \dfrac{\partial f}{\partial h_i}(Y = s  |  X; \eta_n^*)\exp((\betasn)^{\top}X) $ , we get 
for all $u,v\in[q]$ and $i\in[K]$,
\begin{align}
\label{pf:d2_coefficients_8}
[\exp(\tau^*_n+\tau_n)
( \Delta\etani-\Delta \etasni)^{(u)}(\Delta\etani -\Delta\etasni)^{(v)}]
    /\done&\to 0,
\end{align}
% Looking at the coefficients of 
% $
%     \exp((\betan)^{\top}X)
%     \dfrac{\partial^2 f}{\partial h^2}(Y = s  |  X; \eta_n^*)\exp((\betasn)^{\top}X) $ , we get 
% % for all $u,v\in[q]$,
% \begin{align}
% \label{pf:d2_coefficients_9}
% [\exp(\tau^*_n+\tau_n)
% (\Delta\nun-\Delta\nusn)]
%     /\dtwo&\to 0,
% \end{align}
Looking at the coefficients of 
$
\frac{\partial h_i}{\partial \eta_i^{(u)}}(X,\etasn)
    \frac{\partial h_j}{\partial \eta_j^{(v)}}(X,\etasn)
    \exp((\betan)^{\top}X)
    \dfrac{\partial^2 f}{\partial h_i \partial h_j}(Y = s  |  X; \eta_n^*)\exp((\betasn)^{\top}X) $ , we get 
for all $u,v\in[q]$ and $i,j\in[K]$,
\begin{align}
\label{pf:d2_coefficients_10}
[\exp(\tau^*_n+\tau_n)
(\Delta\etani -\Delta \etasni)^{(u)}( \Delta\etanj-\Delta\etasnj)^{(v)}]
    /\done&\to 0,
\end{align}
% Looking at the coefficients of 
% $
% \frac{\partial h}{\partial \eta^{(u)}}(X,\etasn)
%     \exp((\betan)^{\top}X)
%     \dfrac{\partial^3 f}{\partial h^3}(Y = s  |  X; \eta_n^*)\exp((\betasn)^{\top}X) $, we get 
% for all $u\in[q]$,
% \begin{align}
% \label{pf:d2_coefficients_11}
% \exp(\tau^*_n+\tau_n)
% (\Delta\etan-\Delta\etasn)^{(u)}
% (\Delta\nun-\Delta\nusn)
%     /\dtwo&\to 0,
% \end{align}
% Looking at the coefficients of 
% $
%     \exp((\betan)^{\top}X)
%     \dfrac{\partial^4 f}{\partial h^4}(Y = s  |  X; \eta_n^*)\exp((\betasn)^{\top}X) $, we get 
% for all $u\in[q]$,
% \begin{align}
% \label{pf:d2_coefficients_12}
% [\exp(\tau^*_n+\tau_n)
% (\Delta\nun-\Delta\nusn)^2]
%     /\dtwo&\to 0,
% \end{align}

Now, consider equation \eqref{pfeq:d2term_beta_eta} and
recall that all the gating parameters are in compact sets,
% and applying the Cauchy–Schwarz inequality followed by summation over coordinates, 
Using \(\|v\|_2 \le \|v\|_1\) and expanding the \(\ell_1\) norms,
\begin{align*}
\frac{e^{\tau_n}}{\done} \|\beta_n-\beta_n^\ast\|_2
\sum_{i=1}^{K}\|\Delta\eta_{ni}\|_2
& \le 
\frac{e^{\tau_n}}{\done} \|\beta_n-\beta_n^\ast\|_1
\sum_{i=1}^{K}\|\Delta\eta_{ni}\|_1
\\
&= 
\sum_{w=1}^{d}\sum_{i=1}^{K}\sum_{u=1}^{q}
\frac{e^{\tau_n}}{\done} 
\big|(\beta_n-\beta_n^\ast)^{(w)}\big| 
\big|(\Delta\eta_{ni})^{(u)}\big|,
\end{align*}
we got that
\begin{align}   
    \label{eq:nondistinguishable_independent_01}
    \exp(\tau_n)
\|\betan-\betasn\|
\sum_{i=1}^{K}\|\Delta\eta_{ni}\|
    /\done&\to 0.
\end{align}
% By this symmetry in the definition of $ W_n $, we similarly obtain
% \begin{align}   
%     \exp(\tau_n+\tau_n^*)
% \|\betan-\betasn\|
% \|(\Delta\etasn,\Delta\nusn )\|
%     /\dtwo&\to 0.
% \end{align}

{
While it is intuitive that a similar result should hold for $\|\Delta\etasni\|$, a slightly more delicate handle is required. Suppose that 
\begin{align*}
     \exp(\tau_n^*)
\|\betan-\betasn\|
\|\Delta\etasni\|
    /\done \not\to 0.
\end{align*}
By combining this assumption with equation \eqref{pfeq:d2term_beta_eta}, we deduce that there exists at least one coordinate $u$ such that $|(\Delta \etasni)^{(u)}/(\Delta \etani)^{(u)}| \to \infty$, which in turn implies that $(\Delta\etasni)/(\Delta \etasni - \Delta \etani)^{(u)} \to 1$. Therefore, multiplying equation \eqref{pf:d2_coefficients_7} with $(\Delta\etasni)/(\Delta \etasni - \Delta \etani)^{(u)} \to 1$ yields 
\begin{align*}
\exp(\tau^*_n)
(\Delta\etasni)^{(u)}
    /\done \to 0.
\end{align*}
Also, noting that $\|\betan - \betasn\|$ is bounded since the parameters lie to a compact set, we obtain 
\begin{equation*}
    \exp(\tau^*_n)\|\betan-\betasn\|
(\Delta\etasni)^{(u)}
    /\done \to 0,
\end{equation*}
which is a contradiction here. Thus, we have 
\begin{equation}
\
    \exp(\tau_n^*)
\|\betan-\betasn\|
% \sum_{i=1}^{K-1}
\|\Delta\etasni\|
    /\done\to 0.
\end{equation}
% Similarly, also by combining equation \eqref{pfeq:d2term_beta_nu} and \eqref{pf:d2_coefficients_9}, we have 
% \begin{equation}
%     \exp(\tau_n^*)
% \|\betan-\betasn\|
% \|\Delta\nusn\|
%     /\dtwo\to 0.
% \end{equation}
As a result, we have 
\begin{align}
\label{eq:d2proof_new}
    \exp(\tau_n^*)
\|\betan-\betasn\|
\sum_{i=1}^{K}
\|\Delta\etasni\|
    /\done \to 0.
\end{align}
}
% \begin{align}
% \label{pf:d2_coefficients_7}
% \frac{\exp(\tau^*_n+\tau_n)
% (\Delta\etan - \Delta\etasn)^{(u)}}{\dtwo}
% &\to 0, &&\forall u \in [q],
% \\
% \label{pf:d2_coefficients_8}
% \frac{\exp(\tau^*_n+\tau_n)
% (\Delta\etan - \Delta\etasn)^{(u)}(\Delta\etan - \Delta\etasn)^{(v)}}{\dtwo}
% &\to 0, &&\forall u, v \in [q],
% \\
% \label{pf:d2_coefficients_9}
% \frac{\exp(\tau^*_n+\tau_n)
% (\Delta\nun - \Delta\nusn)}{\dtwo}
% &\to 0,
% \\
% \label{pf:d2_coefficients_10}
% \frac{\exp(\tau^*_n+\tau_n)
% (\Delta\etan - \Delta\etasn)^{(u)}(\Delta\etan - \Delta\etasn)^{(v)}}{\dtwo}
% &\to 0, &&\forall u, v \in [q],
% \\
% \label{pf:d2_coefficients_11}
% \frac{\exp(\tau^*_n+\tau_n)
% (\Delta\etan - \Delta\etasn)^{(u)}(\Delta\nun - \Delta\nusn)}{\dtwo}
% &\to 0, &&\forall u \in [q],
% \\
% \label{pf:d2_coefficients_12}
% \frac{\exp(\tau^*_n+\tau_n)
% (\Delta\nun - \Delta\nusn)^2}{\dtwo}
% &\to 0.
% \end{align}
In a similar manner, by considering equations \eqref{pf:d2_coefficients_7} 
% through \eqref{pf:d2_coefficients_12}, 
% and applying analogous reasoning as above, 
we obtain that
\begin{align}
\label{eq:nondistinguishable_independent_02}
    \exp(\tau_n+\tau_n^*)\cdot
    \Vert \Delta\etani-\Delta\eta^*_{ni}\Vert^2
    /\done&\to 0.
\end{align}

Let $u=v$ in the first equation in equation \eqref{pf:d2_coefficients_1}, we achieve that for all $u\in[q]$,
\begin{align}   
    \label{eq:nondistinguishable_independent_4}
    [\exp(\tau_n)
[(\Delta \etani-\Delta \etasni)^{(u)}]^2
+
[\exp(\tau^*_n)
-
\exp(\tau_n)]
[( \Delta \etasni)^{(u)}]^2]
    /\done&\to 0,
\end{align}
which implies that
\begin{align}
    \label{eq:nondistinguishable_independent_5}
[    \exp(\tau_n)
\|(\Delta \etani-\Delta \etasni)\|^2
+
(\exp(\tau^*_n)-\exp(\tau_n))
\| \Delta \etasni\|^2]
    /\done&\to 0
    .
\end{align}

We also have each term inside equation \eqref{eq:nondistinguishable_independent_5} is non-negative, thus  
\begin{align}
\label{eq:nondistinguishable_independent_11}
    (\exp(\tau^*_n)-\exp(\tau_n))\|\Delta\etasni\|^2/\done &\to 0,
    \nonumber\\
    \exp(\tau_n)\|\Delta\etani-\Delta\etasni\|^2/\done &\to 0.
\end{align}

Applying the AM-GM inequality, we have for all $u,v\in[q]$,
\begin{align}
\label{eq:nondistinguishable_independent_12}
    \dfrac{(\exp(\tau^*_n)-\exp(\tau_n))(\Delta\etasni)^{(u)}(\Delta\etasni)^{(v)}}{\done}\to 0,~ \dfrac{\exp(\tau_n)(\Delta\etani-\Delta\etasni)^{(u)}(\Delta\etani-\Delta\etasni)^{(v)}}{\done} &\to 0,
    % ~u,v\in[d],  
    % \\ \label{eq:nondistinguishable_independent_13_hello}   \dfrac{(\lambdasn-\lambdan)(\Delta\asn)^{(u)}(\Delta\bsn)^{ }}{\dtwo}\to 0,~ \dfrac{\lambdan(\Delta\an-\Delta\asn)^{(u)}(\Delta\bn-\Delta\bsn)^{ }}{\dtwo} &\to 0.
\end{align}

Next, by considering the coefficients of $\dfrac{\partial h_i}{\partial \eta_i^{(u)}}(X,\etasn)\dfrac{\partial f}{\partial h_i}(Y=s|X;\eta_n^*) \exp((\betasn)^{\top}X)$,
 we have
\begin{align}
    \label{eq:nondistinguishable_independent_1}
    [\exp(\tau_n)(\Delta\etani)^{(u)}-\exp(\tau_n^*)(\Delta\etasni)^{(u)}]/\done&\to 0,\quad u\in[q].
    % \\
    % \label{eq:nondistinguishable_independent_2}
    % [\exp(\tau_n)(\Delta\nun)^{ }-\exp(\tau_n^*)(\Delta\nusn)^{ }]/\dtwo&\to 0.\quad 
\end{align}
Noting that for $u,v\in[q]$,
\begin{align*}
    &\exp(\tau^*_n)(\detasni)^{(u)}(\detani-\detasni)^{(v)} 
    \\&= (\exp(\tau_n)(\detani)^{(v)} - \exp(\tau^*_n)(\detasni)^{(v)})(\detasni)^{(u)}+(\exp(\tau^*_n)-\exp(\tau_n))(\detani)^{(v)}(\detasni)^{(u)},\\
    &\exp(\tau_n)(\detani)^{(u)}(\detani-\detasni)^{(v)} 
    \\&= \exp(\tau^*_n)(\detasni)^{(u)}(\detani-\detasni)^{(v)} - (\exp(\tau_n)(\detani)^{(u)} - \exp(\tau^*_n)(\detasn)^{(u)})(\detani-\detasni)^{(v)}.
\end{align*}  
Thus, from equation \eqref{eq:nondistinguishable_independent_12} and equation \eqref{eq:nondistinguishable_independent_1}, we achieve that for $u,v\in[q]$,
\begin{align*}
    \exp(\tau^*_n)(\detasni)^{(u)}(\detani-\detasni)^{(v)}/\done &\to 0,
    \\ 
    \exp(\tau_n)(\detani)^{(u)}(\detani-\detasni)^{(v)}/\done &\to 0. 
\end{align*}

% Noting that for all $u\in[d]$,
% \begin{align*}
%     \lambdasn(\dasn)^{(u)}(\dbn-\dbsn)^{ } &= (\lambdan\dbn - \lambdasn\dbsn)(\dasn)^{(u)}+(\lambdasn-\lambdan)(\dbn)(\dasn)^{(u)},\\
%     \lambdan(\dan)^{(u)}(\dbn-\dbsn)^{ } &= \lambdasn(\dasn)^{(u)}(\dbn-\dbsn)^{ } - \left(\lambdan(\dan)^{(u)} - \lambdasn(\dasn)^{(u)}\right)(\dbn-\dbsn)^{ }.
% \end{align*}
% Thus, from 
% % \eqref{eqn:nondistinguishable_independent_14}, 
% equation \eqref{eq:nondistinguishable_independent_1} and equation \eqref{eq:nondistinguishable_independent_1.2}, we have for all $u\in[d]$,
% \begin{align*}
%     \lambdasn(\dasn)^{(u)}(\dbn-\dbsn) /\dtwo &\to 0\\ \lambdan(\dan)^{(u)}(\dbn-\dbsn) /\dtwo &\to 0. 
% \end{align*}

By using the same arguments we will derive 
\begin{align}
    % \lambdan\|\Delta\an\|.\|\Delta\an-\Delta\asn\|/\dtwo\to0,\\
    % \lambdasn\|\Delta\asn\|.\|\Delta\an-\Delta\asn\|/\dtwo\to0,\\
    \label{eq:nondistinguishable_independent_11.3}
    \exp(\tau_n)\|\Delta\etani\|.\|\Delta\etani-\Delta\etasni\|/\done\to0,\\
    \label{eq:nondistinguishable_independent_12.3}
    \exp(\tau^*_n)\|\Delta\etasni\|.\|\Delta\etani-\Delta\etasni\|/\done\to0,
    % \\
    % \label{eq:nondistinguishable_independent_11.4}
    % \exp(\tau_n)\|\Delta\bn\|.\|\Delta\bn-\Delta\bsn\|/\dtwo\to0,\\
    % \label{eq:nondistinguishable_independent_12.4}
    % \exp(\tau^*_n)\|\Delta\bsn\|.\|\Delta\bn-\Delta\bsn\|/\dtwo\to0.
\end{align}
% By using the same arguments to derive equation~\eqref{eq:nondistinguishable_independent_4}, equation~\eqref{eq:nondistinguishable_independent_11} and equation~\eqref{eq:nondistinguishable_independent_12}, we can point out that
% \begin{align}
%     [(\exp(\tau^*_n)-\exp(\tau_n))\|\Delta\nusn\|^2+\exp(\tau_n)\|\Delta\nun-\Delta\nusn\|^2]/\dtwo&\to 0,\nonumber\\
%     \exp(\tau_n)\|\Delta\nun\|.\|\Delta\nun-\Delta\nusn\|/\dtwo&\to 0,\nonumber\\
%     \nonumber
%     \exp(\tau^*_n)\|\Delta\nusn\|.\|\Delta\nun-\Delta\nusn\|/\dtwo&\to 0,\\
%     \exp(\tau_n)\|\Delta\etan\|.\|\Delta\nun-\Delta\nusn\|/\dtwo&\to 0,\nonumber\\
%     \exp(\tau^*_n)\|\Delta\etasn\|.\|\Delta\nun-\Delta\nusn\|/\dtwo&\to 0. \label{eq:nondistinguishable_independent_13}
%     % \\
%     % \exp(\tau_n)\|\Delta\bn\|.\|\Delta\nun-\Delta\nusn\|/\dtwo&\to 0,\nonumber\\
%     % \exp(\tau^*_n)\|\Delta\bsn\|.\|\Delta\nun-\Delta\nusn\|/\dtwo&\to 0\nonumber.
% \end{align}
% Now following the same arguments as above and 
% In the same way,
% consider equations \eqref{pf:d2_coefficients_7} to \eqref{pf:d2_coefficients_12}, we will have that 
% In a similar manner, by considering equations \eqref{pf:d2_coefficients_7} through \eqref{pf:d2_coefficients_12}, and applying analogous reasoning as above, we obtain that
% \begin{align}
% \label{eq:nondistinguishable_independent_14}
%     \exp(\tau_n+\tau_n^*)\cdot
%     \Vert (\Delta\etan,\Delta\nun)-(\Delta\eta^*_n,\Delta\nu^*_n)\Vert^2
%     /\dtwo&\to 0.
% \end{align}

Collecting results in equation~\eqref{eq:nondistinguishable_independent_01},
\eqref{eq:d2proof_new}
and \eqref{eq:nondistinguishable_independent_02}, and equations~\eqref{eq:nondistinguishable_independent_11} to
\eqref{eq:nondistinguishable_independent_12.3},
we obtain that
\begin{align*}
    1=\done/\done\to 0,
\end{align*}
which is a contradiction.

Hence, the coefficients appearing in the expansion of  ${W_n}/\done$ cannot all vanish as $n\to\infty$. Let $m_n$ be the largest (in the absolute values) among these coefficients. The argument above implies that  $1/m_n$ does not diverge to infinity. We now define
\begin{align}
\exp(\tau_n)
{[(\betan-\betasn)^{(w)}][(\Delta\etani)^{(u)}]}
    /m_n&\to \alpha_{11,wu0i},
%     \nonumber
% \\
% \exp(\tau_n)
% {[(\betan-\betasn)^{(w)}](\Delta\nun)}
%     /m_n\to \alpha_{21,w00},
\nonumber\\
    [\exp(\tau_n)(\Delta\etani-\Delta\etasni)^{(u)}
+[\exp(\tau^*_n)-\exp(\tau_n)]
    (-\Delta\etasni)^{(u)}]
    /m_n &\to \alpha_{10,0u0i},
\nonumber\\
[\exp(\tau_n)
{(\Delta \etani-\Delta \etasni)^{(u)}(\Delta \etani-\Delta\etasni)^{(v)}}
+[\exp(\tau^*_n)-\exp(\tau_n)]
( -\Delta \etasni)^{(u)}( -\Delta\etasni)^{(v)}]
    /m_n&\to \beta_{10,0uvi},
% \nonumber\\
% [\exp(\tau_n)
% (\Delta\nun-\Delta\nusn)
% +[\exp(\tau^*_n)-\exp(\tau_n)]
% (-\Delta\nusn)]
    % /m_n&\to \alpha_{20,000},
\nonumber\\
[\exp(\tau_n)
(\Delta \etani-\Delta \etasni)^{(u)}(\Delta \etani-\Delta\etasni)^{(v)}
+[\exp(\tau^*_n)-\exp(\tau_n)]
( -\Delta \etasni)^{(u)}( -\Delta\etasni)^{(v)}]
    /m_n&\to \beta_{20,0uvi},
% \nonumber\\
% [\exp(\tau_n)
% (\Delta\etan-\Delta\etasn)^{(u)}(\Delta\nun-\Delta\nusn)
% +[\exp(\tau^*_n)-\exp(\tau_n)]
% (-\Delta\etasn)^{(u)}(-\Delta\nusn)]
%     /m_n&\to \beta_{30,0u0},
% \nonumber\\
% [\exp(\tau_n)
% (\Delta\nun-\Delta\nusn)^2
% +[\exp(\tau^*_n)-\exp(\tau_n)]
% (-\Delta\nusn)^2]
    % /m_n&\to \beta_{40,000},
\nonumber\\
\exp(\tau^*_n+\tau_n)
(\Delta\etani-\Delta\etasni)^{(u)}
    /m_n&\to \rho_{1,u0i},
\nonumber\\
\exp(\tau^*_n+\tau_n)
( \Delta\etani-\Delta \etasni)^{(u)}(\Delta\etani -\Delta\etasni)^{(v)}
    /m_n&\to \pi_{1,uvi},
% \nonumber\\
% \exp(\tau^*_n+\tau_n)
% (\Delta\nun-\Delta\nusn)
%     /m_n&\to \rho_{2,00},
\nonumber\\
\exp(\tau^*_n+\tau_n)
(\Delta\etani -\Delta \etasni)^{(u)}( \Delta\etani-\Delta\etasni)^{(v)}
    /m_n&\to \pi_{2,uvi},
% \nonumber\\
% \exp(\tau^*_n+\tau_n)
% (\Delta\etan-\Delta\etasn)^{(u)}
% (\Delta\nun-\Delta\nusn)
%     /m_n&\to \pi_{3,u0},
% \nonumber\\
% \exp(\tau^*_n+\tau_n)
% (\Delta\nun-\Delta\nusn)^2
    % /m_n&\to \pi_{4,00},
\label{d2_all_coefficients}
\end{align}
when $n\to\infty$ for all $w\in[d], u,v\in[q]$. Note that at least one among $\alpha_{\gamma\zeta, wuv},\beta_{\gamma\zeta, wuv}$ and $\rho_{\gamma, uv},\pi_{\gamma, uv}$ where $\gamma \in [2], \zeta\in \{ 0,1\}$ must be different from zero. By applying the Fatou's lemma, we get
\begin{align*}
    0=\lim_{n\to\infty}\frac{1}{m_n}\frac{2\bbE_X[d_V(\plbgn(\cdot|X),\plbgs(\cdot|X))]}{\done}\geq \int\liminf_{n\to\infty}\frac{1}{m_n}\frac{|p_{ \Gn}(Y=s|X)-p_{ \Gsn}(Y=s|X)|}{\done}d(X,Y).
\end{align*}
On the other hand,
\begin{align}
\nonumber
    &
    \frac{1}{m_n}\frac{p_{ \Gn}(Y=s|X)-p_{ \Gsn}(Y=s|X)}{\done}
    \\
    \label{eqn:convergence_formular}
    \to &  
    \sum_{0 \leq |\gamma| \leq 2}
    \Bigg[
    \sum_{0 \leq |\zeta| \leq 1}
    E_{\gamma\zeta}(X)
    X^{\zeta}
    +
    K_{\gamma}(X)\exp(\beta^{\top}X)
    \Bigg]
    % \cdot
     \frac{\partial^{\gamma} f}{\partial h^{\gamma}}(Y=s|X)
     \cdot
    \exp(\beta^{\top}X),
\end{align}

where 
\begin{align*}
    &E_{\gamma\zeta}(X)=
     \sum_{1\leq u\leq q}
    \alpha_{11,wu0i}\frac{\partial h_i}{\partial \eta_i^{(u)}}(X,\eta_0)\ \text{where } \gamma = (0\ldots \underbrace{1}_{i}\ldots 0),\ \zeta = (0\ldots \underbrace{1}_{w}\ldots 0), 1\leq i\leq q,\ 1\leq w\leq d,\\
    % &E_{21}(X)=\frac{1}{2}\sum_{1\leq w\leq d}\alpha_{21,w00}\\
&E_{\gamma 0}(X)=
    \sum_{u=1}^{q}\alpha_{10,0u0i}
    \frac{\partial  h_i}{\partial \eta_i^{(u)}}(X,\eta_0)
    +
    \sum_{1\leq u,v\leq q}
    \frac{\beta_{10,0uvi}}{1+\mathbf{1}_{u=v}}\frac{\partial^2 h_i}{\partial \eta_i^{(u)}\partial \eta_i^{(v)}}(X,\eta_0)\ \text{where } \gamma = (0\ldots \underbrace{1}_{i}\ldots 0),
\\
&E_{\gamma 0}(X)=
    % \frac{1}{2}
    % \alpha_{20,000}
    % +
    \sum_{1\leq u,v\leq q}
    \frac{\beta_{20,0uvij}}{1+\mathbf{1}_{u=v}}
    \frac{\partial h_i}{\partial \eta_i^{(u)}}(X,\eta_0)
    \frac{\partial h_j}{\partial \eta_j^{(v)}}(X,\eta_0)\  \text{where } \gamma = (0\ldots \underbrace{1}_{i}\ldots 0) +(0\ldots \underbrace{1}_{j}\ldots 0),
% \\
% &E_{30}(X)=\frac{1}{2}
%     \sum_{u=1}^{q}
%     \beta_{30,0u0}
%     \frac{\partial h}{\partial \eta^{(u)}}(X,\eta_0) ,
% \\
% &E_{40}(X)=\frac{1}{8}
%     \beta_{40,000}.
\end{align*}
and
\begin{align*}
&K_{\gamma}(X)=
\sum_{u=1}^{q}\rho_{1,u0i}\frac{\partial  h_i}{\partial \eta_i^{(u)}}(X,\eta_0)
    +
    \sum_{i=1}^{K-1}
    \sum_{1\leq u,v\leq q}
    \frac{\pi_{1,uvi}}{1+\mathbf{1}_{u=v}}\frac{\partial^2 h_i}{\partial \eta_i^{(u)}\partial \eta_i^{(v)}}(X,\eta_0)\ \text{where } \gamma = (0\ldots \underbrace{1}_{i}\ldots 0),
\\
&K_{\gamma}(X)=
\sum_{1\leq u,v\leq q}\frac{\pi_{2,uvi}}{1+\mathbf{1}_{u=v}}
    \frac{\partial h_i}{\partial \eta_i^{(u)}}(X,\eta_0)
    \frac{\partial h_j}{\partial \eta_j^{(v)}}(X,\eta_0) \  \text{where } \gamma = (0\ldots \underbrace{1}_{i}\ldots 0) +(0\ldots \underbrace{1}_{j}\ldots 0).
% \\
% &K_{3}(X)=\frac{1}{2}
%     \sum_{u=1}^{q}
%     \pi_{3,u0}
%     \frac{\partial h}{\partial \eta^{(u)}}(X,\eta_0) ,
% \\
% &K_{4}(X)=\frac{1}{8}
%     \pi_{4,00}.
\end{align*}

Using Lemma \ref{lemma:independent_second_theorem}, we achieve that all the coefficients are equal to 0. 

% \begin{lemma}[Linear Independence of Gaussian Derivatives]
% \label{lemma:gaussian_derivatives_independent}
% Let $f(y  |  \mu, \sigma^2)$ denote the univariate Gaussian density with mean $\mu$ and variance $\sigma^2 > 0$:
% \[
% f(y  |  \mu, \sigma^2) = \frac{1}{\sqrt{2\pi\sigma^2}} \exp\left( -\frac{(y - \mu)^2}{2\sigma^2} \right).
% \]
% Then, the set of functions
% \[
% \left\{ \frac{\partial^\tau}{\partial \mu^\tau} f(y  |  \mu, \sigma^2) : \tau = 0, 1, 2, \ldots, T \right\}
% \]
% is linearly independent over $\mathbb{R}$ for any fixed $\sigma^2 > 0$ and any finite $T \in \mathbb{N}$.

% Moreover, in the multivariate case where $f(y  |  \mu, \Sigma)$ is the Gaussian density on $\mathbb{R}^d$ with mean $\mu \in \mathbb{R}^d$ and positive-definite covariance matrix $\Sigma \in \mathbb{R}^{d \times d}$, the collection of partial derivatives
% \[
% \left\{ \frac{\partial^{|\alpha|}}{\partial \mu^\alpha} f(y  |  \mu, \Sigma) : \alpha \in \mathbb{N}^d,   |\alpha| \leq T \right\}
% \]
% is linearly independent over $\mathbb{R}$ for any fixed $\Sigma$ and finite $T$.
% \end{lemma}

% It follows from the result $E_2(X)=0$ and 
% $\gamma_{\tau u}=0$ that 
% $\alpha_{30}=\beta_{\tau uv}=0$ for all $u,v$. Next, $E_1(X)=0$ implies that 
% $\alpha_{\tau u}=0$ for all $u$. 
This contradicts the fact that not all 
coefficients
% $\alpha_{\tau u}$, $\beta_{\tau uv}$ and $\gamma_{uv}$ 
vanish. Thus, we obtain the conclusion for this case.

{
\subsubsection*{Case 2: }
In this setting, we assume that $\eta_n$ and $\eta_{s,n}$ converge to a common limit that is distinct from $\eta_0$.

The formulation of the metric $D_2$ given in the proof~\ref{proof:not_equal} implies clearly that $D_1 \lesssim D_2$. 
Therefore, from our hypothesis, it is obvious that
$W_n(X,Y)
/\dtwo\to 0$ as $n\to\infty$. Since $\etan$ and $\etasn$ converge to the same limit $\eta^*\neq\eta_0$, we have $f_0 = f(y=s|x,\eta_0)$ and $f(y=s|x,\eta^*)$ satisfying $f_0$, $f$, and the derivatives of $f$ independent as in Lemma \ref{prop:independence_case_1}. We may therefore follow arguments analogous to those used in Theorem \ref{thm:not_equal} (see Appendix \ref{proof:not_equal}) to arrive at a contradiction.

\subsubsection*{Case 3: } Finally, we examine the case when $G_n$ or $G_n^*$ converges to $G_0$, whereas the other converges some $G' \neq G_0$. Without loss of generality, assume that $G_n \to G'$ and $G_n^* \to G_0$. Taking limits in 
    $\bbE_X[V(p_{ \Gn}(\cdot|X),p_{ \Gsn}(\cdot|X))]/D_1(G_n,G_n^*) \to 0,$
and noting that $ D_1(G_n,G_n^*) \to D_1(G,G_*) \neq 0$ while 
$\bbE_X[V(p_{\Gn}(\cdot|X),p_{\Gsn}(\cdot|X))] \to \bbE_X[V(p_{G}(\cdot|X),p_{G_{*}}(\cdot|X))]$,
it follows that $$\bbE_X[V(p_{G}(\cdot|X),p_{G_{*}}(\cdot|X))]= 0, \text{ or equivalently, } p_{G} = p_{G_{*}}, \text{ a.s. }$$
By the identifiability assumption, we may therefore conclude that 
% \begin{align*}
%     &f(Y|h(X,\eta_0),\nu_0) = \dfrac{1}{1+\exp(\beta^\top X + \tau^*)}f(Y|h(X,\eta_0),\nu_0) + \dfrac{\exp(\beta^\top X + \tau^*)}{1+\exp(\beta^\top X + \tau^*)}f(Y|h(X,\eta),\nu)\\
%     \Rightarrow&\dfrac{\exp(\beta^\top X + \tau^*)}{1+\exp(\beta^\top X + \tau^*)}f(Y|h(X,\eta_0),\nu_0) = \dfrac{\exp(\beta^\top X + \tau^*)}{1+\exp(\beta^\top X + \tau^*)}f(Y|h(X,\eta),\nu)\\
%     \Rightarrow& f(Y|h(X,\eta_0),\nu_0) = f(Y|h(X,\eta),\nu) \text{ (as $\exp(\beta^\top X + \tau^*) \neq 0$)}
% \end{align*}
% This equation means that 
$G' = G_0$, which is a contradiction. 
}
\end{proof}

\begin{lemma}
\label{lemma:quadratic_independency}
    Let $q$, $q_r$, $q_{ut}$ ($1\leq r,u,t\leq K$) such that 
    \begin{enumerate}
        \item $q_{ut} = q_{tu}$ for all $1\leq u,t\leq K$. 
        \item The following equality holds for all $(p_1,\ldots,p_K)$ lies in the $K$-simplex $\Delta_K$ (i.e. $0\leq p_i \leq 1$, and $\sum_{i=1}^K p_i = 1$): 
    \begin{equation}
    \label{eqn:dependency}
        q\cdot p_s + \sum_{r=1}^{K} q_rp_s(\delta_{sr} - p_r) + \sum_{u,t=1}^{K} q_{ut}p_{s} \left[(\delta_{su} - p_u)(\delta_{st} - p_t) - p_u(\delta_{ut} - p_t)\right] = 0.
    \end{equation}
    \end{enumerate}
    Then, it is necessary that 
    \begin{enumerate}
        \item $q = 0$. 
        \item $q_r = q_s$ for all $1\leq r,s \leq K$. 
        \item $2q_{uv} = 2q_{vu} = q_{uu} + q_{vv}$ for all $1\leq u,v \leq K$. 
    \end{enumerate}
\end{lemma}

\begin{proof}
To prove that $q = 0$, we sum up the equation \eqref{eqn:dependency} for $1\leq s\leq K$, and we receive 
\begin{align*}
    0 &= q\sum_{s=1}^Kq_s + \sum_{s=1}^K\sum_{r=1}^{K} q_rp_s(\delta_{sr} - p_r) + \sum_{s=1}^K\sum_{u,t=1}^{K} q_{ut}p_{s} \left[(\delta_{su} - p_u)(\delta_{st} - p_t) - p_u(\delta_{ut} - p_t)\right] \\
    &= q + \sum_{r=1}^{K}q_rp_r - \sum_{s=1}^{K}q_rp_r\left(\sum_{s=1}^Kp_s\right) + 2\left(\sum_{s=1}^Kp_s\right)\left(\sum_{u,t=1}^{K}q_{ut}p_up_t\right) - \sum_{s=1}^K\sum_{u,t=1}^{K}\delta_{su}q_{ut}p_sp_t\\
    &- \sum_{s=1}^K\sum_{u,t=1}^{K}\delta_{st}q_{ut}p_up_s - \sum_{s=1}^K\sum_{u,t=1}^{K}\delta_{su}q_{ut}\delta_{st}p_s -\sum_{s=1}^K\sum_{u,t=1}^{K}q_{ut}p_sp_u\delta_{ut}\\
    &= q +\sum_{r=1}^{K}q_rp_r + 2 \sum_{u,t=1}^{K} q_{ut}p_up_t - 2\sum_{u,t=1}^{K}q_{ut}p_up_t - \sum_{r=1}^{K}q_rp_r = q.
\end{align*}
As a result, we achieve that $q = 0$. By putting $q = 0$ into equation \eqref{eqn:dependency}, we have 
\begin{equation}
    \label{eqn:dependency_shorter_version}
    \sum_{r=1}^{K} q_r(\delta_{sr}-p_r) + \sum_{u,t=1}^{K}  q_{ut}\left((\delta_{su} - p_u)(\delta_{st} - p_t) - p_u(\delta_{ut} - p_t)\right) = 0.
\end{equation}
Equation \eqref{eqn:dependency_shorter_version} implies that 
\begin{align*}
    0 = (q_s+q_{ss}) -\sum_{r=1}^{K}(q_r+q_{rs}+q_{sr}-q_{rr})p_r + 2\sum_{u,t=1}^{K}q_{ut}p_up_t. 
\end{align*}

As a result, we have for all $s$ and $t$, 
\begin{equation*}
    (q_s + q_{ss}) - \sum_{r=1}^K(q_r+q_{rs} + q_{sr} -q_{rr})p_r = (q_t + q_{tt}) - \sum_{r=1}^K(q_r+q_{rt} + q_{tr} -q_{rr})p_r \Leftrightarrow (q_s+q_{ss} - q_t - q_{tt}) - \sum_{r=1}^K (2q_{rs} - 2q_{rt})p_r =0,
\end{equation*}
which holds for all $(p_1,\ldots,p_K) \in \Delta_K$. This only happens when 
\begin{equation}
\label{eqn:s_t_r_relation}
    q_s + q_{ss} - q_{t} - q_{tt} = 2q_{rs} - 2q_{rt}, \ \forall s,t,r \in [K].  
\end{equation}
By letting $r = s$ in \eqref{eqn:s_t_r_relation}, we have 
$q_s -q_t + q_{ss} - q_{tt} = 2q_{ss} - 2q_{st}$, or $q_s-q_t = q_{ss} + q_{tt} - 2q_{st}$. Similarly, we have $q_t-q_s = q_{ss} + q_{tt} - 2q_{st}$, thus $q_t-q_s = q_{ss} + q_{tt} - 2q_{st} = 0$. It follows that $q_r = q_s$ for all $1\leq r,s\leq K$ and $2q_{uv} = 2q_{vu}=q_{uv} + q_{vv}$. Conversely, it is straightforward to show that when $q,q_r,q_{ut}$ satisfies that $q = 0$, $q_r = q_s$ for all $1\leq r,s\leq K$, and $2q_{uv} = 2q_{vu} = q_{uu} + q_{vu}$, then \eqref{eqn:dependency} holds. This completes our proof. 
% As a result, we have that all the coefficients in zero, first, and second order with respect to $\{p_r,1\leq r\leq K\}$ are zeros. Let us consider the coefficient with respect to each order.

% $\bullet$ Second order: We have $q_{ut}+q_{tu} = 0$, thus from the hypothesis that $q_{ut} = q_{tu}$, we achieve that $q_{ut} = q_{tu} = 0$, $1\leq u,t\leq K$. 

% $\bullet$ First order: We have $q_r + q_{rs} +q_{sr} - q_{rr} = 0$, given that $q_{rs} = q_{sr} = q_{rr} = 0$, we have $q_r = 0$, $1\leq r\leq K$. 

% Finally, we have all the coefficients are equal to 0: 
% \begin{equation*}
%         q = 0, \ q_{r}  = 0, \ q_{uv} = 0, \quad 1\leq r,u,v \leq K-1. 
% \end{equation*}
\end{proof}
\begin{lemma}
\label{lemma:independent_second_theorem}
The collection $\mathcal{L} \cup \mathcal{K}$ is linearly independent, i.e. equation  \eqref{eqn:convergence_formular} implies that all the coefficients are equal to zero. 
\end{lemma}

\begin{proof}
    Using Lemma \ref{lemma:quadratic_independency}, we achieve that  
    % \begin{align*}
    %     E_{11}(X) &= \sum_{i=1}^{K-1} \sum_{1\leq w \leq d,1\leq u\leq q}\alpha_{11,wu0i} \dfrac{\partial h_i}{\partial \eta^{(u)}_i}(X,\eta_0) = 0,\\
    %     E_{10}(X) &= \sum_{i=1}^{K-1}\sum_{u=1}^{q} \alpha_{10,0u0i}\dfrac{\partial h_i}{\partial \eta_{i}^{(u)}}(X,\eta_0) + \sum_{i=1}^{K-1}\sum_{1\leq u,v\leq q} \dfrac{\beta_{10,0uvi}}{1+\mathbf{1}_{u=v}}\dfrac{\partial^2h_i}{\partial \eta_i^{(u)}\partial\eta_j^{(v)}}(X,\eta_0) = 0,\\
    %     E_{20}(X) &= \sum_{i=1}^{K-1}\sum_{j=1}^{K-1} \sum_{1\leq u,v\leq q}\dfrac{\beta_{20,0uvij}}{1+\mathbf{1}_{u=v}}\dfrac{\partial h_i}{\partial \eta_i^{(u)}} (X,\eta_0)\dfrac{\partial h_j}{\partial \eta_j^{(v)}}(X,\eta_0) = 0,\\
    %     K_1(X) &= \sum_{i=1}^{K-1}\sum_{u=1}^{q} \varrho_{1,u0i}\dfrac{\partial h_i}{\partial \eta_i^{(u)}}(X,\eta_0) + \sum_{i=1}^{K-1} \sum_{1\leq u,v\leq q} \dfrac{\pi_{1,uvi}}{1+\mathbf{1}_{u=v}}\dfrac{\partial^2 h_i}{\partial \eta_i^{(u)}\partial \eta_i^{(v)}}(X,\eta_0) = 0,\\  
    %     K_2(X) &= \sum_{i=1}^{K-1} \sum_{j=1}^{K-1} \sum_{1\leq u,v\leq q} \dfrac{\pi_{2,uvi}}{1+\mathbf{1}_{u=v}}\dfrac{\partial h_i}{\partial \eta_i^{(u)}}\dfrac{\partial h_j}{\partial \eta_i^{(v)}} = 0. 
    % \end{align*}
    \begin{align*}
        \sum_{0\leq|\zeta| \leq 1 }E_{\gamma\zeta}(X)X^{\zeta} + K_{\gamma}(X)\exp(\beta^\top X) = \sum_{0\leq|\zeta| \leq 1 }E_{\gamma'\zeta}(X)X^{\zeta} + K_{\gamma'}(X)\exp(\beta^\top X) 
    \end{align*}
    $\text{ for } \gamma' \neq \gamma  \text{ and }|\gamma| = |\gamma'| = 1$ and $2H_{ij} = H_{ii} + H_{jj}$, where
    \begin{align*}
        H_{uv} = \sum_{0\leq |\zeta| \leq 1} E_{\gamma \zeta}(X) X^{\zeta} + K_{\gamma}(X)\exp(\beta^\top X) \text{ where } \gamma = (0\ldots \underbrace{1}_{u-\text{th}}\ldots 0) +(0\ldots \underbrace{1}_{v-\text{th}}\ldots 0).
    \end{align*}
    From the hypothesis about the distinguishablity of function $h$, we can achieve that all the coefficients are equal to 0. 
\end{proof}

\subsubsection{Proof of Theorem \ref{thm:d2_minimax}}
% \subsection{$D_2$ loss minimax}
\label{apppf:d2_minimax}
To establish minimax lower bounds in the homogeneous regime (Theorem~\ref{thm:d1_minimax}), we begin by introducing two loss functionals that capture different sources of parameter perturbation. For any
$G_1=(\beta_1,\tau_1,\eta_1)\in\Xi$ and $G_2=(\beta_2,\tau_2,\eta_2)\in\Xi$, define
\begin{align*}
    d_1(G_1,G_2)
    &:=
   \sum_{i=1}^{K}\Vert \Delta \eta_{1i} \Vert ^2
    \times| \exp(\tau_1) -\exp(\tau_2) |
    ,
    \\
    d_2(G_1,G_2)
    &:=\exp(\tau_1)
    \times
    \sum_{i=1}^{K}
    \Vert \Delta \eta_{1i}\Vert
    \left(\Vert \beta_1- \beta_2 \Vert+\Vert \eta_{1i}- \eta_{2i} \Vert \right)
    .
\end{align*}

Symmetry of $d_1$ and is recovered only in the restricted case where $\tau_1=\tau_2=\tau$ and symmetry of $d_1$ and is recovered only in the restricted case where $\eta_{1i}=\eta_{2i}$. Nonetheless both $d_1$ and $d_2$ continue to satisfy a weak triangle inequality in general. To accommodate this asymmetry, our lower-bound analysis relies on a variant of Le Cam’s method tailored to nonsymmetric losses, following Lemma~C.1 of \cite{gadat2020parameter}.

For multinomial logistic MoE model satisfies all assumptions in Theorem \ref{thm:d1_minimax}, based on the Taylor expansion, we have the following results:
\begin{lemma}
\label{prop:lower-distinguish}
    Under the assumptions in Theorem \ref{thm:d1_minimax}, 
we denote
\begin{align*}
    S_{1 } = (\tau_1 , \beta, \eta), 
    S_{2 } = (\tau_2 , \beta, \eta),
    ~\text{and}~
    S'_{1} = (\tau, \beta_1, \eta_1), 
    S'_{2} = (\tau, \beta_2, \eta_2),
\end{align*}
    we achieve for any $r < 1$ that
\begin{align*}
        &\text{(i)}~~
            \lim_{\epsilon \rightarrow 0} 
        \inf_{S_1,S_2
        }
        \left\{\displaystyle
        \frac{\bbE_X[d_H\left(p_{S_1}(\cdot|X), p_{S_2}(\cdot|X)\right)] }{d_1^r\left(S_1,S_2\right)}
        : d_1\left(S_1,S_2\right) \leq \epsilon\right
        \}=0,\\
        &\text{(ii)}~~
            \lim_{\epsilon \rightarrow 0} 
        \inf_{S'_1,S'_2
        }
        \left\{\displaystyle
        \frac{\bbE_X[d_H\left(p_{S'_1}(\cdot|X), p_{S'_2}(\cdot|X)\right)] }{d_2^r\left(S'_1, S'_2\right)}
        : d_2\left(S'_1, S'_2\right) \leq \epsilon\right
        \}=0.
        \end{align*}
\end{lemma}
% We will prove this lemma later.
The proof of Lemma~\ref{prop:lower-distinguish} is deferred to a later section.

\begin{proof}[Proof of Theorem \ref{thm:d2_minimax}]
Denote $\Gs=(\betas,\taus,\etas)$ and assume $r<1$. 
By Lemma \ref{prop:lower-distinguish} (i) , for any sufficiently small $\epsilon > 0$, there exists 
$ G'_{*} = (\beta^{*}_1,\tau^{*},  \eta^{*}_1) $ 
such that 
$d_1( G_{*} , G'_{*}) = d_1(G'_{*},  G_{*} ) = \epsilon $. Moreover, there exists a constant $C_0$ such that 
\begin{align}
\label{apppf:lower1-hellinger-distance}
    \bbE_X\left[d_H(p_{ G_{*}}(\cdot|X), p_{ G'_{*}}(\cdot|X))\right] \leq C_0 \epsilon^r.
\end{align} 
Now we denote $p^n_{ G_{*}}$ as the multinomial logistic model of the $n$-i.i.d. sample $(X_1,Y_1),\cdots,(X_n,Y_n)$.
Using Lemma C.1 in \cite{gadat2020parameter} for metric $d_1$, we have 
% \begin{align*}
%     \inf_{\llbgn \in \Xi }\sup_{\lbg\in \Xi }\mathbb{E}_{\plbg} \Big( \exp^2(\tau) \Vert (\overline{\beta}_n, \overline{\eta}_n, \overline{\nu}_n)-(\beta,\eta,\nu) \Vert^2 \Big) 
%     &\geq \frac{\epsilon^2}{2}\Big(1-\bbE_X[V(p^n_{ G_{*}}(\cdot|X), p^n_{ G'_{*}}(\cdot|X))] \Big)
%     \\
%     &
%     \geq \frac{\epsilon^2}{2}
%     \sqrt{1-\left( 1-C_0^2\epsilon^{2r} \right)^n}.
% \end{align*}
% \begin{align*}
%     d_1(G_1,G_2)
%     &:=
%     \exp(\tau_1)
%     \left(
%         \|\beta_1 - \beta_2\|
%         +
%         \sum_{i=1}^{K-1}
%         \|\eta_{1i} - \eta_{2i}\|
%     \right),
%     \\
%     d_2(G_1,G_2)
%     &:=
%     \bigl|
%         \exp(\tau_1) - \exp(\tau_2)
%     \bigr|^2 .
% \end{align*}
\begin{align*}
\label{proof:lower-distinguish-eq1}
    \inf_{\llbgn \in \Xi }\sup_{\lbg\in \Xi }\mathbb{E}_{\plbg} 
    \left[  \sum_{i=1}^{K}\Vert \Delta \eta_{1i} \Vert ^4
    \times| \exp(\tau_1) -\exp(\tau_2) |^2
    \right]
    &\geq \frac{\epsilon^2}{2}\Big(1-\bbE_X[V(p^n_{ G_{*}}(\cdot|X), p^n_{ G'_{*}}(\cdot|X))] \Big)
    \\
    &
    \geq \frac{\epsilon^2}{2}
    \sqrt{1-\left( 1-C_0^2\epsilon^{2r} \right)^n}.
\end{align*}

The final inequality follows from the definitions of the total variation distance and the Hellinger distance, together with \eqref{apppf:lower1-hellinger-distance}.
Let $\epsilon^{2r}={C_0^{-2}}n^{-1}$, then for every $r<1$ we have
\begin{align*}
    \inf_{\llbgn\in \Xi }\sup_{\lbg\in \Xi }\mathbb{E}_{\plbg} 
    % \Big( \exp^2(\tau) \Vert (\overline{\beta}_n, \overline{\eta}_n, \overline{\nu}_n)-(a,b,\nu) \Vert^2 \Big) 
    \left[ 
    \sum_{i=1}^{K}\Vert \Delta \eta_{1i} \Vert ^4
    \times| \exp(\tau_1) -\exp(\tau_2) |^2
    \right]
    \geq c_1 n^{-1/r},
\end{align*}
where $c_1$ denotes a positive constant.
By applying an analogous line of reasoning and invoking Lemma \ref{prop:lower-distinguish} part (ii), we obtain
\begin{align*}
    \inf_{\llbgn\in \Xi }\sup_{\lbg\in \Xi }\mathbb{E}_{\plbg} \left[ \sum_{i=1}^{K}
    \Vert \Delta \eta_{1i}\Vert^2
    \left(\Vert \beta_1- \beta_2 \Vert^2+\Vert \eta_{1i}- \eta_{2i} \Vert^2 \right) \right] \geq c_2 n^{-1/r},
\end{align*}
for some positive constant $c_2$. 
Consequently, we establish all of the results for Theorem \ref{thm:d1_minimax}. 
\end{proof}

\begin{proof}[Proof of Lemma \ref{prop:lower-distinguish} (i)]
% We consider two sequences
% \begin{align*}
% S'_{1,n} &= (\tau_{1,n}, \beta_n, \eta_n, \nu_n), \\
% S'_{2,n} &= (\tau_{2,n}, \beta_n, \eta_n, \nu_n),
% \end{align*}
% with different $\tau_{1,n} \neq \tau_{2,n}$ but the same $(\beta_n, \eta_n, \nu_n)$. 
We consider two parameter sequences  
$S'_{1,n} = (\tau_{1,n}, \beta_n, \eta_n), 
S'_{2,n} = (\tau_{2,n}, \beta_n, \eta_n),$
which differ only in their scale parameters, with $\tau_{1,n} \neq \tau_{2,n}$, while share the same $(\beta_n, \eta_n)$.
Then we have that
\begin{align*}
p_{S'_{1,n}}(y=s  |  x) - p_{S'_{2,n}}(y=s  |  x)
= \frac{e^{\beta_n^\top x} \left( e^{\tau_{2,n}} - e^{\tau_{1,n}} \right)}
{ \left( 1 + e^{\beta_n^\top x + \tau_{1,n}} \right) \left( 1 + e^{\beta_n^\top x+ \tau_{2,n}} \right) }
\cdot \left[\frac{\exp(h(x,\eta_{ns}))}{\sum_{j=1}^K \exp(h(x,\eta_{nj}))}-
\frac{\exp(h(x,\eta_{0s}))}{\sum_{j=1}^K \exp(h(x,\eta_{0j}))}
% f(Y  |  h(X, \eta_n), \nu_n) - f_0(Y  |  h_0(X,\eta_0), \nu_0) 
\right].
\end{align*}

By a standard bound for the squared Hellinger distance, we have
\begin{align*}
\mathbb{E}_X \left[
d_{H}^2 \left(
p_{S'_{1,n}}(\cdot | X),
p_{S'_{2,n}}(\cdot | X)
\right)
\right]
\le
C
\int
\left(
\frac{
p_{S'_{1,n}}(Y | X) - p_{S'_{2,n}}(Y | X)
}{
p_{S'_{2,n}}(Y | X)
}
\right)^2
  d(X,Y).
\end{align*}
Since $(\beta_n,\eta_n)$ range over a compact set and
$p_{S'_{2,n}}(Y | X) \ge c > 0$, the denominator is uniformly lower bounded.
Consequently, there exists a constant $C'>0$ such that
\begin{align*}
\mathbb{E}_X \left[
d_{H}^2 \left(
p_{S'_{1,n}}(\cdot | X),
p_{S'_{2,n}}(\cdot | X)
\right)
\right]
\le
C'\bigl(
\exp(\tau_{1,n}) - \exp(\tau_{2,n})
\bigr)^2.
\end{align*}

Recalling the definition of the distance
$d_2 \left(
(\tau_{1,n},\beta_n,\eta_n),
(\tau_{2,n},\beta_n,\eta_n)
\right)
:= \bigl|\exp(\tau_{1,n}) - \exp(\tau_{2,n})\bigr|^2,$
we obtain
\begin{align*}
\frac{
\mathbb{E}_X \left[
d_{H}^2 \left(
p_{S'_{1,n}}(\cdot | X),
p_{S'_{2,n}}(\cdot | X)
\right)
\right]
}{
d_2^r(S'_{1,n},S'_{2,n})
}
&\le
C'\bigl|\exp(\tau_{1,n}) - \exp(\tau_{2,n})\bigr|^{2(1-r)}
 \longrightarrow  0,
\end{align*}
when $\exp(\tau_{1,n}) - \exp(\tau_{2,n}) \to 0$ and $r<1$.
Therefore,
${
\mathbb{E}_X \left[
d_{H}^2 \left(
p_{S'_{1,n}}(\cdot | X),
p_{S'_{2,n}}(\cdot | X)
\right)
\right]
}/{
d_2^r(S'_{1,n},S'_{2,n})
}
\longrightarrow 0,$
which completes the proof of part~(ii).

\end{proof}

\begin{proof}[Proof of Lemma \ref{prop:lower-distinguish} (ii) ]
Consider two sequences
$    S_{1,n} = (\tau_{n},\beta_{1,n},\eta_{2,n}),\ 
    S_{2,n} =(\tau_{n},\beta_{1,n},\eta_{2,n})$
which share the same value of $\tau_n$. By the definition of the multinomial logistic MoE model, we have
\begin{align*}
    p_{S_{l,n}}(y=s | x) 
    &= \frac{1}{1 + \exp(\beta_{l,n}^\top   + \tau_n)} 
    \frac{\exp(h(x,\eta_{0s}))}{\sum_{i=1}^K \exp(h(x,\eta_{0i}))}
    + \frac{\exp(\beta_{l,n}^\top   + \tau_n)}{1 + \exp(\beta_{l,n}^\top   + \tau_n)} 
    \frac{\exp(h(x,\eta_{l,ns}))}{\sum_{j=1}^K \exp(h(x,\eta_{l,nj}))},
\end{align*}
for $l = 1,2$, and denote $f_l(y=s | x)={\exp(h(x,\eta_{l,ns}^*))}/{\sum_{j=1}^K \exp(h(x,\eta_{l,nj}^*))}$.
Since $(\tau_n,\beta_{l,n})$ lie in a compact set, and softmax function are 
non-negative, the squared Hellinger distance satisfies
\begin{align*}
    \bbE_X&[d_{H}^2(p_{S_{1,n}}(\cdot|X), p_{S_{2,n}}(\cdot|X))]
    \leq C \int \left( \frac{p_{S_{1,n}}(Y  |  X) - p_{S_{2,n}}(Y  |  X)}{p_{S_{2,n}}(Y  |  X)} \right)^2 d(X, Y) \\
    &\leq C' \int \left[
    \frac{
        \exp(\beta_{1,n}^\top X) f_1(y=s | x)
        - \exp(\beta_{2,n}^\top X) f_2(y=s | x)
    }{
        \exp(\beta_{2,n}^\top X) f_2(y=s | x)
    }
    \right]^2 d(X, Y),
\end{align*}
for some constants $C, C'$ depending on the compactness bounds.

% Consider the Taylor expansion of the map $(\beta, \eta) \mapsto \exp(\beta^\top X) f(y=s | x)$
% at the point $(\beta_{2,n}, \eta_{2,n})$, expanded up to first order with integral remainder, that is refer to the equation \eqref{eq:u_taylor_first_order}. 
% It follows that 
% \begin{align*}
%     \frac{\bbE_X[h^2(p_{S_{1,n}}(\cdot|X), p_{S_{2,n}}(\cdot|X))]}{d_{1}^{2r}(S_{1,n}, S_{2,n})} \to 0.
% \end{align*}
% since $\tau_n$ lies in a compact set.

We consider the first-order Taylor expansion with integral remainder for 
$u(y=s | x;\beta,\eta)
:= \exp(\beta^\top x)  f(y=s | x;\eta)$
defined in the Appendix \ref{proof:not_equal}
% of the mapping
% $(\beta,\eta) \mapsto \exp(\beta^\top X)  f(y=s | x)$
around the point $(\beta_{2,n}, \eta_{2,n})$; see equation~\eqref{eq:u_taylor_first_order}. It then follows that
$    {
    \mathbb{E}_X \left[
    d_{H}^2 \left(
    p_{S_{1,n}}(\cdot | X),
    p_{S_{2,n}}(\cdot | X)
    \right)
    \right]
    }/{
    d_1^{2r}(S_{1,n}, S_{2,n})
    }
    \longrightarrow 0,$
where we have used the fact that $\tau_n$ ranges over a compact set.
This completes the proof of part~(i) of the lemma.

\end{proof}

%{\color{red} Only proofs of section 3}

\subsection{Heterogeneous-expert regime}
In this section, we provide the proofs for Theorem \ref{thm:not_equal} and Theorem \ref{thm:d1_minimax} in homogeneous-expert regime. 
\subsubsection{Proof of Theorem \ref{thm:not_equal}}
\label{proof:not_equal}

\begin{proof}
%     Let $\overline{G}=(\Bar{\beta},\Bar{\tau},\Bar{\eta})$ , following the analogous schema in \cite{Ho-Nguyen-EJS-16}, we only need to demonstrate that
%  \begin{align*}
%      \lim_{\varepsilon\to 0}\inf_{G,G_*}\left\{\frac{\bbE_X[d_V(p_G(\cdot|X),p_{ G_*}(\cdot|X))]}{D_1(G,G_*)}: 
%      D_1(G,\overline{G})\vee D_1(G_*,\overline{G})\leq\varepsilon\right\}>0.
%  \end{align*}
%  Assume by contrary that the above claim is not true. 
% Then, there exist two sequences $G_n=(\beta_n,\tau_n,\eta_{n})$ and $G_{*,n}=(\beta_n^*,\tau^*_n,\eta_{n}^*)$,
% such that when $n$ tends to infinity, we get
%  \begin{align*}
%      \begin{cases}
%         D_1(G_n,\overline{G}) \to 0 ,\\
%         D_1(G_{*,n},\overline{G}) \to 0 ,\\
%           \mathbb{E}_X \bigl[V\bigl(p_{G_n}( \cdot |  X), p_{G_{\ast,n}}( \cdot |  X)\bigr)\bigr]/D_1(G_n,G_{\ast,n}) \to 0.
%      \end{cases}
%  \end{align*}
%  In this proof, we will take into account only the most challenging setting of $(\betan, \eta_{n})$ and $(\betasn, \eta_{n}^*)$ when they converge to the same limit point $(\beta',\eta')$, where $(\beta',\eta')$ is not necessarily equal to $(\Bar{\beta},\Bar{\eta} )$ .
% Recall that
% $(\eta_{K}^\ast)=(0_q)$,
% we also set
% $(\eta_{nK}^{\ast })=(0_q)$.
Let $\overline{G} = (\bar{\beta}, \bar{\tau}, \bar{\eta})$. Following an argument analogous to that of \cite{Ho-Nguyen-EJS-16}, it suffices to show that
\begin{align*}
    \lim_{\varepsilon \to 0} 
    \inf_{G, G_*}
    \left\{
    \frac{\mathbb{E}_X \bigl[ d_V\bigl(p_G(\cdot \mid X), p_{G_*}(\cdot \mid X)\bigr) \bigr]}
         {D_2(G, G_*)}
    :
    D_2(G, \overline{G}) \vee D_2(G_*, \overline{G}) \le \varepsilon
    \right\}
    > 0 .
\end{align*}
Suppose, by contradiction, that the above claim does not hold. Then there exist two sequences
$G_n = (\beta_n, \tau_n, \eta_n)$ and $G_{*,n} = (\beta_n^*, \tau_n^*, \eta_n^*)$ such that, as $n \to \infty$,
\begin{align*}
    \begin{cases}
        D_2(G_n, \overline{G}) \to 0, \\
        D_2(G_{*,n}, \overline{G}) \to 0, \\
        \mathbb{E}_X \bigl[
        d_V\bigl(p_{G_n}(\cdot \mid X), p_{G_{*,n}}(\cdot \mid X)\bigr)
        \bigr] \big/ D_2(G_n, G_{*,n}) \to 0 .
    \end{cases}
\end{align*}
In this proof, we focus on the most challenging case in which
$(\beta_n, \eta_n)$ and $(\beta_n^*, \eta_n^*)$ converge to the same limit point $(\beta', \eta')$, where $(\beta', \eta')$ is not necessarily equal to $(\bar{\beta}, \bar{\eta})$.
Recall that $\eta_K^* = 0_q$; accordingly, we also set $\eta_{nK}^* = 0_q$.

\textbf{Step 1.}
Define
\[
T_n(s)
:=
\bigl[1+\exp(\beta_n^\top X+\tau_n)\bigr]
\Bigl[
p_{G_n}(Y=s \mid X)
-
p_{G_{*,n}}(Y=s \mid X)
\Bigr].
\]
This quantity can be decomposed as
\begin{align*}
T_n(s)
&=
\exp(\tau_n)
\Bigl[
u(Y=s \mid X;\beta_n,\eta_n)
-
u(Y=s \mid X;\beta_n^*,\eta_n^*)
\Bigr]
=: \mathcal{I}_{n}^{(1)}
\\
&\quad
-
\exp(\tau_n)
\Bigl[
\exp(\beta_n^\top X)
-
\exp(\beta_n^{*\top} X)
\Bigr]
p_{G_{*,n}}(Y=s \mid X)
=: \mathcal{I}_{n}^{(2)}
\\
&\quad
+
\Bigl[
\exp(\tau_n)
-
\exp(\tau_n^*)
\Bigr]
\exp(\beta_n^{*\top} X)
\Bigl[
f(Y=s \mid X;\eta_n^*)
-
p_{G_{*,n}}(Y=s \mid X)
\Bigr].
\end{align*}
Here,
\[
u(Y=s \mid X;\beta,\eta)
:= \exp(\beta^\top X)  f(Y=s \mid X;\eta),
\qquad s \in [K].
\]

We further decompose the terms $\mathcal{I}_{n}^{(1)}$ and $\mathcal{I}_{n}^{(2)}$.
Let $h_i := h(X,\eta_i)$ for all $i \in [K]$. Recall in the proof~\ref{app_proof: d2_loss}, we denote
\begin{align*}
f(Y=s \mid X;\eta)
&=
\frac{\exp(h_s)}
{\sum_{i=1}^{K} \exp(h_i)}
=
\frac{\exp\bigl(h(X,\eta_s)\bigr)}
{\sum_{i=1}^{K} \exp\bigl(h(X,\eta_i)\bigr)}.
\end{align*}
\noindent
% By means of the first‑order Taylor expansion, $\Ione_{n}$ can be written as
% \begin{align*}
% \Ione_{n}
% &=
%     \exp \bigl(\tau_{n}\bigr)
%     \sum_{|\alpha|=1}
%       \frac{1}{\alpha!} 
%       \bigl(\beta_{n}-\beta_{n}^*\bigr)^{\alpha_{1}}
%       \prod_{i=1}^{K-1}
%         \bigl(\eta_{ni} - \eta_{ni}^*\bigr)^{\alpha_{2i}}
%       \frac{
%         \partial^{|\alpha|}
%           u \bigl(Y=s| X; \beta_{n}^{\ast},\eta_{n}^{\ast}\bigr)
%       }{
%         \partial(\beta_{n})^{\alpha_{1}} 
%         \displaystyle\prod_{i=1}^{K-1}
%           \partial(\eta_{ni})^{\alpha_{2i}} 
%       }
%        + 
%       R_{1}(X,Y).
% \end{align*}
% Here we denote that 
% $\alpha := (\alpha_1, \alpha_{21}, \dots, \alpha_{2(K-1)})$,  
% where $\alpha_1 \in \mathbb{N}^d$ and $\alpha_{2i} \in \mathbb{N}^q$ for any $i \in [K - 1]$.
% Additionally, $R_1(X, Y)$ is a Taylor remainder such that $R_1(X, Y)/D_1(G_n, G_{*,n}) \to 0$ as $n \to \infty$.
By a first-order Taylor expansion, the term $\mathcal{I}_{n}^{(1)}$ can be expressed as
\begin{align*}
\mathcal{I}_{n}^{(1)}
&=
    \exp \bigl(\tau_{n}\bigr)
    \sum_{|\alpha|=1}
      \frac{1}{\alpha!} 
      \bigl(\beta_{n}-\beta_{n}^*\bigr)^{\alpha_{1}}
      \prod_{i=1}^{K}
        \bigl(\eta_{ni} - \eta_{ni}^*\bigr)^{\alpha_{2i}}
      \frac{
        \partial^{|\alpha|}
          u \bigl(Y=s|X; \beta_{n}^{\ast},\eta_{n}^{\ast}\bigr)
      }{
        \partial(\beta_{n})^{\alpha_{1}} 
        \displaystyle\prod_{i=1}^{K}
          \partial(\eta_{ni})^{\alpha_{2i}} 
      }
       + 
      R_{1}(X,Y).
\end{align*}
Here, $\alpha := (\alpha_1, \alpha_{21}, \ldots, \alpha_{2K})$ is a multi-index, where
$\alpha_1 \in \mathbb{N}^d$ and $\alpha_{2i} \in \mathbb{N}^q$ for each $i \in [K]$. 
The remainder term $R_1(X,Y)$ satisfies  $R_1(X, Y)/D_2(G_n, G_{*,n}) \to 0$ as $n \to \infty$.

From the formulation of $u$, we have
\begin{align}
\label{eq:u_taylor_first_order}
&u(Y=s \mid X;\beta_n,\eta_n) - u(Y=s \mid X;\beta_n^\ast,\eta_n^\ast)
=
\exp \bigl((\beta_n^\ast)^\top X\bigr)\times
\\&
\Bigg[
f(Y=s \mid X;\eta_n^\ast) X^\top(\beta_n-\beta_n^\ast)
 + 
\sum_{i=1}^{K}\sum_{j=1}^{q}
\bigl(\eta_{ni,j}-\eta_{ni,j}^\ast\bigr) 
\frac{\partial h}{\partial \eta_{i,j}}(X,\eta_{ni}^\ast) 
\frac{\partial f}{\partial h_i}(Y=s \mid X;\eta_n^\ast)
\Bigg]
 + R_u(X,Y).
\end{align}
Then we will have that
\begin{align*}
\mathcal{I}_{n}^{(1)} &= 
\exp(\tau_{n}) 
\sum_{|\alpha|=1} \frac{1}{\alpha!} 
(\beta_{n}-\beta_{n}^\ast)^{\alpha_1} 
\prod_{i=1}^{K} 
(\eta_{ni}-\eta_{ni}^\ast)^{\alpha_{2i}}
\\
&\quad 
\times
\exp((\beta_{n}^*)^\top X)
\times 
X^{\alpha_1} 
\times
\frac{
    \partial^{\sum_{i=1}^{K} |\alpha_{2i}|} f
}{
    \partial h_1^{|\alpha_{21}|}
    \cdots
    \partial h_{K}^{|\alpha_{2K}|}
}
(Y = s  |  X; \eta_n^*) 
      \prod_{i=1}^{K}\dfrac{\partial h(X,\eta^*_{ni})}{\partial (\eta_{ni})^{\alpha_{2i}}}
+ R_1(X, Y).
\end{align*}

Let $\varrho_1 = \alpha_1 \in \mathbb{N}^d$, 
$\varrho_2 = (\varrho_{2i})_{i\in[K]} := (\alpha_{2ij})_{i\in[K],j\in[q]} \in \mathbb{N}^{K\times q}$ and
\begin{align*}
\mathcal{I}_{\varrho_1, \varrho_2} := \left\{ \alpha = (\alpha_1, \alpha_{21}, \dots, \alpha_{2K}) : 
\alpha_1 = \varrho_1,\ (\alpha_{2i})_{1\leq i \leq K} = \varrho_{2} \right\},
\end{align*}
we can rewrite $\mathcal{I}_{n}^{(1)}$ as
\begin{align*}
\mathcal{I}_{n}^{(1)} 
&= \exp(\tau_n)
   \sum_{|\varrho_1|+|\varrho_2|=1} 
   \sum_{\alpha \in \mathcal{I}_{\varrho_1,\varrho_2}} 
   \frac{1}{\alpha!} 
   (\beta_n-\beta_n^*)^{\alpha_1} 
   \prod_{i=1}^{K} 
   (\eta_{ni} - \eta_{ni}^*)^{\alpha_{2i}} 
\\
&\quad \times 
   X^{\varrho_1} \exp((\beta_{n}^*)^\top X) 
   \frac{\partial^{|\varrho_2|} f}{
      \partial h_1^{|\varrho_{21}|} \cdots 
      \partial h_{K}^{|\varrho_{2K}|}
   }(Y = s  |  X; \eta_n^*)
    \prod_{i=1}^{K}\dfrac{\partial h(X,\eta^*_{ni})}{\partial (\eta_{ni})^{\rho_{2i}}}
   + R_1(X,Y).
\end{align*}
Similarly, the term $\mathcal{I}_{n}^{(2)}$ admits the expansion
\begin{align*}
   \mathcal{I}_{n}^{(2)}=&  
   \exp(\tau_n)
  \sum_{|\gamma|=1} 
  \frac{1}{\gamma!} 
  (\beta_n-\beta_n^*)^\gamma 
  \times X^\gamma \exp \big( (\beta_{n}^*)^\top X \big) 
  p_{G_{*,n}}(Y = s | X) 
  + R_2(X,Y).
\end{align*}
Here 
${R_2(Y|X)}/{D_1(\Gn,\Gsn)}\to0$ 
as $n\to\infty$,
where $R_2(X,Y)$ is Taylor remainder.
Putting the above decompositions together, we obtain that 
\begin{align*}
T_n(s) 
&= 
   \sum_{\substack{|\varrho_1| + |\varrho_2| = 0}}^{1} 
   U_{\varrho_1,\varrho_2}^n \times 
   X^{\varrho_1} \exp \big( (\beta_{n}^*)^\top X \big) 
   \frac{\partial^{|\varrho_2|} f}{
      \partial h_1^{|\varrho_{21}|} \cdots 
      \partial h_{K}^{|\varrho_{2K}|}
   }(Y = s  |  X; \eta_n^*)
    \prod_{i=1}^{K}\dfrac{\partial h(X,\eta^*_{ni})}{\partial (\eta_{ni})^{\rho_{2i}}}
   \\
&\quad +  
   \sum_{|\gamma|=0}^{1} 
   W_\gamma^n  \times 
   X^\gamma \exp \big( (\beta_{n}^*)^\top X \big) 
  p_{G_{*,n}}(Y = s  |  X) 
   + R(X,Y).
\end{align*}
where $R(X,Y)$ is the sum of Taylor remainders such that $R(X,Y)/D_1(X,Y) \to 0$ as $n\to \infty$ and 
\begin{align*}
U_{\varrho_1,\varrho_2}^n
&= 
\begin{cases}
\displaystyle 
\exp(\tau_n)  \sum_{\alpha \in \mathcal{I}_{\varrho_1,\varrho_2}} 
\frac{1}{\alpha!} 
(\beta_n - \beta_n^*)^{\alpha_1} 
\prod_{i=1}^{K} 
(\eta_{ni}-\eta_{ni}^*)^{\alpha_{2i}} , 
& (\varrho_1, \varrho_2) \neq (0_d, 0_{K}), \\[1.5ex]
\displaystyle 
\exp(\tau_n)
-
\exp(\tau_n^*), 
& (\varrho_1, \varrho_2) = (0_d, 0_{K}),
\end{cases}
\\[2ex]
W_\gamma^n(j) 
&= 
\begin{cases}
\displaystyle 
- \exp(\tau_n) \frac{1}{\gamma!} 
(\beta_n - \beta_n^*)^\gamma, 
& |\gamma| \neq 0_d, \\[1.5ex]
\displaystyle 
- \exp(\tau_n)
+
\exp(\tau_n^*), 
& |\gamma| = 0_d.
\end{cases}
\end{align*}

\textbf{Step 2.} 
In this step, we will use a contradiction argument to demonstrate that not all the coefficients in the set 
\begin{align}
\label{pfeq:coefficient_set_one}
\mathcal{S}_1=
\left\{ 
\frac{U_{\varrho_1,\varrho_2}^n}{D_{2n}},
\frac{W^n_{\gamma}}{D_{2n}}
:
0 \leq |\varrho_1| + |\varrho_2| \leq 1,
0 \leq |\gamma| \leq 1
\right\}
\end{align}
vanish as $ n \to \infty $ where
$D_{2n}:=\dtwo$. 
Specifically, suppose that all these coefficients converge to zero, when $n\to\infty$, then by taking the summation 
of $U_{\varrho_1,\varrho_2}^n /D_{2n}$, $\varrho_1 \in \{e_1, e_2, \ldots, e_d\}$ and $\varrho_2 = 0_{K}$, 
where $e_i := (0,\ldots,0, \underbrace{1}_{i\text{-th}}, 0,\ldots,0) \in \mathbb{R}^d$, 
we achieve that
\begin{align}
\frac{1}{D_{2n}} \cdot 
\exp(\tau_n) \cdot \| \beta_n-\beta_n^* \|_1  \to  0. 
\end{align}
Similarly, for $\varrho_1 = 0_d$ and $\varrho_2 \in \{e'_1, e'_2, \ldots, e'_{K}\}$, 
where $e'_{ij} := (0,\ldots,0, \underbrace{e''_j}_{i\text{-th}}, 0,\ldots,0) \in \mathbb{R}^{q\times K}$, 
$e''_j := (0,\ldots,0, \underbrace{1}_{j\text{-th}}, 0,\ldots,0) \in \mathbb{R}^q$,
we have
\begin{align}
\frac{1}{D_{2n}} 
\exp(\tau_n)\cdot 
\sum_{i=1}^{K} 
\|\eta_{ni}- \eta_{ni}^* \|_1  \to  0.  
\end{align}
% Combine the limits in equations we just analyses, we get
% \begin{align*}
% \frac{1}{D_{1n}} \cdot 
% \exp(\tau_n) \cdot 
% \Big[ \|\beta_n-\beta_n^*\|_1 
% + \sum_{i=1}^{K-1} \big( |\Delta a_{ni} | + \|\Delta b_{ni} \|_1 \big) \Big]  \to  0.
% \end{align*}
Combine the limits and recall the topological equivalence between 1-norm and 2-norm, the above limit is equivalent to
\begin{align}
\label{eq:pf_d1_sum_1}
\frac{1}{D_{2n}} \cdot  
\exp(\tau_n) \cdot 
\left[ \|\beta_n-\beta_n^*\| 
+ \sum_{i=1}^{K} (\|\eta_{ni}-\eta_{ni}^* \|) \right]  \to  0.  
\end{align}
Given that our parameter lies in a compact set, there exists a positive constant $C$ such that $|\exp(\tau_n^*)/\exp(\tau_n)|\leq C$, thus we will have that
\begin{align}
\label{eq:pf_d1_sum_2}
\frac{1}{D_{2n}} \cdot  
\exp(\tau_n^*) \cdot 
\left[ \|\beta_n-\beta_n^*\| 
+ \sum_{i=1}^{K} (\|\eta_{ni}-\eta_{ni}^* \|) \right]  \to  0.  
\end{align}
Note that
\begin{align}
\label{eq:pf_d1_sum_3}
 \frac{\big| U_{0_d, 0_{K}}^n \big|}{D_{2n}} 
= \frac{1}{D_{2n}} \cdot 
\left| \exp(\tau_n) - \exp(\tau_n^*) \right| 
 \to  0. 
\end{align}

By taking the sum of limits in equations \eqref{eq:pf_d1_sum_1} to \eqref{eq:pf_d1_sum_3}, 
we deduce that 
\[
1 = {D_{2n}}/D_{2n}  \to  0 \quad \text{as } n \to \infty,
\]
which is a contradiction. Thus, at least one among the limits of 
$U_{\varrho_1,\varrho_2}^n/D_{2n}$ and $W_\gamma^n(j)/D_{2n}$ 
is non-zero.

\textbf{Step 3.} Finally, we will leverage Fatou’s lemma to point out a contradiction to the result in Step~2.

Let us denote by $m_n$ the maximum of the absolute values of 
$U_{\varrho_1,\varrho_2}^n /D_{2n}$ and $W_\gamma^n /D_{2n}$ 
 , 
$0 \le |\varrho_1| + |\varrho_2| \le 1  $ and 
$0 \le |\gamma| \le 1  $. 
Then, it follows from Fatou’s lemma that
\begin{align*}
0 
&= \lim_{n \to \infty} 
\frac{\mathbb{E}_X \big[ 2d_V(p_{G_n}(\cdot  |  X), p_{G_{*,n}}(\cdot  |  X)) \big]}{m_n D_{2n}} \\
&\ge \int \sum_{s=1}^K 
\liminf_{n \to \infty} 
\frac{\big| p_{G_n}(Y=s  |  X) - p_{G_{*,n}}(Y=s  |  X) \big|}{m_n D_{2n}} 
  dX  \ge  0.
\end{align*}

As a result, we get that 
$\lvert p_{G_n}(Y=s  |  X) - p_{G_{*,n}}(Y=s  |  X) \rvert / [m_n D_{2n}]$ 
converges to zero, which implies that 
$T_n(s)/[m_n D_{2n}] \to 0$ as $n \to \infty$ for any $s \in [K]$ and almost surely $X$.  
Let $U_{\varrho_1,\varrho_2}^n /[m_n D_{2n}] \to \tau_{\varrho_1,\varrho_2} $ and 
$W_\gamma^n  \to \eta_\gamma $ as $n$ approaches infinity, then the previous result indicates that
\begin{align}
\label{pfeq:case1_sum}
&  \sum_{|\varrho_1|+|\varrho_2|=0}^{1} 
\tau_{\varrho_1,\varrho_2}  \times X^{\varrho_1} \exp \big( (\beta_{n}^*)^\top X \big) 
% \frac{\partial^{|\varrho_2|} f}{\partial h_1^{\varrho_{21}} \cdots \partial h_{K-1}^{\varrho_{2(K-1)}}}(Y=s  |  X; a_n^*, b_n^*) 
   \frac{\partial^{|\varrho_2|} f}{
      \partial h_1^{|\varrho_{21}|} \cdots 
      \partial h_{K}^{|\varrho_{2K}|}
   }(Y = s  |  X; \eta_n^*)
        \prod_{i=1}^{K}\dfrac{\partial h(X,\eta^*_{ni})}{\partial (\eta_{ni})^{\rho_{2i}}}
\notag \\
&\quad +   \sum_{|\gamma|=0}^{1 } 
\eta_\gamma  \times X^\gamma \exp \big( (\beta_{n}^*)^\top X \big) p_{G_{*,n}}(Y=s  |  X)  =  0, 
\quad \text{for any } s \in [K], \text{ a.s. } X. 
\end{align}

Here, at least one among $\tau_{\varrho_1,\varrho_2}(j)$ and $\eta_\gamma(j)$ is different from zero.  
Assume the set
\begin{align}
\label{pfeq:case1_set}
\mathcal{F} := &\Bigg\{ 
X^{\varrho_1} \exp \big( (\beta_{n}^*)^\top X \big) 
% \frac{\partial^{|\varrho_2|} f}{\partial h_1^{\varrho_{21}} \cdots \partial h_{K-1}^{\varrho_{2(K-1)}}}(Y=s  |  X; a_n^*, b_n^*)
   \frac{\partial^{|\varrho_2|} f}{
      \partial h_1^{|\varrho_{21}|} \cdots 
      \partial h_{K}^{|\varrho_{2K}|}
   }(Y = s  |  X; \eta_n^*)
        \prod_{i=1}^{K}\dfrac{\partial h(X,\eta^*_{ni})}{\partial (\eta_{ni})^{\rho_{2i}}}
:  
 0 \le |\varrho_1|+|\varrho_2| \le 1  \Bigg\} \notag \\
&   \cup \Big\{ 
X^\gamma \exp \big( (\beta_{n}^*)^\top X \big) p_{G_{*,n}}(Y=s  |  X):  
 0 \le |\gamma| \le 1  \Big\}
\end{align}
is linearly independent, we deduce that $\tau_{\varrho_1,\varrho_2}  = \eta_\gamma  = 0$ for any 
$ 0 \le |\varrho_1|+|\varrho_2| \le 1 ,   
0 \le |\gamma| \le 1 ,$ which is a contradiction.

Thus, it suffices to show that $\mathcal{F}$ is a linearly independent set to attain the conclusion, which is proved in Proposition \ref{prop:independence_case_1}. 
\end{proof}

{
% \color{orange}
\begin{proposition}

\label{prop:independence_case_1}
    Suppose that the pretrained function $h_0$, training function $h$ satisfying $h_0 \neq h$, and in addition, the function $h$ satisfying: the equation 
    \begin{equation*}
        \sum_{i=1}^K\sum_{|\varrho|=1} \alpha_{\varrho,i} \dfrac{\partial h(X,\eta^*_{i})}{\partial(\eta_{i})^{\alpha_{\varrho}}} = 0.
    \end{equation*}
    implies that $\alpha_{\varrho,i} = 0$ for each $\varrho$ and $i$.
    Then, the family \eqref{pfeq:case1_set} is linear independent. 

\end{proposition}

\begin{proof}
Suppose that there exists coefficients $\tau_{\varrho_1,\varrho_2}$ and $\eta_{\omega}$ ($0 \leq |\varrho_1| + |\varrho_2| \leq 1$, $0 \leq |\omega| \leq 1$) such that 
     \begin{align*}
\sum_{|\varrho_1| = 0}^1\left[ \eta_{\varrho_1}p_{G_{*}}(Y=s  |  X)   + \sum_{ 0 \leq |\varrho_2| \leq 1-|\varrho_1|}\tau_{\varrho_1,\varrho_2}\frac{\partial^{|\varrho_2|} f}{
      \partial h_1^{|\varrho_{21}|} \cdots 
      \partial h_{K}^{|\varrho_{2K}|}
   }(Y = s  |  X; \eta^*)
        \prod_{i=1}^{K}\dfrac{\partial h(X,\eta^*_{i})}{\partial (\eta_{i})^{\rho_{2i}}}\right] \times X^{\varrho_1}\exp\left((\betas)^\top X\right) = 0.
\end{align*} 
It is obvious that the set $\{X^{\varrho_1}\exp\left((\betas)^\top X\right),\ 0\leq |\varrho_1| \leq 1\}$ is linearly independent, thus, we have for $0 \leq |\varrho_1| \leq 1$:
\begin{equation*}
    \eta_{\varrho_1}p_{G_{*}}(Y=s  |  X)   + \sum_{ 0 \leq |\varrho_2| \leq 1-|\varrho_1|}\tau_{\varrho_1,\varrho_2}\frac{\partial^{|\varrho_2|} f}{
      \partial h_1^{|\varrho_{21}|} \cdots 
      \partial h_{K}^{|\varrho_{2K}|}
   }(Y = s  |  X; \eta^*)
        \prod_{i=1}^{K}\dfrac{\partial h(X,\eta^*_{i})}{\partial (\eta_{i})^{\rho_{2i}}} = 0. 
\end{equation*}
We consider situations based on $|\varrho_1|$:

$\bullet$ When $|\varrho_1| = 1$, we have 
\begin{equation*}
    \eta_{\varrho_1}p_{G_{*}}(Y=s|X) + \tau_{\varrho_1,0}f(Y=s|X;\eta^*) = 0.
\end{equation*}
By substituting the equation above for $s \in [K]$ and summing them up, we achieve $\eta_{\varrho_1} + \tau_{\varrho_1,\varrho_2} = 0$. As a result, we have 
\begin{equation*}
\eta_{\varrho_1}\left(p_{G_{*}}(Y=s  |  X) - f(Y=s|X;\eta^*)\right) = 0. 
\end{equation*}
If $\eta_{\varrho_1} \neq 0$, then $p_{G_{*,n}}(Y=s  |  X)  = f(Y=s|X;\eta^*)$, thus $\dfrac{\exp(h_{0}(X,\eta_{0s}))}{\sum_{i=1}^K\exp(h_{0}(X,\eta_{0i}))} = \dfrac{\exp(h(X,\eta^*_{s}))}{\sum_{i=1}^K\exp(h(X,\eta^*_{i}))} $. This equality cannot happens when $h \neq h_0$, as a result, $\eta_{\varrho_1} = 0$. 

$\bullet$ When $|\varrho_1| = 0$, we have 
\begin{equation*}
    \eta_0p_{G_{*}}(Y=s|X) + \tau_{0,0}f(Y=s|X;\eta^*) + \sum_{|\varrho_2| = 1} \tau_{0,\varrho_2} \frac{\partial^{|\varrho_2|} f}{
      \partial h_1^{|\varrho_{21}|} \cdots 
      \partial h_{K}^{|\varrho_{2K}|}
   }(Y = s  |  X; \eta_n^*)
        \prod_{i=1}^{K}\dfrac{\partial h(X,\eta^*_{i})}{\partial (\eta_{i})^{\rho_{2i}}} = 0. 
\end{equation*}
In other words, 
\begin{align*}
    &\eta_0p_{G_{*}}(Y=s|X) + \tau_{0,0}f(Y=s|X;\eta^*) \\
    +&\sum_{i=1}^{K-1} \left(\sum_{|\varrho_{2i}| = |\varrho_2| = 1}\tau_{0,\varrho_2} \prod_{i=1}^{K}\dfrac{\partial h(X,\eta^*_{i})}{\partial (\eta_{i})^{\rho_{2i}}} -\sum_{|\varrho_{2K}| = |\varrho_2| = 1}\tau_{0,\varrho_2} \prod_{i=1}^{K}\dfrac{\partial h(X,\eta^*_{i})}{\partial (\eta_{i})^{\rho_{2i}}}\right)\frac{\partial f}{
      \partial h_i
   }(Y = s  |  X; \eta^*) = 0. 
\end{align*}
Using the similar argument as in proof of Theorem 3.1, \cite{nguyen2024general}, the following set is linearly independent: 
\begin{equation*}
    \left\{p_{G_{*}}(Y=s|X), \dfrac{\partial^{|\varrho_2|}}{\partial h_1^{|\varrho_{21}|}\cdots\partial h_{K-1}^{|\varrho_{2(K-1)}|}},\ 0\leq |\varrho_2| \leq 1\right\}. 
\end{equation*}
As a result, we achieve that 
\begin{equation}
    \eta_0 = \tau_{0,0} = 0, \quad \sum_{|\varrho_2| = |\varrho_{2u}|=1} \tau_{0,\varrho_2} \prod_{i=1}^{K}\dfrac{\partial h(X,\eta^*_{u})}{\partial(\eta_{u})^{\varrho_{2u}}} = \sum_{|\varrho_2| = |\varrho_{2v}|=1} \tau_{0,\varrho_2} \prod_{i=1}^{K}\dfrac{\partial h(X,\eta^*_{v})}{\partial(\eta_{v})^{\varrho_{2v}}}, \ 1\leq u < v \leq q. 
\end{equation}
From the hypothesis, we achieve that $\tau_{0,\varrho_{2i}} = 0$ for all $i$. This completes our proof of independency. 
\end{proof}
}

\subsubsection{Proof of Theorem \ref{thm:d1_minimax}}
\label{proof:d1_minimax}

In what follows, we present a proof of Theorem~\ref{thm:d2_minimax} for the non-overlapping regime. 
\begin{proof}[Proof of Theorem \ref{thm:d2_minimax}]
The proof proceeds along the same lines as those developed in the proof of Theorem \ref{thm:d1_minimax}. Specifically, let 
$S_1 = (\tau_1, \beta_1,\eta_1)$, $S_2= (\tau_2,\beta_2,\eta_2)$ where $\eta_i=\{\eta_{i1},\eta_{i2},\cdots, \eta_{iN} \}$, $i\in\{1,2\}$.
% \begin{align*}
% \begin{cases}
%     d_{\prime}(S_1,S_2)
%     =  \Vert \Delta \eta_1, \Delta\nu_1 \Vert ^2
%     | \exp(\tau_1) -\exp(\tau_2) |
%     ,
%     \\
%     d_{\prime\prime}(S_1,S_2)
%     =\exp(\tau_1)\Vert \Delta \eta_1, \Delta\nu_1 \Vert
%     \Vert (\beta_1,\eta_1,\nu_1) - (\beta_2,\eta_2,\nu_2) \Vert
%     .
% \end{cases}
% \end{align*}
\begin{align*}
\begin{cases}
    d_{\prime}(G_1,G_2):=
    \exp(\tau_1)
    \left(
        \|\beta_1 - \beta_2\|
        +
        \sum_{i=1}^{K}
        \|\eta_{1i} - \eta_{2i}\|
    \right)
    ,
    \\
    d_{\prime\prime}(S_1,S_2)
    =\bigl|
        \exp(\tau_1) - \exp(\tau_2)
    \bigr|^2
    .
\end{cases}
\end{align*}
% \begin{align*}
%     d_1(G_1,G_2)
%     &:=
%     \exp(\tau_1)
%     \left(
%         \|\beta_1 - \beta_2\|
%         +
%         \sum_{i=1}^{K-1}
%         \|\eta_{1i} - \eta_{2i}\|
%     \right),
%     \\
%     d_2(G_1,G_2)
%     &:=
%     \bigl|
%         \exp(\tau_1) - \exp(\tau_2)
%     \bigr|^2 .
% \end{align*}
It is immediate that the metrics $d_{'}$ and $d_{''}$ satisfy the weak triangle inequality. Following the same schema as in Lemma \ref{prop:lower-distinguish}, we can demonstrate two subsequent results for any $r > 1$: 
\begin{itemize}
    % \item [(i)] Two sequences can be found  
    % \begin{align*}
    % \begin{cases}
    %     S_{1,n}=(\tau_n,\beta_{1,n},\eta_{n},\nu_{n})\in \Xi(l_n),\\
    %     S_{2,n}=(\tau_n,\beta_{2,n},\eta_{n},\nu_{n})\in \Xi(l_n),
    % \end{cases}
    % \end{align*}
    % such that $d_{\prime}(S_{1,n},S_{2,n}) \to 0$ and $h(p_{S_{1,n}}, p_{S_{2,n}})/d_{\prime}^r(S_{1,n},S_{2,n})\to 0$ as $n \to \infty$. 

    \item [(i)] Two sequences   
$        S_{1,n}=(\tau_{1,n},\beta_{n},\eta_{n})\in \Xi(l_n),$ and $
        S_{2,n}=(\tau_{1,n},\beta_{n},\eta_{n})\in \Xi(l_n),$
        can be found
    such that $d_{\prime}(S_{1,n},S_{2,n}) \to 0$ and $\bbE_X[h(p_{S_{1,n}}(\cdot|X), p_{S_{2,n}}(\cdot|X))]/d_{\prime}^r(S_{1,n},S_{2,n})\to 0$ as $n \to \infty$.

    \item [(ii)] Two sequences $S'_{1,n}=(\tau_n,\beta_{1,n},\eta_{1,n})\in \Xi(l_n),$ and $
        S'_{2,n}=(\tau_n,\beta_{2,n},\eta_{2,n})\in \Xi(l_n),$ can be found  
    such that $d_{\prime\prime }(S_{1,n},S_{2,n}) \to 0$ and $\bbE_X[h(p_{S'_{1,n}}(\cdot|X), p_{S'_{2,n}}(\cdot|X))]/d_{\prime\prime }^r(S'_{1,n},S'_{2,n})\to 0$ as $n \to \infty$.
\end{itemize}
We can omit the justification for the above results as it can follow a similar approach as in Lemma \ref{prop:lower-distinguish}. This leads to the conclusion of the theorem.   
\end{proof}

%%%%%%%%%%%%%%%%%%%%%%%%%%%%%%%%%%%%%%%%%%%%%%%%%%%%%%%%%%%%
% \section{PROOFS FOR AUXILIARY RESULTS}
\section{Proof of Auxiliary Results}
\label{appendix:ProofsforAuxiliaryResults}
\subsection{Proof of Density Estimation Rate in Proposition \ref{prop:density-rate}}
\label{appendix:ConvergenceRateofDensityEstimation}

% \subsection{General theory for the Proof of Theorem \ref{theorem:ConvergenceRateofDensityEstimation}}
% \label{appendix:ConvergenceRateofDensityEstimation}

% At first we restate Proposition \ref{theorem:ConvergenceRateofDensityEstimation}:

% \begin{proposition}
% \label{appendixtheorem:ConvergenceRateofDensityEstimation}
% Assume that the function $f_0$ is bounded with tail 
% $\mathbb{E}_X
% \left(
% -\log f_0(Y|h (X,\eta_0),\nu_0)
% \right)
% \gtrsim
% Y^q
% $
% for almost surely $Y\in\mathcal{Y}$
% for some $q>0$.
% and $f$ is the density function of an univariate Gaussian distribution.
%     % Assume the following assumption holds:
%     % A2. Given a universal constant $J > 0$, there exists $N > 0$, possibly depending on  $\Xi$, such that for all $n \geq N$ and all $\epsilon > (\log(n)/n)^{1/2}$, we have $\mathcal{J}_B(\epsilon, \overline{P}^{1/2}(\Xi, \epsilon)) \leq J \sqrt{n} \epsilon^2$.
% Then, there exists a constant $C > 0$ depending only on $\Xi$ such that for all $n \geq 1$,
% \begin{align*}
%     \sup_{\Gs\in\Xi}
%     \mathbb{E}_{p_{ \Gs}}
%     h(p_{ \widehat{G}_n},p_{ \Gs})
%     \leq
%     C\sqrt{\log n/n}.
% \end{align*}
% \end{proposition}

% \subsection*{Notation and Preliminaries}

We first recall some standard notation that will be used throughout the proof. Let $(\mathcal{P}, d)$ be a metric space equipped with a metric on $\mathcal{P}$. For any $\epsilon>0$, an $\epsilon$-net of $(\mathcal{P}, d)$ is a collection of balls of radius $\epsilon$ whose union contains the entire space $\mathcal{P}$. The \emph{covering number} $N(\epsilon, \mathcal{P}, d)$ is defined as the smallest number of such ball needed to cover $\mathcal{P}$, and the associated  \emph{entropy number} is given by $$H(\epsilon, \mathcal{P}, d) := \log N(\epsilon, \mathcal{P}, d).$$
The \emph{bracketing number} $N_B(\epsilon, \mathcal{P}, d)$ is defined as the smallest integer $n$ for which there exists pairs $\{(\underline{f}_i, \overline{f}_i)\}_{i=1}^n$ satisfying $\underline{f}_i < \overline{f}_i$, $d(\underline{f}_i, \overline{f}_i) < \epsilon$, such that every element of $\mathcal{P}$ lies within at least one of these brackets. The associated \emph{bracketing entropy} is given by $$H_B(\epsilon, \mathcal{P}, d) := \log N_B(\epsilon, \mathcal{P}, d).$$

When $\mathcal{P}$ consists of probability density functions, we equip it with the $d$ to be the $L^2(m)$ metric, where $m$ denotes the Lebesgue measure. Particularly, let $\mathcal{P}(\Xi) := \{ p_{\lambda} : \lambda \in \Xi \}$, and define the symmetrized density
\(
\bar{p}_\lambda := \frac{1}{2}(p^* + p_\lambda),
\)
where $p^*$ denotes the true density. We then define the following sets:
$\overline{\mathcal{P}}(\Xi) := \{ \bar{p}_\lambda : \lambda \in \Xi \}$ 
and
$\overline{\mathcal{P}}^{1/2}(\Xi) := \{ \bar{p}_\lambda^{1/2} : \bar{p}_\lambda \in \overline{\mathcal{P}}(\Xi) \}.$
% \begin{align*}
%     \overline{\mathcal{P}}(\Xi) &:= \{ \bar{p}_\lambda : \lambda \in \Xi \}, \\
%     \overline{\mathcal{P}}^{1/2}(\Xi) &:= \{ \bar{p}_\lambda^{1/2} : \bar{p}_\lambda \in \overline{\mathcal{P}}(\Xi) \}.
% \end{align*}
To analyze convergence rates, we focus on a localized form of the symmetrized class:
\(
\overline{\mathcal{P}}^{1/2}(\Xi, \epsilon) := \{ \bar{p}_\lambda^{1/2} \in \overline{\mathcal{P}}^{1/2}(\Xi) : h(\bar{p}_\lambda, p^*) \leq \epsilon \},
\)
where $h(\cdot,\cdot)$ represents the Hellinger metric. 

We then quantify the complexity of this class using the \emph{bracketing entropy integral} introduced in \cite{Vandegeer-2000}:$$
\mathcal{J}_B(\epsilon, \overline{\mathcal{P}}^{1/2}(\Xi, \epsilon), m) := \int_{\epsilon^2 / 2^{13}}^{\epsilon} \sqrt{H_B(u, \overline{\mathcal{P}}^{1/2}(\Xi, \epsilon), m)}   du \vee \epsilon,
$$ where $a \vee b := \max\{a, b\}$. For simplicity, we suppress the dependence on $m$ whenever it is unambiguous from the context.

The proof begins by deriving upper bounds for the covering number and the bracketing entropy.
\begin{lemma}
\label{applemma:convergence-rate-2}
   Suppose that $\Xi$ is a bounded subset of $\mathbb{R}^d\times\mathbb{R} \times \mathbb{R}^{q\times K}$, and $\epsilon \in (0,1/2)$, then  
\begin{enumerate}[(i)]
    \item $\log N(\epsilon,\mathcal{P}(\Xi),\Vert\cdot\Vert_\infty)\lesssim\log (1/\epsilon)$,
    \item $H_B(\epsilon,\mathcal{P}(\Xi),h)\lesssim\log (1/\epsilon)$.
\end{enumerate}
\end{lemma}
\begin{proof}[Proof of Lemma \ref{applemma:convergence-rate-2}]
\textbf{Part (i)}. Firstly, we consider two sets $\Theta_1 = \{\eta \in \mathbb{R}^{d\times K}: (\beta,\tau,\eta) \in \Xi\}$ and $\Theta_2 = \{(\beta,\tau) \in \mathbb{R}^d \times \mathbb{R}: (\beta,\tau,\eta) \in \Xi\}$. From the hypothesis about compactness of $\Xi$, we have $\Theta_1$ and $\Theta_2$ are both compact sets. Consequently, one can construct $\epsilon$-covers $\Theta_{1\epsilon}$ and $\Theta_{2\epsilon}$ of $\Theta_1$ and $\Theta_2$, respectively, such that 
\begin{equation*}
    |\Theta_{1\epsilon}| \lesssim \mathcal{O}((1/\epsilon)^{d\times K}) \text{ and } |\Theta_{2\epsilon}| \lesssim \mathcal{O}((1/\epsilon)^{d+1}).  
\end{equation*}
To construct a $\epsilon$-cover of $\mathcal{P}(\Xi)$, we can follow the construction below. Consider a mixing measure of the form $G = \delta_{(\beta,\tau)} \times \sum_{j=1}^K\delta_{\eta_j}$, we introduce the  measure $\overline{G} = \delta_{(\overline{\beta},\overline{\tau})} \times \sum_{j=1}^K\delta_{{\overline{\eta}}_j}$ where $(\overline{\beta},\overline{\tau}) \in \Theta_{1\epsilon}$ be the closet point to $(\beta,\tau)$ in $\Theta_{1\epsilon}$, and $\eta \in \Theta_{2\epsilon}$ be the closet point to $\eta$ in $\Theta_{2\epsilon}$. Consider the set $\mathcal{A}$ defined as 
\begin{equation*}
    \mathcal{A} := \{g_{G} \in \mathcal{P}(\Xi):(\overline{\beta},\overline{\tau}) \in \Theta_{1\epsilon}, \overline{\eta} \in \Theta_{2\epsilon}\},
\end{equation*}
the it is obvious that $g_{\bar{G}} \in \mathcal{A}$. We now demonstrate that $\mathcal{A}$ forms an $\epsilon$ cover of the metric space $(\mathcal{P}(\Theta), \|\cdot\|_1)$ though it may not be minimal. To this end, we seek an upper bound for the quantity $\|g_{G} - g_{\overline{G}}\|_1$.  

For the softmax function, let $\sigma(X,\beta,\tau) := \exp(\beta^\top X +\tau)/(1+\exp(\beta^\top X +\tau))$. This function is infinitely differentiable and, subsequently, Lipschitz in a compact set. As a result, we have 
\begin{equation*}
    \|\sigma(X,\beta,\tau) - \sigma(X,\overline{\beta} ,\overline{\tau})\|_{\infty} := \sup_{X \in \mathcal{X}} \left|\dfrac{\exp(\beta^\top X +\tau)}{1+\exp(\beta^\top X +\tau} - \dfrac{\exp(\overline{\beta}^\top X +\overline{\tau})}{1+\exp(\overline{\beta}^\top X +\overline{\tau})}\right|\lesssim \epsilon.
\end{equation*}
For a general softmax function, it is similarly possible to achieve the result about the Lipschitz property: 
\begin{equation*}
    \|f(Y=s|X,\eta) - f(Y=s|X,\overline{\eta})\|_\infty \lesssim \epsilon. 
\end{equation*}
From this, remind that the Lipschitz property is preserved through the product of two bounded function, we have
\begin{align*}
    \|p_G(Y=s|X) - p_{\overline{G}}(Y=s|X)\|_{\infty} &\leq \|\sigma(X,\beta,\tau) -  \sigma(X,\overline{\beta} ,\overline{\tau})\|_{\infty}\|f_0(Y=s|X,\eta_0)\|_{\infty}\\
    &+ \|\sigma(X,\beta,\tau)f(Y=s|X,\eta) -  \sigma(X,\overline{\beta} ,\overline{\tau})f(Y=s|X,\overline{\eta})\|_{\infty}\\
    &\lesssim \epsilon. 
\end{align*}
Thus, we can bound the $\ell_1$ distance between $g_G$ and its approximation $g_{\overline{G}}$ by 
\begin{equation*}
    \|g_G-g_{\overline{G}}\|_{1}\leq \int_{\mathcal{X}}\left(\sum_{s=1}^K \|p_G(Y=s|X) - p_{\overline{G}}(Y=s|X)\|_{\infty}\right)dX \lesssim \epsilon,
\end{equation*}
i.e. $\|g_G-g_{\overline{G}}\|_{1}\lesssim \epsilon$, which means that $\mathcal{A}$ is $\epsilon$-cover of the metric space $(\mathcal{P}(\Xi),\|\cdot\|_1)$. As a result, we can use the cardinality of $\mathcal{A}$ to bound the covering number of $(\mathcal{P}(\Xi),\|\cdot\|_1)$:
\begin{equation*}
    N(\epsilon,\mathcal{P}(\Xi),\|\cdot\|_1) \lesssim |\mathcal{A}| = |\Omega_{1\epsilon}| \times |\Omega_{2\epsilon}| \leq \mathcal{O}((1/\epsilon)^{q\times K +d+1}).
\end{equation*}
Finally, we achieve that $\log N(\epsilon,\mathcal{P}(\Xi),\|\cdot\|_1) \lesssim \log(1/\epsilon)$. 

\textbf{Part (ii). } For $\epsilon$ fixed, consider a $\eta$-cover $\{p_1,\ldots,p_N\}$ of $\mathcal{P}(\Xi)$, where $N := N(\eta,\mathcal{P}(\Xi),\|\cdot\|_1)$, where $\eta$ will be chosen later. We now turn to the construction of brackets of the form $[L_i(Y|X), U_i(Y|X)]$ for $1\leq i\leq N$ given below:
\begin{equation*}
    p_{L,i}(Y=s|X) := \max\{p_i(Y=s|X)-\eta,0\}, \quad p_{U,i}(Y=s|X) := \min\{p_i(Y=s|X)+\eta,1\}, \quad 1\leq s\leq K.
\end{equation*}
This construction guarantees that $\mathcal{P}(\Xi) \subset \cup_{i=1}^N[p_{L,i}(Y=s|X),p_{U,i}(Y=s|X)]$ and $|p_{U,i}(Y=s|X) - p_{L,i}(Y=s|X)| \leq 2\eta$. Moreover, we have 
\begin{equation*}
    \|p_{U,i}(\cdot|X) - p_{L,i}(\cdot|X)\|_1 = \sum_{s=1}^K |p_{U,i}(Y=s|X) - p_{L,i}(Y=s|X)|\leq 2K\eta. 
\end{equation*}
Recall that $H_B(2K\eta,\mathcal{P}(\Xi),\|\cdot\|_1)$  denotes the logarithm of the minimal number of brackets of width $2K\eta$ needed to cover $\mathcal{P}(\Xi)$. Consequently, we have 
\begin{equation*}
    H_B(2K\eta,\mathcal{P}(\Xi),\|\cdot\|_1) \leq \log N(\eta,\mathcal{P}(\Xi),\|\cdot\|_1) \overset{\text{part (i)}}{\lesssim} \log(1/\eta).
\end{equation*}
Choose $\eta = \epsilon/(2K)$, we achieve that $H_B(\epsilon,\mathcal{P}(\Xi),\|\cdot\|_1) \lesssim \log(1/\epsilon)$. In addition, $h \leq \|\cdot\|_1$ (Hellinger distance is less than $\ell_1$ distance), we thus arrive at the desired conclusion:
\begin{equation*}
    H_B(\epsilon,\mathcal{P}(\Xi),h) \lesssim \log(1/\epsilon).
\end{equation*}
This completes our proof. 
\end{proof}
Lemma \ref{applemma:convergence-rate-2} is the main ingredient for the estimation of the \textit{bracketing entropy integral}, which lead to a general result about density convergence as in Proposition \ref{prop:density-rate}

\begin{lemma}
\label{applemma:convergence-rate-1}
(i). ({Estimation of bracketing entropy integral}) There exists a universal constant $J > 0$ and a constant $N>0$, which may depend on $\Xi$, such that for all $n \geq N$ and all $\epsilon > (\log(n)/n)^{1/2}$, we have 
\begin{align}
\label{assumption:A2}
   \mathcal{J}_B(\epsilon, \overline{P}^{1/2}(\Xi, \epsilon)) \leq J \sqrt{n} \epsilon^2. 
\end{align}    
(ii). (Model convergence rate) One can find a constant $C > 0$ that depends solely on $\Xi$, for which the following holds for all $n\geq1$,
\begin{align*}
    \sup_{\Gs\in\Xi}
    \mathbb{E}_{p_{ \Gs,n}}
    \bbE_X[d_{H}(p_{ \widehat{G}_n}(\cdot|X),p_{ \Gs}(\cdot|X))]
    \leq
    C(\log n/n)^{1/2}.
\end{align*}

\end{lemma}

\begin{proof}[Proof of Lemma \ref{applemma:convergence-rate-1}] \textbf{Part (i). }  First, we observe that 
\begin{align*}
    H_B(\delta,\overline{\mathcal{P}}^{1/2}(\Xi,\delta),\mu)
    \overset{\text{(a)}}{\leq}
    H_B(\delta,\overline{\mathcal{P}}^{1/2}(\Xi),\mu)\overset{\text{(b)}}{=}
    H_B\left(\frac{\delta}{\sqrt{2}},\overline{\mathcal{P}}(\Xi),d_{H}\right)\overset{\text{(c)}}{\leq} H_B(\delta,{\mathcal{P}}(\Xi),d_{H})
   \overset{\text{(d)}}{\lesssim} \log\left(\frac{1}{\delta} \right),
\end{align*}
where (a) is from the fact that $\overline{\mathcal{P}}^{1/2}(\Xi,\delta) \subset \overline{\mathcal{P}}^{1/2}(\Xi)$, (b) is due to the definition of Hellinger distance, (c) is from the inequality $$ d_{H}^2 \left( \frac{f_1 + f^*}{2}, \frac{f_2 + f^*}{2} \right) \leq \frac{d_{H}^2(f_1, f_2)}{2},$$ and (d) is correct thanks to part (i) of Lemma \ref{applemma:convergence-rate-2}. 
Thus, from the definition of bracketing entropy integral, we have 
\begin{align*}
    \mathcal{J}_B(\epsilon, \overline{\mathcal{P}}^{1/2}(\Xi, \epsilon), m) = \int_{\epsilon^2 / 2^{13}}^{\epsilon} \sqrt{H_B(u, \overline{\mathcal{P}}^{1/2}(\Xi, \epsilon), m)}   du \vee \epsilon \lesssim \epsilon \left(\log\left(\dfrac{2^{13}}{\epsilon^2}\right)\right)^{1/2} < n\epsilon^2,
\end{align*}
for all $d\epsilon > \sqrt{\displaystyle\frac{\log n}{n}}$. This completes our proof.

\textbf{Part (ii). } In this part, we introduce the empirical process $\mu_n(\hlbgn)$ defined as 
\begin{align*}
    \mu_n(\hlbgn):=\sqrt{n}\int_{\plbgs>0}
    \frac{1}{2}\log\left(\frac{\barphlbgn}{\plbgs}\right)(\barphlbgn-\plbgs)d(X,Y).
\end{align*}
By Lemma 4.1 and 4.2 in \cite{Vandegeer-2000}, we have
\begin{align*}
    \frac{1}{16}\bbE_X[d_{H}^2(\phlbgn(\cdot|X),\plbgs(\cdot|X))]\leq \bbE_X[d_{H}^2(\barphlbgn(\cdot|X),\plbgs(\cdot|X))] \leq \frac{1}{\sqrt{n}}\mu_n(\hlbgn),
\end{align*}
Thus, we can assess the behavior of Hellinger's distance $\bbE_X[d_{H}^2(\phlbgn(\cdot|X),\plbgs(\cdot|X))]$ through the empirical process $\mu_n(\hlbgn)$. We have for any $\delta>\delta_n:=\sqrt{\log n/n}$, we have
\begin{align*}
    &\mathbb{P}_{G_{*,n}}(\bbE_X[d^2_{H}(\phlbgn(\cdot|X),\plbgs(\cdot|X))]\geq\delta)
    \\&\leq\mathbb{P}_{G_{*,n}}
    \left(
    \mu_n(\hlbgn)-\sqrt{n}\bbE_X[d_{H}^2(\phlbgn(\cdot|X),\plbgs(\cdot|X))]\geq0,
    \bbE_X[d_{H}^2(\phlbgn(\cdot|X),\plbgs(\cdot|X))]\geq\frac{\delta}{4}
    \right)
    \\&\leq\mathbb{P}_{G_{*,n}}
    \left(
    \sup_{\lbg:\bbE_X[d_{H}(\bar{p}_{\lbg}(\cdot|X),\plbgs(\cdot|X))]\geq\delta/4}
    \left[
    \mu_n(\lbg)-\sqrt{n}\bbE_X[d_{H}^2(\barplbg(\cdot|X),\plbgs(\cdot|X))]
    \right]\geq0
    \right)
    \\&\leq\sum_{s=0}^S\mathbb{P}_{G_{*,n}}
    \left(
    \sup_{\lbg:2^s\delta/4\leq \bbE_X[d_{H}^2(\bar{p}_{\lbg}(\cdot|X),\plbgs(\cdot|X))]\leq 2^{s+1}\delta/4}
    \left|
    \mu_n(\lbg)
    \right|
    \geq\sqrt{n}2^{2s}(\frac{\delta}{4})^2
    \right)
    \\&\leq\sum_{s=0}^S\mathbb{P}_{G_{*,n}}
    \left(
    \sup_{\lbg: \bbE_X[d^2_{H}(\bar{p}_{\lbg}(\cdot|X),\plbgs(\cdot|X))]\leq 2^{s+1}\delta/4}
    \left|
    \mu_n(\lbg)
    \right|
    \geq\sqrt{n}2^{2s}(\frac{\delta}{4})^2
    \right)
\end{align*}
where $S$ is a smallest number such that $2^S\delta/4 > 1$, i.e. $S = \lceil \log_2(4/\delta)\rceil $. To estimate the tail of this empirical process can be estimated, we can utilize Lemma \ref{applemma:convergence-rate-3}. This result along with its proof is formulated in Theorem 5.11, \cite{Vandegeer-2000}. 
\begin{lemma}
\label{applemma:convergence-rate-3}
    Let $R > 0$, $k \geq 1$ and  
 $\mathcal{G}$ is a subset in $\Xi$ where $\Gs\in\mathcal{G}\subset\Xi$ .
 Given $C_1<\infty$, for all $C$ sufficiently large, and for $n\in\mathbb{N}$ and $t>0$ is in the following range
 \begin{align}
     t\leq(8\sqrt{n}R)\wedge(C_1\sqrt{n}R^2/K),
 \end{align}
\begin{align}
     t\geq C^2(C_1+1)\Bigg( R\vee\int^{R}_{t/(2^6\sqrt{n})}H_B^{1/2}\big(\frac{u}{\sqrt{2}},\linephalf(\Xi,R),\mu \big) du\Bigg),
\end{align}
then we will have
\begin{align}
    \mathbb{P}_{G_{*,n}}
    \Big(
    \sup_{G\in\mathcal{G},\bbE_X[d_{H}(\barplbg(\cdot|X),\plbgs(\cdot|X))]\leq R}
    |\mu_n(G)|\geq t
    \Big)
    \leq
    C\exp
    \left(
    -\frac{t^2}{C^2(C_1+1)R^2}
    \right).
\end{align}
\end{lemma}

% \textcolor{blue}{,as $h(\barplbg,\plbgs)\leq1$} 

Using Lemma \ref{applemma:convergence-rate-3} with $R=2^{s+1}\delta, C_1=15$ and $t=\sqrt{n}2^{2s}(\delta/4)^2$, we can therefore verify that each condition in Lemma 3. Concretely, for condition (i) in Lemma 3, it is met since $2^{s-1} \delta / 4 \leq 1$ for all $s \leq S$. For the condition (ii), it is still satisfied since
\begin{align*}
    \int^R_{t/2^6\sqrt{n}}
    H_B^{1/2}\left(\frac{u}{\sqrt{2}},\mathcal{P}^{1/2}(\Xi,R),\mu  \right)
    du\vee2^{s+1}\delta &=\sqrt{2}\int^{R/\sqrt{2}}_{R^2/2^{13}}
    H_B^{1/2}\left({u},\mathcal{P}^{1/2}(\Xi,R),\mu  \right)
     du\vee2^{s+1}\delta\\
     &\leq2\mathcal{J}_B\left(R,\mathcal{P}^{1/2}(\Xi,R),\mu  \right)
     \\&\leq2J\sqrt{n}2^{2s+1}\delta^2
     \\&=2^6Jt.
\end{align*}
Given that two conditions in Lemma \ref{applemma:convergence-rate-3}, we have
\begin{align}
    \mathbb{P}_{G_{*,n}}
    \left( \bbE_X[d_{H}(\phlbgn(\cdot|X),\plbgs(\cdot|X))]>\delta \right)
    \leq C\sum_{s=0}^{\infty}
    \exp\left(-\frac{2^{2s}n\delta^2}{2^{14}C^2} \right)
    \leq 
    c\exp\left( -\frac{n\delta^2}{c} \right),
\end{align}
here $c$ is a large constant not depending on $\Gs$.
The bound on supremum of expectation can be derived as: 
\begin{align*}
    \mathbb{E}_{p_{G_{*,n}}}\bbE_X[d_{H}({\phlbgn(\cdot|X),\plbgs(\cdot|X)})]
    &=\int^{\infty}_{0}\mathbb{P}
    \left( \bbE_X[d_{H}({\phlbgn(\cdot|X),\plbgs(\cdot|X)})]>\delta \right)d\delta
    \\&\leq\delta_n+c\int^{\infty}_{\delta_n}\exp \left(-\frac{n\delta^2}{c^2} \right)d\delta
    \\& \leq \Tilde{c}\delta_n,
\end{align*}
here $\Tilde{c}$ is independent from $\lbgs$ and $\delta_n:=(\log n/n)^{1/2}$.
So we can conclude that 
\begin{align*}
    \sup_{\Gs\in\Xi}
    \mathbb{E}_{p_{ \Gs,n}}
    \bbE_X[d_{H}(p_{ \widehat{G}_n}(\cdot|X),p_{ \Gs}(\cdot|X))]
    \leq
    C(\log n/n)^{1/2}.
\end{align*}
\end{proof}

\subsection{Proof of Proposition \ref{appendix:identifiability_proof}}
\label{appendix:identifiability_proof}
In this appendix, we derive a proof of Proposition \ref{prop:identifiability} about the identifiability of softmax-gated contaminated mixture of multinomial logistic experts model. 
\begin{proof}

Consider the following equation 
\begin{align}
\label{eqn:identifiability_equation}
\nonumber
    &\dfrac{1}{1+\exp(\beta^\top x+\tau)}\cdot \dfrac{\exp(h_0(x,\eta_{0s}))}{\sum_{i=1}^K\exp(h_0(x,\eta_{0s}))} + \dfrac{\exp(\beta^\top x+\tau)}{1+\exp(\beta^\top x+\tau)}\cdot \dfrac{\exp(h(x,\eta_{s}))}{\sum_{i=1}^K\exp(h(x,\eta_{s}))} \\
    =&\dfrac{1}{1+\exp(\beta'^\top x+\tau')}\cdot \dfrac{\exp(h_0(x,\eta_{0s}))}{\sum_{i=1}^K\exp(h_0(x,\eta_{0s}))} + \dfrac{\exp(\beta^\top x+\tau)}{1+\exp(\beta'^\top x+\tau')}\cdot \dfrac{\exp(h(x,\eta'_{s}))}{\sum_{i=1}^K\exp(h(x,\eta'_{s}))}. 
\end{align}

Using result from \cite{GruenLeisch2008Identifiability}, we achieve that two models share the same gating set of the mixing measure: 
\begin{equation*}
    \left\{\dfrac{1}{1+\exp(\beta^\top x+\tau)},\dfrac{\exp(\beta^\top x+\tau)}{1+\exp(\beta^\top x+\tau)}\right\} = \left\{\dfrac{1}{1+\exp(\beta'^\top x+\tau)},\dfrac{\exp(\beta'^\top x+\tau)}{1+\exp(\beta'^\top x+\tau)}\right\}.
\end{equation*}
As a result, we have $\beta^\top x + \tau = \beta'^{\top} + \tau'$, which implies $\beta = \beta'$ and $\tau = \tau'$. Recall that $u_0(Y=s|X;\eta_0) = \dfrac{\exp(h_0(x,\eta_{0s}))}{\sum_{i=1}^K\exp(h_0(x,\eta_{0s}))}$ and $u(Y=s|X;\beta,\eta) = \exp(\beta^\top x) \dfrac{\exp(h(x,\eta_{s}))}{\sum_{i=1}^K\exp(h_0(x,\eta_{s}))}$. Equation \eqref{eqn:identifiability_equation} above becomes
\begin{equation*}
    u_0(Y = s  |  X;\eta_0)  + \exp(\beta^\top x +\tau)u(Y=s|X;\beta,\eta) = u_0(Y = s  |  X; \eta_0)  + \exp(\beta^\top x +\tau)u(Y=s|X;\beta,\eta') 
\end{equation*}
Using the same argument as in the proof of Proposition 2.1, \cite{nguyen2023demystifying}, it is straightforward to obtain $\eta  = \eta'$. In other word, $G = G'$. 
\end{proof}

\medskip
% \section{Additional Experimental Results}
% \label{app:exp}
% In addition to the experiments presented in the main text, conduct additional experiments of heterogeneous setting with the ground truth parameters independent of the sample size. More specifically, the true parameters are $\beta^*=\frac{1}{\sqrt{d}}\mathbf{1}_d$ and $\tau^*=0.2$, with $\eta_0^{*}$ and $\eta^*$ fixed matrices of size $d \times K$. We define $\eta_0^{*}$ and $\eta^*$ as follows:
% \[
% \eta^* = e_2 v^\top,
% \qquad
% \eta_0^* = e_2 u^\top,
% \]
% where $e_2\in\mathbb{R}^d$ denotes the second canonical basis vector.
% The corresponding vectors are given by
% \[
% v = (-1,  0.7,  0),
% \qquad
% u = (1,  -0.5,  0).
% \]
\section{Additional Experimental Results}
\label{app:exp}

In this appendix, we present additional experimental results that complement the findings reported in Section~\ref{sec:experiments}. 
While the main text focuses on settings in which the ground truth parameters vary with the sample size, 
here we consider a heterogeneous-expert regime with fixed ground truth parameters in order to isolate the effect of sample size alone on estimation accuracy.

\paragraph{Experimental setup.}
We consider the same two-component gated mixture-of-experts model as in the main experiments.
In contrast to Section~\ref{sec:experiments}, the ground truth parameters are fixed and do not depend on the sample size $n$.

Specifically, we set:
$$\beta^* = \frac{1}{\sqrt{d}}\mathbf{1}_d, \qquad \tau^* = 0.2, $$
and fix the expert parameters $\eta_0^*$ and $\eta^*$ as matrices in $\mathbb{R}^{d\times K}$ defined by:
\[
\eta^* = e_2 v^\top,
\qquad
\eta_0^* = e_2 u^\top,
\]
where $e_2\in\mathbb{R}^d$ denotes the second canonical basis vector.
The corresponding class-specific coefficient vectors are
\[
v = (-1,  0.7,  0),
\qquad
u = (1,  -0.5,  0).
\]
All other aspects of the data generation process and estimation procedure follow those described in Section~\ref{sec:experiments}.
\paragraph{Results.}
Figure~\ref{fig:appendix_heterogeneous_fixed} reports the estimation errors as functions of the sample size $n$ in the heterogeneous setting with fixed ground truth parameters. \\
Consistent with the results in the main text, we observe stable and well-behaved convergence patterns for all parameters. In particular, the estimation errors of the mixing-weight parameter $\exp(\tau)$, the gating parameters $\beta$, and the adapter expert parameters $\eta$ all exhibit approximately polynomial decay with respect to $n$. The fitted slopes on the log--log scale are close to the parametric benchmark rate $O(n^{-1/2})$, indicating efficient statistical estimation in this regime.\\
These results further demonstrate that the favorable convergence behavior observed in Section~4 persists even when the ground truth parameters are fixed, suggesting that expert heterogeneity rather than the specific scaling of parameters with $n$ plays the dominant role in governing estimation accuracy.
\begin{figure}[t]
    \centering
    \begin{subfigure}[t]{\textwidth}
        \centering
        \includegraphics[width=\textwidth]{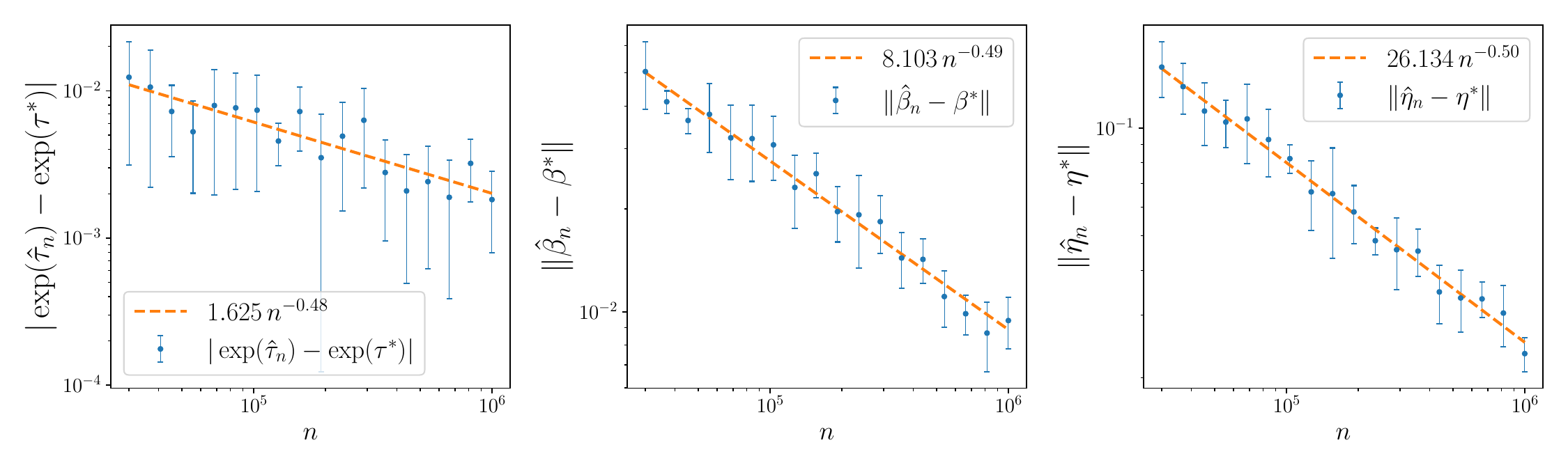}
        \caption{Case (i): $h_0(x)=x^\top\eta_0^*$, $h(x)=\tanh(x^\top\eta^*)$.}
        \label{fig:appendix_case1}
    \end{subfigure}
    \\
    % \hfill
    \begin{subfigure}[t]{\textwidth}
        \centering
        \includegraphics[width=\textwidth]{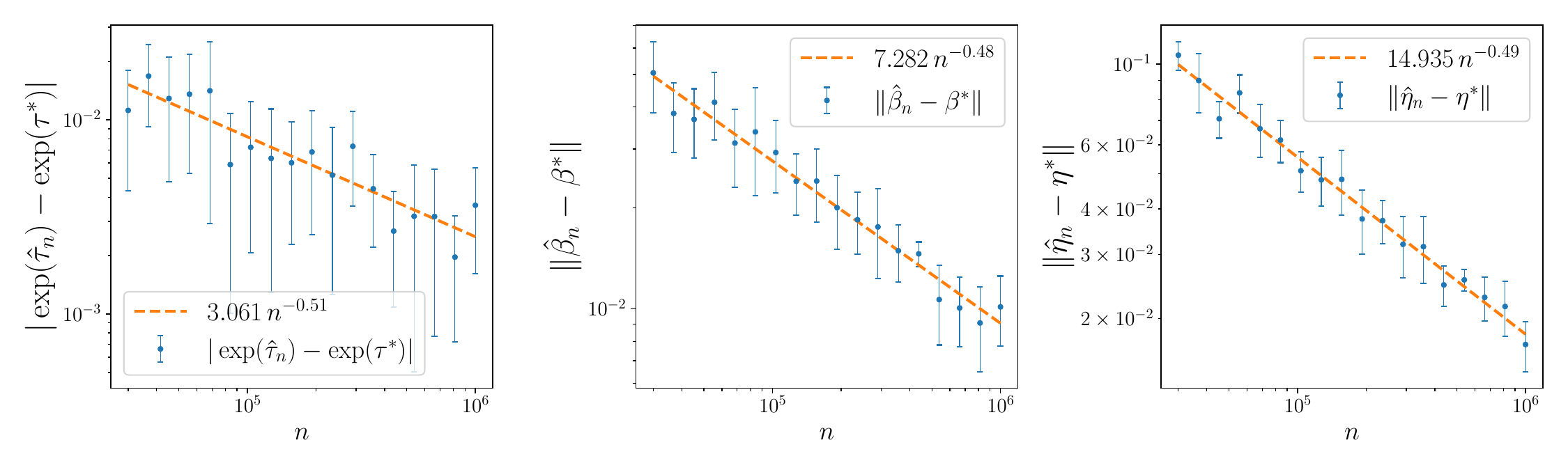}
        \caption{Case (ii): $h_0(x)=\tanh(x^\top\eta_0^*)$, $h(x)=x^\top\eta^*$.}
        \label{fig:appendix_case2}
    \end{subfigure}

    \caption{
    Heterogeneous setting with fixed ground truth parameters.
    Log--log plots of parameter estimation errors as functions of the sample size $n$.
    Blue dots denote mean estimation errors across independent runs with error bars indicating one standard deviation,
    while orange dashed lines represent fitted power-law trends.
    }
    \label{fig:appendix_heterogeneous_fixed}
\end{figure}
\section{Computational Infrastructure}
\label{app:infras}
All numerical experiments were performed on a MacBook Pro equipped with an Apple M4 Max chip.

\bibliography{references}

@inproceedings{hu2022lora,
  title     = {{L}o{R}{A}: {L}ow-{R}ank {A}daptation of {L}arge {L}anguage {M}odels},
  author    = {Edward J. Hu and Yelong Shen and Phillip Wallis and Zeyuan Allen-Zhu and Yuanzhi Li and Shean Wang and Lu Wang and Weizhu Chen},
  booktitle = {Proceedings of the 10th International Conference on Learning Representations (ICLR)},
  year      = {2022}
}

@InProceedings{nguyen2024deviated,
  title = 	 {On Parameter Estimation in Deviated {G}aussian Mixture of Experts},
  author =       {Nguyen, Huy and Nguyen, Khai and Ho, Nhat},
  booktitle = 	 {Proceedings of The 27th International Conference on Artificial Intelligence and Statistics},
  year = 	 {2024},
}

@inproceedings{ han2024fusemoe,
title={FuseMoE: Mixture-of-Experts Transformers for Fleximodal Fusion},
author={Xing Han and Huy Nguyen and Carl William Harris and Nhat Ho and Suchi Saria},
booktitle={The Thirty-eighth Annual Conference on Neural Information Processing Systems},
year={2024},
url={https://openreview.net/forum?id=jfE7XCE89y}
}

@inproceedings{yan2025contaminated,
    title = {Understanding Expert Structures on Minimax Parameter Estimation in Contaminated Mixture of Experts},
    author = {Fanqi Yan and Huy Nguyen and Dung Le and Pedram Akbarian and Nhat Ho},
    booktitle = {Proceedings of The 28th International Conference on Artificial Intelligence and Statistics},
    year = 2025
}

@inproceedings{nguyen2025cosine,
    author = {Huy Nguyen and Pedram Akbarian and Trang Pham and Trang Nguyen and Shujian Zhang and Nhat Ho},
    title = {Statistical Advantages of Perturbing Cosine Router in Mixture of Experts},
    booktitle = {International Conference on Learning Representations},
    year = 2025
}

@book{Vandegeer-2000,
author= "S. van de Geer",
title="Empirical Processes in M-estimation",
publisher= "Cambridge University Press",
year="2000"
}

@article{Ho-Nguyen-EJS-16,
	author="N. Ho and X. Nguyen",
	title="On strong identifiability and convergence rates of parameter estimation in finite mixtures",
	journal="Electronic Journal of Statistics",
	volume="10",
	pages="271-307",
	year="2016"
}

@article{Jacob_Jordan-1991,
	author="R. A. Jacobs and M. I. Jordan and S. J. Nowlan and G. E. Hinton",
	title="Adaptive mixtures of local experts",
	journal="Neural Computation",
	volume="3",
	page="79-87",
	year="1991"
}

@inproceedings{
ceron2024rl,
title={Mixtures of Experts Unlock Parameter Scaling for Deep {RL}},
author={Johan Samir Obando Ceron and Ghada Sokar and Timon Willi and Clare Lyle and Jesse Farebrother and Jakob Nicolaus Foerster and Gintare Karolina Dziugaite and Doina Precup and Pablo Samuel Castro},
booktitle={Forty-first International Conference on Machine Learning},
year={2024}
}

@INPROCEEDINGS{shazeer2017topk,
   AUTHOR = "N. Shazeer and A. Mirhoseini and K. Maziarz and A. Davis and Q. Le and G. Hinton and J. Dean",
   TITLE = "Outrageously Large Neural Networks: The Sparsely-Gated Mixture-of-Experts Layer",
   BOOKTITLE = 	 "In International Conference on Learning Representations", 
   YEAR = 	 2017
}

@article{fedus2021switch,
	author="W. Fedus and B. Zoph and N. Shazeer",
	title="Switch Transformers: Scaling to Trillion Parameter Models with Simple and Efficient Sparsity",
	journal="Journal of Machine Learning Research",
        volume = "23",
        pages = "1-39",
	year="2022"
}

@inproceedings{Du_Glam_MoE,
  title = "GLaM: Efficient Scaling of Language Models with Mixture-of-Experts",
  author = "N. Du and Y. Huang and A. M. Dai and S. Tong and D. Lepikhin and Y. Xu and M. Krikun and Y. Zhou and A. Yu and O. Firat and B. Zoph and L. Fedus and M. Bosma and Z. Zhou and T. Wang and E. Wang and K. Webster and M. Pellat and K. Robinson and K. Meier-Hellstern and T. Duke and L. Dixon and K. Zhang and Q. Le and Y. Wu and Z. Chen and C. Cui",
  booktitle = "ICML",
  year = "2022"
}

@ARTICLE{zeevi1998approximation,
  author={Zeevi, A.J. and Meir, R. and Maiorov, V.},
  journal={IEEE Transactions on Information Theory}, 
  title={Error bounds for functional approximation and estimation using mixtures of experts}, 
  year={1998},
  volume={44},
  number={3},
  pages={1010-1025}}

@inproceedings{
yun2024flexmoe,
title={Flex-MoE: Modeling Arbitrary Modality Combination via the Flexible Mixture-of-Experts},
author={Sukwon Yun and Inyoung Choi and Jie Peng and Yangfan Wu and Jingxuan Bao and Qiyiwen Zhang and Jiayi Xin and Qi Long and Tianlong Chen},
booktitle={The Thirty-eighth Annual Conference on Neural Information Processing Systems},
year={2024}
}

@article{deepseekv3,
  title={Deepseek-v3 technical report},
  author={DeepSeek-AI and others},
  journal={arXiv preprint arXiv:2412.19437},
  year={2024}
}

@inproceedings{lepikhin_gshard_2021,
	title = {{GS}hard: {Scaling} {Giant} {Models} with {Conditional} {Computation} and {Automatic} {Sharding}},
	booktitle = {International {Conference} on {Learning} {Representations}},
	author = {D. Lepikhin and H. Lee and Y. Xu and D. Chen and O. Firat and Y. Huang and M. Krikun and N. Shazeer and Z. Chen},
	year = {2021},
}

@inproceedings{Riquelme2021scalingvision,
 author = {C. Riquelme and J. Puigcerver and B. Mustafa and M. Neumann and R. Jenatton and A. Susano Pint and D. Keysers and N. Houlsby},
 booktitle = {Advances in Neural Information Processing Systems},
 pages = {8583--8595},
 publisher = {Curran Associates, Inc.},
 title = {Scaling Vision with Sparse Mixture of Experts},
 volume = {34},
 year = {2021}
}

@inproceedings{chen2022theory,
 author = {Chen, Zixiang and Deng, Yihe and Wu, Yue and Gu, Quanquan and Li, Yuanzhi},
 booktitle = {Advances in Neural Information Processing Systems},
 editor = {S. Koyejo and S. Mohamed and A. Agarwal and D. Belgrave and K. Cho and A. Oh},
 pages = {23049--23062},
 publisher = {Curran Associates, Inc.},
 title = {Towards Understanding the Mixture-of-Experts Layer in Deep Learning},
 volume = {35},
 year = {2022}
}

@inproceedings{nguyen2023demystifying,
      title={Demystifying Softmax Gating Function in {G}aussian Mixture of Experts}, 
      author={Huy Nguyen and TrungTin Nguyen and Nhat Ho},
      booktitle = "Advances in Neural Information Processing Systems",
      year={2023}
}

@inproceedings{nguyen2024general,
      title={A General Theory for Softmax Gating Multinomial Logistic Mixture of Experts}, 
      author={Huy Nguyen and Pedram Akbarian and TrungTin Nguyen and Nhat Ho},
      booktitle ="Proceedings of the ICML",
      year={2024}
}

@inproceedings{chow_mixture_expert_2023,
	title = {A {Mixture}-of-{Expert} {Approach} to {RL}-based {Dialogue} {Management}},
	url = {https://openreview.net/forum?id=4FBUihxz5nm},
	booktitle = {The {Eleventh} {International} {Conference} on {Learning} {Representations}},
	author = {Chow, Yinlam and Tulepbergenov, Azamat and Nachum, Ofir and Gupta, Dhawal and Ryu, Moonkyung and Ghavamzadeh, Mohammad and Boutilier, Craig},
	year = {2023},
}

@inproceedings{
li2023sparse,
title={Sparse Mixture-of-Experts are Domain Generalizable Learners},
author={Bo Li and Yifei Shen and Jingkang Yang and Yezhen Wang and Jiawei Ren and Tong Che and Jun Zhang and Ziwei Liu},
booktitle={The Eleventh International Conference on Learning Representations },
year={2023}
}

@article{mendes2011convergence,
    author = {Mendes, Eduardo F. and Jiang, Wenxin},
    title = {On Convergence Rates of Mixtures of Polynomial Experts},
    journal = {Neural Computation},
    volume = {24},
    number = {11},
    pages = {3025-3051},
    year = {2012},
    month = {11},
    issn = {0899-7667},
    doi = {10.1162/NECO_a_00354},
    url = {https://doi.org/10.1162/NECO_a_00354},
    eprint = {https://direct.mit.edu/neco/article-pdf/24/11/3025/871220/neco_a_00354.pdf},
}

@article{gadat2020parameter,
author = {S{\'e}bastien Gadat and Jonas Kahn and Cl{\'e}ment Marteau and Cathy Maugis-Rabusseau},
title = {{Parameter recovery in two-component contamination mixtures: The $L^{2}$ strategy}},
volume = {56},
journal = {Annales de l'Institut Henri Poincaré, Probabilités et Statistiques},
number = {2},
publisher = {Institut Henri Poincaré},
pages = {1391 -- 1418},
year = {2020},
doi = {10.1214/19-AIHP1007},
URL = {https://doi.org/10.1214/19-AIHP1007}
}

@article{jordan1994hierarchical,
  title={Hierarchical mixtures of experts and the EM algorithm},
  author={Jordan, Michael I and Jacobs, Robert A},
  journal={Neural computation},
  volume={6},
  number={2},
  pages={181--214},
  year={1994},
  publisher={MIT Press}
}

@article{GruenLeisch2008Identifiability,
  title   = {Identifiability of Finite Mixtures of Multinomial Logit Models with Varying and Fixed Effects},
  author  = {Gr{\"u}n, Bettina and Leisch, Friedrich},
  journal = {Journal of Classification},
  volume  = {25},
  number  = {2},
  pages   = {225--247},
  year    = {2008},
  doi     = {10.1007/s00357-008-9012-1}
}

@inproceedings{
yan2025on,
title={On Minimax Estimation of Parameters in Softmax-Contaminated Mixture of Experts},
author={Fanqi Yan and Huy Nguyen and Le Quang Dung and Pedram Akbarian and Nhat Ho and Alessandro Rinaldo},
booktitle={The Thirty-ninth Annual Conference on Neural Information Processing Systems},
year={2025},
url={https://openreview.net/forum?id=Dx1qQ9OAbb}
}

@article{broyden1967quasi,
  title={Quasi-Newton methods and their application to function minimisation},
  author={Broyden, Charles G},
  journal={Mathematics of Computation},
  volume={21},
  number={99},
  pages={368--381},
  year={1967},
  publisher={JSTOR}
}

@article{fletcher1970new,
  title={A new approach to variable metric algorithms},
  author={Fletcher, Roger},
  journal={The computer journal},
  volume={13},
  number={3},
  pages={317--322},
  year={1970},
  publisher={Oxford University Press}
}

@article{goldfarb1970family,
  title={A family of variable-metric methods derived by variational means},
  author={Goldfarb, Donald},
  journal={Mathematics of computation},
  volume={24},
  number={109},
  pages={23--26},
  year={1970}
}

@article{shanno1970conditioning,
  title={Conditioning of quasi-Newton methods for function minimization},
  author={Shanno, David F},
  journal={Mathematics of computation},
  volume={24},
  number={111},
  pages={647--656},
  year={1970}
}

@article{hendrycks2016gaussian,
  title={Gaussian Error Linear Units (Gelus)},
  author={Hendrycks, D},
  journal={arXiv preprint arXiv:1606.08415},
  year={2016}
}

@inproceedings{devlin2019bert,
  title={Bert: Pre-training of deep bidirectional transformers for language understanding},
  author={Devlin, Jacob and Chang, Ming-Wei and Lee, Kenton and Toutanova, Kristina},
  booktitle={Proceedings of the 2019 conference of the North American chapter of the association for computational linguistics: human language technologies, volume 1 (long and short papers)},
  pages={4171--4186},
  year={2019}
}

@article{brown2020language,
  title={Language models are few-shot learners},
  author={Brown, Tom and Mann, Benjamin and Ryder, Nick and Subbiah, Melanie and Kaplan, Jared D and Dhariwal, Prafulla and Neelakantan, Arvind and Shyam, Pranav and Sastry, Girish and Askell, Amanda and others},
  journal={Advances in neural information processing systems},
  volume={33},
  pages={1877--1901},
  year={2020}
}

@article{bacharoglou2010,
 ISSN = {00029939, 10886826},
 URL = {http://www.jstor.org/stable/20721762},
 author = {Athanassia G. Bacharoglou},
 journal = {Proceedings of the American Mathematical Society},
 number = {7},
 pages = {2619--2628},
 publisher = {American Mathematical Society},
 title = {Approximation OF PROBABILITY DISTRIBUTIONS BY CONVEX MIXTURES OF GAUSSIAN MEASURES},
 volume = {138},
 year = {2010}
}

@inproceedings{
ersoy2025hdee,
title={{HDEE}: Heterogeneous Domain Expert Ensemble},
author={Oguzhan Ersoy and Jari Kolehmainen and Gabriel Passamani Andrade},
booktitle={ICLR 2025 Workshop on Modularity for Collaborative, Decentralized, and Continual Deep Learning},
year={2025},
url={https://openreview.net/forum?id=5ukL6nPcYe}
}

@inproceedings{wang-etal-2025-hmoe,
    title = "{HM}o{E}: Heterogeneous Mixture of Experts for Language Modeling",
    author = "Wang, An  and
      Sun, Xingwu  and
      Xie, Ruobing  and
      Li, Shuaipeng  and
      Zhu, Jiaqi  and
      Yang, Zhen  and
      Zhao, Pinxue  and
      Han, Weidong  and
      Kang, Zhanhui  and
      Wang, Di  and
      Okazaki, Naoaki  and
      Xu, Cheng-zhong",
    editor = "Christodoulopoulos, Christos  and
      Chakraborty, Tanmoy  and
      Rose, Carolyn  and
      Peng, Violet",
    booktitle = "Proceedings of the 2025 Conference on Empirical Methods in Natural Language Processing",
    month = nov,
    year = "2025",
    address = "Suzhou, China",
    publisher = "Association for Computational Linguistics",
    url = "https://aclanthology.org/2025.emnlp-main.1115/",
    doi = "10.18653/v1/2025.emnlp-main.1115",
    pages = "21943--21957",
    ISBN = "979-8-89176-332-6"
}

@ARTICLE{chen2025hetero,
  author={Chen, Bowen and Chen, Keyan and Yang, Mohan and Zou, Zhengxia and Shi, Zhenwei},
  journal={IEEE Geoscience and Remote Sensing Letters}, 
  title={Heterogeneous Mixture of Experts for Remote Sensing Image Super-Resolution}, 
  year={2025},
  volume={22},
  number={},
  pages={1-5},
  doi={10.1109/LGRS.2025.3557928}}

@article{li2024moe,
  title={MoE-CT: a novel approach for large language models training with resistance to catastrophic forgetting},
  author={Li, Tianhao and Li, Shangjie and Xie, Binbin and Xiong, Deyi and Yang, Baosong},
  journal={arXiv preprint arXiv:2407.00875},
  year={2024}
}

@inproceedings{zhao2024sparse,
  title={Sparse moe with language guided routing for multilingual machine translation},
  author={Zhao, Xinyu and Chen, Xuxi and Cheng, Yu and Chen, Tianlong},
  booktitle={The Twelfth International Conference on Learning Representations},
  year={2024}
}

@inproceedings{zhou2025moe,
  title={Moe-lpr: Multilingual extension of large language models through mixture-of-experts with language priors routing},
  author={Zhou, Hao and Wang, Zhijun and Huang, Shujian and Huang, Xin and Han, Xue and Feng, Junlan and Deng, Chao and Luo, Weihua and Chen, Jiajun},
  booktitle={Proceedings of the AAAI Conference on Artificial Intelligence},
  volume={39},
  number={24},
  pages={26092--26100},
  year={2025}
}

@inproceedings{cao2025m,
  title={M-MoE: Mixture of Mixture-of-Expert Model for CTC-based Streaming Multilingual ASR},
  author={Cao, Songjun and Wang, Xiong and Zhang, Yike and Zhang, Xiaoming and Ma, Long},
  booktitle={ICASSP 2025-2025 IEEE International Conference on Acoustics, Speech and Signal Processing (ICASSP)},
  pages={1--5},
  year={2025},
  organization={IEEE}
}
\bibliographystyle{abbrv}
\end{document}